\documentclass{amsart}

\usepackage{type1cm}        

\makeindex		     

\usepackage{graphicx}        
\usepackage{multicol}        
\usepackage[bottom]{footmisc}

\usepackage{newtxtext}       %
\usepackage{newtxmath}       

\usepackage{enumerate}
\usepackage{enumitem}

\usepackage[pagebackref,bookmarksopen]{hyperref}
\hypersetup{
	colorlinks=true,
	linkcolor=blue,
	citecolor=magenta,
	urlcolor=cyan,
}
\usepackage{bookmark} 

\newcommand{\ZZ}{{\mathbb Z}}
\newcommand{\RR}{{\mathbb R}}
\newcommand{\CC}{{\mathbb C}}
\newcommand{\QQ}{{\mathbb Q}}
\newcommand{\KK}{{\mathbb K}}
\newcommand{\PP}{{\mathbb P}}
\newcommand{\NN}{{\mathbb N}}
\newcommand{\LL}{{\mathbb L}}
\newcommand{\cA}{\mathcal{A}}
\newcommand{\cB}{\mathcal{B}}
\newcommand{\cE}{\mathcal{E}}
\newcommand{\cJ}{\mathcal{J}}
\newcommand{\cF}{\mathcal{F}}
\newcommand{\cG}{\mathcal{G}}
\newcommand{\CH}{\mathcal{H}}
\newcommand{\cX}{\mathcal{X}}
\newcommand{\CD}{\mathcal{D}}
\newcommand{\CaL}{\mathcal{L}}
\newcommand{\cM}{\mathcal{M}}
\newcommand{\cP}{\mathcal{P}}
\newcommand{\cQ}{\mathcal{Q}}
\newcommand{\cS}{\mathcal{S}}
\newcommand{\cT}{\mathcal{T}}
\newcommand{\cU}{\mathcal{U}}
\newcommand{\cV}{\mathcal{V}}
\newcommand{\Char}{{\rm Char}}
\newcommand{\ord}{{\rm ord}}

\newtheorem{theorem}{Theorem}[section]

\newtheorem{environment name}[theorem]{Caption}{\bf}{}
\newtheorem{proposition}[theorem]{Proposition}{\bf}{}
\newtheorem{corollary}[theorem]{Corollary}
\newtheorem{problem}[theorem]{Problem}
\theoremstyle{definition}
\newtheorem{definition}[theorem]{Definition}
\newtheorem{example}[theorem]{Example}
\theoremstyle{remark}
\newtheorem{remark}[theorem]{Remark}

\begin{document}

\title{Complements to ample divisors and Singularities.}

\author{Anatoly Libgober}
\address{Department of Mathematics, University of Illinois, Chicago, IL 60607, }
\email{libgober@uic.edu}

\begin{abstract}
The paper reviews recent developments in the 
study of Alexander invariants of quasi-projective manifolds using 
methods of singularity theory. Several 
results in topology of the complements to singular plane curves and
hypersurfaces in projective space extended to the case of curves on simply
connected smooth projective surfaces.
\end{abstract}

\maketitle 


\section{Introduction} 

These notes review interactions between singularity theory
and the study of fundamental groups and more generally the
homotopy type of the complements to divisors on smooth projective
varieties. The main question considered here is how the local topology of singularities as well
as their global geometry affect the topology of the complement.
Several surveys updating the state of the subject at respective
points in time were written 
over the years (cf. \cite{degtbook}) but most often focusing on
specific situations: complements to plane curves, arrangements of
lines or hyperplanes etc. reflecting that earlier studies of
the complements were mainly focused on the case 
of plane curves.
Below we consider the complements to divisors $D$ on smooth
projective surfaces $X$ and their fundamental groups, sometimes
 indicating how a generalization to the case of homotopy types of the
complements in manifolds of dimension greater than two looks like, but mostly referring to other publicaions for
additional details on homotopy invariants beyond fundamental
groups. An earlier appearances of the studies of the complements in the
context of general pairs $(X,D)$ and their fundamental groups can be
traced to the 80s. 
Some results on the topology of the complements in such set up  did appear in
\cite{SGA1}, \cite{fujita} \cite{popp}. A much earlier, beautiful
results, especially those showing the role of the abelian varieties in
the subject were obtained by Italian school (cf. \cite{ciliberto} for
a modern exposition).

The invariants of the fundamental groups with known strong relation
to singularity theory  are the Alexander type
invariants, introduced in \cite{homologyabcov} and called their {\it
  characteristic varieties}. The connections besides singularity
theory run through the knot theory, the Hodge theory of quasi-projective varieties,
study of elliptic fibrations, symplectic geometry to mention a few.

There are three major approaches to the study of characteristic varieties
of fundamental groups. One is topological, allowing their calculation in
terms of a presentation of the fundamental group via generators and
relations, obtained typically using braid monodromy. The other one is
geometric, going through a study of homology of the abelian covers and 
eventually leading to determination of characteristic varieties in terms of local type of
singularities and dimensions of the linear systems determined by the 
divisor and the local type of singularities. Finally, one can
calculate the characteristic varieties using Deligne extensions of 
bundles endowed with a flat connection.
Whole theory is a combination of methods and ideas from all
these areas.

Many results related to the discussion of this paper are presented in
volume \cite{jaca} where for the most part the case of plane curves was considered. The exposition which follows, describes 
  a generalization to the context of the complements to divisors on smooth simply
connected surfaces.  
 A very fruitful approach to a study of the
complements to divisors is via resolutions of singularities,
reducing the case when a divisor has arbitrary singularities to the case of divisors with normal crossings. In this
way one replaces the complexity of the divisor by complexity of
compactification and the complexity of individual components.
The goal here is rather to study how complexity of singularities affects
fundamental groups of the complements. Trying to make this paper 
more independent, we included some basic material which is scattered
through existing literature and for which we could not find good
references (e.g. theory of branched covers, the
relation between quasi-adjunction and multiplier ideals etc.). We also 
survey several results on the fundamental groups which appear in the last 
10-20 years providing an overview of the new results in this area. 
Several results here are new or did not appear in
the literature: they include the divisibility of Alexander
polynomials of complements on simply-connected surfaces, extending the
case of plane curves (cf. Theorem \ref{theoremdivisibility}), calculation of characteristic
varieties in terms of classes of irreducible components in Picard
group and invariants of quasi-adjunction of singularities
(cf. Theorem \ref{charvarposition}) and others.

The content of the paper is as follows. In section \ref{basics}
we discuss an analog of classical method of Van Kampen
(cf. \cite{vankampen})  to obtain presentations of the fundamental 
groups of the complements to divisors on smooth surfaces in terms
of mapping class group valued monodromy associated to a divisor.
We also review conditions on a divisor which allow to deduce that 
the fundamental group of the complement is abelian. 
In section \ref{alexinvariants} we firstly extend the theory of
Alexander invariants of plane algebraic curves (cf. \cite{mealex}, \cite{me2009})
to the complements of curves on smooth projective surfaces
 (for an earlier work cf. \cite{degt2001}). 
In particular we obtain a result unifying the divisibility
theorems in the case of plane curves, showing the divisibility of global
Alexander polynomials respectively in terms of local Alexander polynomials
and the Alexander polynomials at infinity. 
Many results depend on some sort of positivity assumptions of 
the components which suggest an interesting problem understanding 
the fundamental groups and its invariants when positivity is lacking.
Theory of Alexander invariants is closely related to the study of
homology of abelian covers. In section \ref{branchedcoverssect} we 
present basic definitions and then describe approaches enumerating 
covers either in terms of subgroups of fundamental groups or in terms
of eigensheaves of direct images of 
the structure sheaf. The most interesting results about Alexander
invariants are obtained through interaction of topological and
algebro-geometric view points. The last part of this section deals with 
multivariable Alexander invariants from topological view point. We
included a brief discussion of multivariable Alexander invariants for
quasi-projective invariants in higher dimensions including recent
results on propagation (cf. \cite{maximbook} for another recent
overview of this and related aspects). Section \ref{section4}
discusses a calculation of characteristic varieties in terms of superabundances of 
the linear systems associated with a divisor on a smooth projective surface 
using ideals of quasi-adjunction of singularities of the divisor. 
The ideals of quasi-adjunction,  defined in terms of branched covers of the germs of divisors, can be viewed as the  
multiplier ideals which received much attention over last 20-30
years. The role of these ideals in the study of the fundamental groups is to specify
the linear system which dimensions determine the characteristic
varieties and hence allow to give their geometric description. 
The section contains also another description of characteristic varieties
using the Deligne's extension and ends with a brief review of the relations between the
characteristic varieties and  
other invariants studied in Singularity theory, including
Bernstein-Sato  polynomials and Hodge decomposition of characteristic
varieties. Section \ref{asymptsection} mostly is based on recent preprint
\cite{asymptotics} which describes the results on distribution of 
Alexander type invariants when complexity (in appropriate sense) of the divisor 
increases. We describe the finiteness results when one
searches for fundamental groups of the complements  with large free
quotients. The last section discusses several recent calculations of
the fundamental groups of the complements. In the 80s scarcity of 
examples of quasi-projective groups and fundamental groups of the
complements was viewed as impediment to development of general theory. In recent 
years this problem was amply addressed and we present some of the most
consequencial results. 

In these notes, we tried at least to direct a reader to the most
importnat recent developments but nevertheless several important 
topics were not covered here. Those missing include the relation between the
Alexander invariants and the Mordell-Weil groups of isotrivial 
fibrations (cf. \cite{mathannalen}), Chern numbers of algebraic
surfaces and arrangements of curves (cf. \cite{urzua}), free subgroups of the fundamental
groups (cf. \cite{zaidenberg}), virtual nilpotence of virtually
solvable quasi-projective groups (cf. \cite{aranori}), singularities
of varieties of representations of the fundamental groups (cf. \cite{millson}), the complements to symplectic curves
(cf. \cite{ludmil1},\cite{ludmil2}, \cite{golla}) among
others. 
 
The theory described below appears to be far from completion.
Many interesting problems remain very much open (some are mentioned
throughout the text) and
a thorough understanding of the fundamental groups or homotopy 
type of quasi-projective varieties is still out of reach. 

Finally, I want to thank Alex Degtyarev as well as the referee of
this paper for reading the final version of the
text and very helpful comments.

\section{ Braid monodromy, presentations of fundamental groups and
  sufficient conditions for commutativity}\label{basics}

\subsection{Braid monodromy presentation of fundamental groups.}
In the case of plane curves, Zariski-van Kampen method\index{Zariski-van Kampen method}
(cf. \cite{zariski29}, \cite{vankampen}) is the oldest
tool for finding presentations of the fundamental groups\index{presentation of the fundamental group} of the
complements.
A convenient way to state the theorem is in terms of braid monodromy.
Its systematic use was initiated in \cite{moishezon} and in such form admits a natural generalization
to the complements to divisors on 
arbitrary algebraic surface which we describe in this section. Braid monodromy became an
important tool in symplectic geometry (cf. \cite{ludmil1}). A good
exposition of braid
monodromy of curves on ruled surfaces can be found in
\cite{degt11c}, Section 5.1. 

Let $X$ be a smooth projective surface and let $D$
be a reduced divisor on $X$. To describe a presentation of
$\pi_1(X\setminus D,p), p\in X\setminus D$ we make several choices,
on which the presentation will depend.
\begin{itemize} 
\item Select a pencil \footnote{i.e. a family of divisors parametrized by $\PP^1$}
 of hyperplane sections of
$X \subset \PP^N$ 
, generic for the pair
$(X,D)$. Its base locus is a generic
codimension 2 subspace $P \subset \PP^N$ and we can consider the projection
with the center at $P$,
i.e. the map $\PP^N\setminus P\rightarrow \PP^1$ sending to $p\in
\PP^N\setminus P$ to the 
 hyperplane containing $P$ and $p$. Its restriction to $X$ produces a regular map 
$\pi: X \setminus X\cap P \rightarrow \PP^1$. Denoting by $\widetilde X$ the
blow up of the surface $X$ at the base locus $P\cap X$ of the pencil,
we obtain a regular map ${\widetilde X} \rightarrow \PP^1$. Assuming that $P$ was
selected so that $D \cap P=\emptyset$ and still denoting by $D$ its
preimage in $\widetilde X$ we obtain the map $\tilde \pi: {\widetilde X}\setminus D \rightarrow \PP^1$.
Seifert-van Kampen theorem implies that $\pi_1( X\setminus
D)=\pi_1({\widetilde X}\setminus D)$ and so we can do calculations
 on
$\widetilde X$.
\bigskip

\item  Let $B=\{b_1, \dots ,b_k\} \subset \PP^1$ be the set consisting of
the critical values of $\tilde \pi$ \footnote{those are absent in the
  classical case  on pencils of lines $X=\PP^2$ of Zariski-van Kampen theorem.} and the images of the fibers
of $\pi$, either containing a singular point of $D$ or containing a
point of $D$ which is critical point of restriction $\pi\vert_ D$.
\bigskip
\item Let  $\Omega \subset \PP^1$ be a subset, containing $B$ and isotopic to a
disk in $\PP^1$, and let $b_0 \in \partial \Omega$ be a point on the
boundary of $\Omega$.  
\bigskip
\item 
Let 
$\partial B_{\epsilon}(p)$ be the boundary of a
small ball $B_{\epsilon}(p)$ in $X$ \footnote{we assume that there are no
  vanishing cycles corresponding to critical points of $\pi$ and no
  points of $D$ inside this ball} 
centered at a point $p \in X \cap
P$ or, equivalently, the boundary of a small regular neighborhood of the
exceptional curve $E_p$ in $\tilde X$ contracted to $p\in X$. 
The map $\tilde \pi$ restricted to $\partial B_{\epsilon}(p)$ is the Hopf fibration $\partial B_{\epsilon}=S^3\rightarrow
\PP^1=S^2$.  Using its trivialization over $\Omega$, we define a section over $\Omega\setminus B$:
 $s_p: \Omega\setminus B \rightarrow \tilde \pi^{-1}(\Omega\setminus
 B)$.  
\bigskip
\item 
Let $F_{b_i}, i=0,1,...,k$ be the fiber of $\tilde \pi$
over $b_i$. The curves $F_{b_i}, i=1,....,k$ either have singularities at critical
points of $\pi$ or contain singular points of $D$ or have non-transversal intersections with $D$, while
$F_{b_0}$ is  smooth closed Riemann surface having genus
$g=\frac{F_{b_0}(F_{b_0}+K)}{2}+1$ where $K$ is the canonical divisor
of $X$. 
\bigskip
\item
For any $p \in P\cap X$, let $\bar F_{b_0}^{\circ}$ be
the surface with one connected boundary component obtained by removing from $F_{b_0}$ 
its intersection with the above regular 
 neighborhood of $E_p$.
Denote by 
$\cM(\bar F_g^{\circ},[d])=Diff^+( \overline{F_{b_0}\setminus (F_{b_0} \cap B_{\epsilon}(p))}, [F_{b_0} \cap D])$ 
 the mapping class group of the Riemann surface with boundary with
$d$ marked points (cf. \cite{farb}) i.e. the group of isotopy classes of
orientation preserving diffeomorphisms taking the subset $[d]$
 of cardinality $d$ into itself and constant on the boundary of the
 Riemann surface.
\end{itemize}
\begin{definition} \it  The braid monodromy\index{braid monodromy} of the pair $(X,D)$ (for selected
  pencil on $X$) is the monodromy map 
\begin{equation}\label{monodromymap}
\mu: \pi_1(\Omega\setminus B,b_0)
\rightarrow \cM(\bar F_g^{\circ},[d])
\end{equation}
obtained by 

a) selecting a loop (denoted in b) and c) below as $\gamma$) for each homotopy class in $\pi_1(\Omega\setminus
B,b_0)$,  

b) a trivialization of the locally trivial fibration $\pi^{-1}(\gamma)
\rightarrow \gamma$ i.e. a differentiable map $\pi^{-1}(b_0) \times [0,1] \rightarrow
\pi^{-1}(\gamma)$ inducing a diffeomorphism of the fiber over $t \in
[0,1]$ onto the fiber over the image of $t$ in parametrization $[0,1]
\rightarrow \gamma$ of the loop. 

c) assigning to $\gamma$ the diffeomorphism of $\bar F_g^{\circ}=\pi^{-1}(b_0)$
sending a point $q \in \pi^{-1}(b_0)$ to the point $q' \in \pi^{-1}(b_0)$
to which the trivialization mentioned in b) takes the end point $q \times 1$ of the segment $q\times
[0,1] \subset \pi^{-1}(b_0) \times [0,1]$ in $\pi^{-1}(\gamma)$.
\end{definition}

One verifies that, though the diffeomorphism in c) depends on both, the
loop $\gamma$ in a) and the trivialization in b), its class in the
mapping class group does not depend on these choices.

Recall that the mapping class group $\cM(\bar F_g^{\circ},[d])$ acts on
$\pi_1(\bar F_g^{\circ}\setminus [d],q)$ (here $q$ is the base point which
we assume is on the boundary of $\bar F_g^{\circ}$). For example in the case $g=0$ the group $\cM(\bar F_0^{\circ},[d])$ is
the Artin's braid group on $d$-strings i.e. the group of orientation
preserving diffeomorphisms of a 2-disk $\Delta$, constant on the boundary and taking
into itself a given subset of $\Delta$ of cardinality $d$. It has a well known
presentation:
\begin{equation}\label{artin} 
 < \sigma_1,...,\sigma_{d-1}, \ \ \ \vert 
   \sigma_i\sigma_{i+1}\sigma_i=\sigma_{i+1}\sigma_i\sigma_{i+1} \ \ \
\sigma_i\sigma_j=\sigma_j\sigma_i, 1<\vert i-j\vert \ \ \  >
 \end{equation}
Note that the center of (\ref{artin}) is generated by
$[\sigma_1(\sigma_2\sigma_1)(\sigma_3\sigma_2\sigma_1)....(\sigma_{d-1}...\sigma_1)]^2$
(cf. \cite{gonzalez} Sect. 4.3).
The action on the free group $\pi_1(\Delta \setminus [d],p)$ is
given by 
\begin{equation}
   \sigma_i(t_i)=t_{i+1}, \sigma_i(t_{i+1})=t_{i+1}^{-1}t_it_{i+1},
   \sigma_i(t_j)=t_j, \ \ j \ne i,i+1 \footnote{this implies 
that $\sigma_i^{-1}(t_i)=t_it_{i+1}t_i^{-1}, \sigma_i^{-1}(t_{i+1})=t_i$.}
 \end{equation}
(which is the canonical action of the mapping class group on the
fundamental group for appropriate choice of generators $t_i$ of the latter).
This way in the case of $X=\PP^2$ one obtains the monodromy with
the values in the Artin's braid group, the case described in \cite{moishezon}.
The homomorphism (\ref{monodromymap}) in \cite{moishezon} is described
in a more combinatorial form, as a product of collection of braids. The
ordered collection of factors in this product is the collection of braids corresponding to so
called ``good ordered system of generators'' of the free group $\pi_1(\Omega
\setminus B,b_0)$ (cf. \cite{moishezon} for details).  

To define the final ingredient for our presentation of
$\pi_1(X\setminus D)$,  we consider 
the gluing map of the boundaries of $\pi^{-1}(\Omega)$ and
$\pi^{-1}(\PP^1 \setminus
\Omega)$ which can be viewed as a map $\Phi: \pi^{-1}(\partial
\Omega) \rightarrow \pi^{-1}(\partial (\PP^1\setminus \Omega))$, both
spaces being locally trivial fibrations over $\partial \Omega=S^1$, preserving
the set  $D \cap \pi^{-1}(\partial \Omega)$ and commuting with projection onto $S^1$.
Such map takes the loop $s_p(\partial \Omega)$  (as above, $s_p$ is
a section of restriction of the Hopf bundle over $\PP^1$) to the loop $S^1 \rightarrow
S^1 \times (F_{b_0} \setminus [d]) \rightarrow F_{b_0} \setminus [d]$ and 
hence determines a conjugacy class in the fundamental group of its
target. We shall denote this class $\rho_{X,D}$.  In the case of plane
curve of degree $d$ transversal to the line at infinity and pencil of lines,
complement to the base point is the total space of line bundle ${\mathcal{O}}_{\PP^1}(1)$, the gluing map $\Phi$ induced by positive generator
of $\pi_1(GL_2(\CC))$ which shows that  
 $\rho_{X,D}=\gamma_1
\cdot...\cdot \gamma_d$ is the product of standard ordered system of
generators of fundamental group of the complement in generic fiber to
the intersection of this fiber with the curve.

The mapping class group valued monodromy  determines the fundamental group as follows:

\begin{theorem}\label{vankampen} One has the isomorphism:
\begin{equation}\label{presentation}\pi_1(X\setminus D)=\pi_1(C_0\setminus C_0\cap
  D)/\{(\mu(\gamma_j)\alpha_i)\alpha_i^{-1},\rho_{X,D}\}
\end{equation}
\end{theorem}

Theorem \ref{vankampen} reduces a calculation of the fundamental group to  the
calculation of the braid monodromy and the element $\rho_{X,D}$. The literature
on calculations of braid monodromies of curves is very large and is
very hard to review. We refer to \cite{moishezon} and
\cite{moishteicher} where explicite expressions were obtained for the 
braid monodromy of smooth plane curves, branching curves of generic
projections of smooth surfaces in $\PP^3$, generic arrangements of
lines and branching curves of generic projections of various embeddings of 
quadric. The survey \cite{teicher} and the book \cite{degt11c} also are good
references for more recent developments.
We refer to the former for examples of calculations of fundamental
groups using van Kampen method and references to other works on
calculation of braid monodromy and the latter for  computer use
in calculations of braid monodromy and the fundamental groups.

Besides the fundamental group, the braid monodromy defines the
homotopy type of the complement (cf. \cite{me2complex} for precise statement). It is
however and open problem, if the homotopy type of the complement $X\setminus
D$ is determined by the fundamental group and the topological Euler
characteristic of the complement (cf. \cite{me2complex} for
a discussion of this problem). Considering dependence of the braid monodromy on
the curve and numerous choices made in its construction, in 
\cite{ajcarmona} the authors found conditions implying that the
homeomorphism type of the triples
$(\PP^2,L,C)$, where $C$ is a plane curve and $L$ is one of the lines
of the pencil used to construct the braid monodromy (the line at
infinity), determines the braid monodromy. 
Braid monodromy is an essential tool in showing
the existence of symplectic singular curves not isotopic to algebraic
ones (cf. \cite{ludmil1} and \cite{moishezonchisini}).

\subsection{Abelian fundamental groups} The question ``whether the
fundamental group of the complement to a nodal curve is abelian'' 
was known as ``Zariski problem'' since it was realized that Severi's
proof of irreducibility of the family of plane curves with fixed
degree and the number of nodes is incomplete
(cf. \cite{severi}). Zariski derived commutativity of the fundamental
groups of the complements to nodal curves using that irreducibility implies
existence of degeneration of a nodal curve to a union of lines without
points of multiplicity greater than two. Once one has degeneration, the relation between
 fundamental groups of the complements to a curve and to its degenerations
(i.e. that given a degeneration $C_0=limC_t$ one has surjection
$\pi_1(\PP^2\setminus C_0)\rightarrow \pi_1(\PP^2\setminus C_t)$ which
is a consequence of definition of the braid monodromy and presentation
(\ref{presentation})) implies the commutativity. 
Severi statements (and with it the Zariski
proof \cite{zariskialgsurf}) was eventually validated 
(cf. \cite{harris-86}). A proof of commutativity, based on connectedness theorem, was found prior to this 
by W.Fulton (\cite{fulton-80}) for algebraic fundamental group and by P.Deligne \cite{deligne}
in topological case. 

The central result on commutativity of fundamental groups of
complements to divisors is due to Nori
with the key step being a generalization of Lefschetz hyperplane section
theorem (cf. \cite{nori}). See \cite{fulton87} Chapter 5 and \cite{lazarsfeld}
Chapter 3.
\begin{theorem}\label{nori} (Nori's weak Lefschetz theorem). Let $U$ be a
  connected complex manifold of dimension greater than one 
 and let $i: H \rightarrow U$ be the embedding of a connected compact
  complex-analytic subspace defined by a locally principal sheaf of
  ideals.  Let $q: U\rightarrow X$ be a locally invertible map to a
  smooth projective variety, $h=q \circ i$, and $R \subset X$ be
  a Zariski closed subset. Assume
  that $\mathcal{O}_U(H)\vert H$ is ample. Then

A:  $G={\rm Im} \pi_1(U\setminus q^{-1}(R)) \rightarrow \pi_1(X\setminus
R)$ is a subgroup of a finite index.

B:
If $q(H)\cap R=\emptyset$ then $\pi_1(H) \rightarrow \pi_1(X\setminus
R)$ is a subgroup of a finite index.

C: If $\dim X=\dim U=2$ then index of subgroup $G$ of $\pi_1(X\setminus
R)$ is at most  $\frac{(Div(h))^2}{H^2}$ where the
divisor in numerator is the first Chern class of $h_*\mathcal{O}_H$,
the Cartier divisor on $X$ corresponding to
the divisor $H$ on $U$ (cf. \cite{nori}, 3.16 for details). 
\end{theorem}

If $q$ is embedding and $H$ is reduced, this becomes Zariski-Lefschetz hyperplane
section theorem  (cf. \cite{zarhyperplane}, \cite{hamm}). One has to
note a subtlety in the finiteness of index in A and the index bound in
C (cf. \cite{hain09}). Typically $\pi_1(q(H))$ is much bigger than
$\pi_1(H)$:  for example if $H \rightarrow q(H)$ is
normalization and $H$ is rational and $q(H)$ nodal then $\pi_1(q(H))$ is free 
group with the rank equal to the number of nodes. Nevertheless the following is
still open:


\begin{problem}\label{noriproblem} (M.Nori) Let $D$ be an effective divisor of a surface $X$ and $D^2>0$.
Let $N$ be a normal subgroup of $\pi_1(X)$ generated by the images of
the fundamental groups of the normalizations of all irreducible
components of $D$. Is the index of $N$ in $\pi_1(X)$ finite? 
In particular, can a surface with infinite fundamental  group contain
a rational curve with positive self-intersection?
\end{problem}


One of the main consequence of theorem \ref{nori} is the following:
\begin{corollary}\label{norinodal} Let $D$ and $E$ be curves on smooth projective surface
 intersecting transversally and such that $D$ has nodes as the only singularities. Assume
  that for each irreducible component $C$ of $D$ one has $C^2>2r(C)$.
Then $N=Ker (\pi_1(X\setminus (D\cup E))\rightarrow \pi_1(X\setminus E))$
is a finitely generated abelian group and the centralizer of $N$ has a
finite index in $\pi_1(X\setminus (D\cup E))$.
\end{corollary}

This immediately implies that the fundamental group of a nodal curve
in $\PP^2$ is abelian (indeed, for irreducible curve of degree $d$ the
maximal number of nodes $r(C)=\frac{(d-1)(d-2)}{2}$ satisfies $2r(C)<d^2$).
Moreover, for an irreducible plane curve with $r(C)$ nodes and
$\kappa(C)$ cusps (with local equation $u^2=v^3$) one obtains that $\pi(\PP^2 \setminus C)$ is abelian
if $C^2>6\kappa(C)+2r(C)$ (apply \ref{norinodal} to resolution of
cusps only). On non-simply connected surfaces, the kernel
$\pi_1(X\setminus C) \rightarrow \pi_1(X)$ belongs to the center if
$C^2>4r(C)$ (though for $4r(C) \ge C^2 >2r(C)$ the centralizer of this
kernel still has a finite index, see \cite{nori} p. 324). 

A result pointing out toward a positive answer to the Problem
\ref{noriproblem} appears in \cite{lasell} and can be stated as follows. Let
$X$ be a smooth projective variety and $Y$ be a subvariety such that 
$\pi_1(Y) \rightarrow \pi_1(X)$ is surjective. Let $f:
Z\rightarrow Y$ be  dominant morphism where all irreducible
components of $Z$ are normal. Let $N$ be the normal subgroup of
$\pi_1(X)$ generated by the images of irreducible components of
$Z$. Then for any $n$, $\pi_1(X)/N$ has only finitely many 
$n$-dimensional complex representations, all of which are semi-simple 
 (an obvious attribute a finite group).

Applications of Nori's results include \cite{shimada97},
\cite{shimada98}.
Paper \cite{shimada03} studies further exact sequence of the fiber spaces.
Papers \cite{tokunaga00}, \cite{tokunaga03}  give conditions in opposite direction than
the one considered by Nori, guaranteeing that the fundamental group of the
complement is NON abelian. An important outcome of Nori's Weak
Lefschetz theorem is that it provides an instance for the finiteness of the index of the image of the
fundamental groups for compositions $H \rightarrow C \rightarrow X$
where as above $H,X$ are smooth and $H \rightarrow C$ is dominant.
This more general context was considered in \cite{hain09} in the
framework of the study of the representations of the fundamental groups of 
varieties dominating divisors in the moduli spaces of (pointed) curves 
(with level structure), under the heading of ``non-abelian strictness theorems''.

\section{Alexander Invariants}\label{alexinvariants}

\subsection{Alexander polynomials}\label{alexandercurves}

Alexander polynomial of knots and links was introduced by James W. Alexander in
1928 (cf. \cite{alexander}). In response to a question by D.Mumford
(cf. \cite{mumfordalex}),
who noticed its relation to a construction used by O.Zariski, 
the Alexander polynomials were put in \cite{mealex} in the context of complements to plane algebraic
curves. This extension blends the algebraic geometry and the methods introduced by Fox
(cf. \cite{fox}) and Milnor (cf. \cite{milnoralex}) for the study of
knots.  Various generalizations, in which (a zero set of) polynomial was replaced by
a subvariety of a torus and involving
germs of singularities (cf. \cite{isolatednonnormal}), extensions to higher dimensions
(cf. \cite{meannals}) and to curves in complex surfaces (cf. \cite{degt2001}), were considered as well. 
A twisted versions (cf. \cite{cogoflorens}, \cite{maxim06}, \cite{me2009}) were
studied more recently. Below we shall describe the Alexander
polynomials in the context of divisors on simply connected surfaces 
and refer to \cite{medevelopment} for the history of the subject and
further references.

Let $X$ be a smooth simply connected projective surface and let
$D$ be a divisor on $X$ with irreducible components $D_i$. Let $\{
[D_i]\}=H^2(D,\ZZ)=\oplus_i H^2(D_i,\ZZ)$ denote a free abelian group generated by the
cohomology classes corresponding to the irreducible components of $D$. For $\alpha \in H_2(X,\ZZ)$, we put $D_{\alpha}=\sum_i
(\alpha,[D_i])[D_i] \in \{[D_i]\}$, where $[D_i] \in H_2(X,\ZZ)$ is
the fundamental class of the component $D_i$ and denote by $\{D_{\alpha}\}$ the
subgroup of $\{[D_i]\}$ generated by the classes $D_{\alpha}, \alpha \in H_2(X,\ZZ)$.
$\{D_{\alpha}\}$ is the image of the homomorphism $H_2(X,\ZZ)
\rightarrow H^2(D,\ZZ)$ obtained using 
the excision and duality
isomorphisms giving $H_2(X,X-D,\ZZ)=H_2(T(D),\partial T(D),\ZZ)=H^2(D,\ZZ)$ where $T(D)$ is
a tubular neighborhood of $D$ in $X$ and $\partial T(D)$ is its
boundary.  From the exact
sequence:
\begin{equation}\label{homologysequence}
     H_2(X,\ZZ) \rightarrow H^2(D,\ZZ) \rightarrow
     H_1(X\setminus D,\ZZ) \rightarrow H_1(X,\ZZ)=0
\end{equation} 
we deduce that 
\begin{equation}\label{homologycomplement}
    \{[D_i]\}/\{D_{\alpha}\}=H_1(X\setminus D,\ZZ).
\end{equation}
For example for an irreducible projective (resp. affine) plane curve $D$ of
degree $d$ we obtain $H_1(\PP^2\setminus D,\ZZ)=\ZZ/d\ZZ$
(resp. $H_1(\CC^2\setminus D)=\ZZ$). 

Alexander polynomial is an invariant of the complement to a reduced divisor $D$
and  a surjection $\phi:\pi_1(X\setminus D) \rightarrow C$ where $C$
is a cyclic group.   We state the definition for a finite CW complex 
$Y$ endowed with a surjection $\phi: \pi_1(Y)\rightarrow C$ such that 
 $H_1(Y_{\phi},\QQ)$ is finite dimensional where $Y_{\phi}$ is the covering space
 corresponding  to the subgroup $Ker(\phi) \subset
\pi_1(Y)$ (cf.\cite{hatcher} sect.1.3)\footnote{In this case we
  call $Y_{\phi}$ 1-finite}. 

If  $C$ is finite then the finiteness of the dimension of $H_1(Y_{\phi},\QQ)$ is automatic \footnote{An example of infinite cyclic
covers which is infinite in dimension 1 is given by the complement to
a set $[3]$ containing 3
points in $\PP^1$. Let $(a,b)$ be generators of the free group $\pi_1(\PP^1\setminus
[3])$ and $\phi$ is the quotient of the normal subgroup generated by
$b$. Then $\PP^1\setminus [3]$ is homotopy equivalent to a wedge of
two circles and $(\PP^1\setminus [3])_{\phi}$ 
can be viewed as a real line with the circle attached at each integer
point of this line with the covering group $\ZZ$ acting via translations. In
particular $H_1(\PP^1\setminus [3])_{\phi},\ZZ)$ is a free abelian group
with countably many generators.}. If $H_1(Y,\ZZ)$ is
infinite cyclic then $H_1(Y_{\phi},\QQ)$ also is finite-dimensional as
follows for example from ({\ref{multt-1}}) below. 

For the covering map $Y_{\phi}\rightarrow Y$, we have the exact compactly
supported homology sequence corresponding to the sequence of chain complexes 
\begin{equation}0 \rightarrow C_*(Y_{\phi},\QQ) \buildrel t-1 \over \rightarrow C_*(Y_{\phi},\QQ) \rightarrow C_*(Y,\QQ) \rightarrow 0
\end{equation} 
Here the first two terms are viewed as the modules
over the group ring $\QQ[\ZZ]=\QQ[t,t^{-1}]$, where $t$ denotes
preferred generator of $C=\ZZ$ in multiplicative notations, and the
left map being multiplication by $t-1$. Hence
\begin{equation}\label{multt-1}
H_2(Y,\QQ) \rightarrow H_1(Y_{\phi},\QQ) \buildrel t-1\over \rightarrow H_1(Y_{\phi},\QQ) \rightarrow H_1(Y,\QQ) 
\end{equation}
$$\rightarrow 
H_0(Y_{\phi},\QQ) \buildrel (t-1) \over \rightarrow H_0(Y_{\phi},\QQ)
$$
Consider 
the cyclic decomposition 
of $H_1(Y_{\phi},\QQ)$, viewed as a module over
$\QQ[t,t^{-1}]$,
\begin{equation}\label{cyclic}
   H_1(Y_{\phi},\QQ)=\oplus \QQ[t,t^{-1}]^{a_0}\oplus_p \QQ[t,t^{-1}]/(p(t)) 
\end{equation}
where the summation is over a finite number of monic
polynomials $p$.

One of immediate consequences is that if $rk H_1(Y,\QQ)=1$,  then the multiplication by $t-1$ in the
top row in (\ref{multt-1}) is surjective (since clearly the multiplication
by $t-1$ is trivial on $H_0$) and hence in (\ref{cyclic}) 
  $a_0=0$ \footnote{the condition $a_0=0$ is equivalent to finite dimensionality of $H_1(Y_{\phi},\QQ)$ 
 over $\QQ$} 
Moreover, $(t-1)^{\alpha}, \alpha \in \NN$ is not among the polynomials $p(t)$.
\begin{definition} Let $Y$ be a CW-complex as above.

If $a_0=0$ in the decomposition
(\ref{cyclic}) one defines the Alexander polynomial\index{Alexander polynomial} $\Delta(t)$ of
$(Y,\phi)$ as the order of the $\QQ[t,t^{-1}]$-module $H_1(Y_{\phi},\QQ)$ i.e. as the product:
\begin{equation}\label{order}
   \Delta(t)=\prod p(t) 
\end{equation}
In the case when $X$ is a smooth projective surface and $D$ is a reduced
divisor, we call $\Delta(t)$, the {\it global} Alexander polynomial of
$X \setminus D$ (and the surjection $\phi$ of its fundamental group).
\end{definition}
$\Delta(t)$ has integer coefficients, is well defined up to $\pm
t^i,i \in \ZZ$ and, it follows from (\ref{multt-1}) that, $rk H_1(Y,\QQ)=1$ implies $\Delta(1)\ne 0$. 
If the target of $\phi$ is a finite cyclic group then, since
$\QQ[\ZZ_n]=\QQ[t,t^{-1}]/(t^n-1)$, instead of (\ref{cyclic}) one has  
\begin{equation}
      H_1(Y_{\phi})=\oplus [\QQ[t,t^{-1}]/(t^{{\rm ord} C}-1)]^{a_o}\oplus \QQ[t,t^{-1}]/p(t)
\end{equation}
and the Alexander polynomial defined to be the order (\ref{order}) of
this $\QQ[t,t^{-1}]$-module. 

This construction, when applied to the intersection of $D$ with a small
sphere about a singular point $P$ of $D$ and when $\phi$ is given by
evaluation of the
linking number in this sphere with $D$, yields the {\it local Alexander
polynomial}. It is not hard to show the following (cf. \cite{le}):
\begin{proposition} The local Alexander polynomial coincides with the
  characteristic 
polynomial of the local monodromy of the singularity of $D$ at $P$
\end{proposition}
 
\subsection{A Divisibility Theorem\index{Divisibility Theorem}}\label{divisibility11} This is the
central result on the Alexander polynomials allowing to obtain
information about $\Delta(t)$ in terms of geometry of $D$. In many
cases it leads to its determination or makes  possibilities for $\Delta(t)$
rather limited. The case of curves in
$\PP^2$ appears in \cite{mealex}.
\begin{theorem}\label{theoremdivisibility} Let $D=D_1\bigcup
  D_2$ be a divisor on $X$ such that $D_1$ is ample.
Let $\phi_{X\setminus D}: \pi_1(X\setminus D) \rightarrow 
H_1(X\setminus D,\ZZ)\rightarrow C$ be a
surjection onto a cyclic group $C$ (either infinite or finite) and 
  $T(D_1)$ denotes a small regular neighborhood of the divisor $D_1$. 
Assume also that $\phi$ maps the meridian \footnote{i.e. a loop consisting of a path connecting the
  base point with a point in vicinity of the irreducible component of
  $D$, the oriented boundary of a small disk in $X$ transversal to
  this component of $D$ at its
  smooth point and not intersecting the other components of $D$, with the same path used to return back
  to the base point; orientation of the small disk must be
positive i.e. such that its orientation will be compatible with the
complex orientations of smooth locus of divisor and the ambient
manifold. As an element of the fundamental group, only the conjugacy
class of a meridian is well defined.} of each irreducible component
of $D_1$
to the generator of $C$ corresponding to the variable $t$ of the Alexander polynomial. Then

1. The cyclic cover $(X\setminus D)_{\phi}$ is 1-finite 
and so is $(T(D_1)\setminus D\cap T(D_1))_{\phi_T}$ 
where $\phi_T$ is the composition $\pi_1((T(D_1)\setminus D\cap T(D_1)))
\rightarrow 
H_1((T(D_1)\setminus D\cap T(D_1)))\rightarrow
H_1(X\setminus D) \rightarrow C$ of the map induced be embedding and
the surjection $\phi_{X\setminus D}$. 

2.Let  $\Delta_{\phi_{X\setminus D}}, \Delta_{\phi_T}$
be the Alexander polynomials of $X\setminus D$ and $T(D_1)\setminus
D\cap T(D_1)$ corresponding to surjections $\phi$ and $\phi_T$ respectively.
 One has the following
divisibility:
\begin{equation} \label{divisibility}
  \Delta_{\phi_{X\setminus D}}(t) \vert \Delta_{\phi_T}
\end{equation} 

3. Let $\{p_i\}$ be the set consisting of singular points of $D_1$ and the
points $D_1\cap D_2$. For each $p_i$ let $B_{p_i}$ denotes a small ball in $X$
centered at this point. Let $\Delta_{p_i}$ denotes the Alexander polynomial
of $B_{p_i}\setminus D \cap B_{p_i}$ relative to the map $\phi_i:
H_1(B_{p_i}\setminus D\cap B_{p_i}) \rightarrow C$ induced by
embedding $B_{p_i}\setminus D \cap B_{p_i}\rightarrow X\setminus D$.
Then
\begin{equation}
\Delta_{\phi_T}=(t-1)^{\alpha}\prod \Delta_{p_i} \ \ \
\alpha \in \ZZ.
\end{equation}

In particular, the roots of the Alexander polynomials
$\Delta_{\phi_{X\setminus D}}$ and $\Delta_{\phi_T}$ are roots of unity.
\end{theorem}
\begin{proof} Ampleness of $D_1$ implies that for $n \gg 0$ there exist a
  smooth curve $\tilde D_1$ on $X$ linearly equivalent to $nD_1$ and belonging
  to $T(D_1)$. Moreover,  we can assume that $\tilde D_1$ is transversal to all
  components of $D$.

  Weak Lefschetz theorem (cf. \cite{hamm},\cite{nori})
 implies that the composition in the middle row of the following
 diagram is a surjection: 
  \begin{equation}\label{diag}
 \begin{matrix}  & & Ker \phi_T & \rightarrow &
   Ker \phi_{X\setminus D} \cr
     & & \downarrow & & \downarrow \cr
   \pi_1(\tilde D_1\setminus \tilde D_1\cap D) &
   \rightarrow &
  \pi_1(T(D_1)\setminus D \cup T(D_1)) & \rightarrow & \pi_1(X\setminus
  D) \cr
    & & \downarrow & & \downarrow \cr
   & & H_1(T(D_1)\setminus D \cap T(D_1),\ZZ)
 & & H_1(X\setminus D,\ZZ) \cr
 & &  & \searrow & \downarrow \cr
 & &  & & C \cr
  \end{matrix}
\end{equation} 
Therefore the right map in that row and hence also $Ker \phi_T \rightarrow Ker
\phi_{X\setminus D}$ both are surjective. The
condition that meridians are taken by $\phi_T$ to non-zero element of $C$
implies that the covering space $(T(D_1)\setminus D \cap T(D_1))_{\phi_T}$ is
1-finite (cf. \cite{manuel}) 
and surjectivity of the maps of the kernels implies that so is
$(X\setminus D)_{\phi}$. Since the map in the top row in (\ref{diag})
is surjective, $\QQ[C]$-module $H_1((X\setminus D)_{\phi},\QQ)$
is a quotient of $H_1((T(D_1)\setminus D \cap T(D_1))_{\phi_T},\QQ)$
hence the divisibility relation (\ref{divisibility}) follows.

Finally, taking $T(D_1)$ sufficiently thin, $(T(D_1) \setminus D_1)\setminus\bigcup_i
(B_{p_i} \setminus D\cap B_{p_i})$ can be assumed isotopic
to the trivial $C^{\infty}$-fibration $(T(D_1) \setminus D_1) \setminus \bigcup_i (B_{p_i} \setminus
 D \cap B_{p_i} ) \rightarrow  D_1\setminus \{p_i\}$,  
with the fiber being isotopic to a punctured 2-disk. Due to
assumption that meridians of all components are mapped to generator
corresponding to $t$, the Alexander polynomial of $(T(D_1) \setminus D_1) \setminus
\bigcup_i (B_{p_i} \setminus D\cap B_{p_i})$ is a power of $t-1$. The
decomposition 
\begin{equation}\label{union}
   T(D_1)\setminus D \cap T(D_1)=\left[ (T(D_1) \setminus D_1)\setminus \bigcup_i (B_{p_i} \setminus
  D\cap  B_{p_i})\right] \cup \bigcup_i (B_{p_i}\setminus D\cap B_{p_i})
\end{equation}
induces decomposition of the cover $(T(D_1)\setminus D)_{\phi_T}$
of $T(D_1)\setminus D$
corresponding to subgroup $Ker \phi_T$ of $\pi_1(T(D_1)\setminus D
\cap T(D_1))$
into a union of preimages of each subspace on the
right in (\ref{union}). Now the Mayer -Vietoris sequence implies the
part 3 of the Theorem  
(also the 1-finiteness of $T(D_1)\setminus T(D_1)\cap D)$).
\end{proof}

\begin{corollary}[\cite{mealex}]  Let $C$ be an irreducible curve in $\PP^2$ and $L$ be the
  line at infinity. Then $H_1(\PP^2\setminus C\bigcup L,\ZZ)=\ZZ$ and
  the Alexander polynomial of $\PP^2\setminus C \bigcup L$ 
with respect to the abelianization, divides the product of the
Alexander polynomials of links of all singularities of $C \bigcup L$.
It also divides the Alexander polynomial of the link at infinity
i.e. the Alexander polynomial of the complement  $S_{\infty} \setminus C \cap S_{\infty}$
where $S_{\infty}$ is the boundary of a small (in the metric on
$\PP^2$) regular neighborhood of $L \subset \PP^2$.
\end{corollary}

\begin{proof} It follows from (\ref{homologycomplement}) and Theorem
  \ref{theoremdivisibility} applied to $C$ and $L$ separately. More
  precisely, the part 2 (resp. part 2 and 3) of Theorem \ref{theoremdivisibility} show
 that the global Alexander polynomial divides the Alexander polynomial at infinity (resp. of the
 product of local Alexander polynomials).
\end{proof}

\begin{corollary}[\cite{degt2001}]Let $D$ be a divisor of a simply connected surface
  $X$. Let $S$ be a subset of the set of singular points of $D$
  belonging to 
 an irreducible component $D'$ of $D$ such
  that on log-resolution $\tilde X$ of singularities of $D'$ outside of $S$,
  for proper preimage $\tilde D'$ one has $(\tilde D')^2>0$.
Then one has divisibility:
\begin{equation}
       \Delta_{\phi_{X\setminus D}} \vert \prod_{p_i \in  S}\Delta_{{p_i}}
\end{equation}

\end{corollary}
\begin{proof} Condition on self-intersection implies that $\tilde D'$
  is ample. Now the claim follows immediately from the Theorem \ref{theoremdivisibility}
applied to the proper preimage of $D$ on $\tilde X$ and its component
$\tilde D'$ since $X \setminus D=\tilde X\setminus \tilde D$ because
only points on deleted divisor $D$ are blown up.
\end{proof}

\begin{example} {\it Milnor fibers of homogeneous polynomials and
    arrangements of  lines}
Let $\cA \subset \PP^2$ be an arrangement of lines given by equations
$L_i(x,y,z)=0, i=1,..N$. 
Milnor fiber $\prod L_i(x,y,z)=1$ of the cone $\prod L_i=0$ over this
arrangement  (denoted below $M_{\cA_L}$) can be identified with the $\ZZ/N\ZZ$-cyclic cover of the
complement $\PP^2\setminus \cA$. Theorem \ref{theoremdivisibility}
gives restrictions on the degree of the characteristic
polynomial of the monodromy operator acting on $H_1(M_{\cA},\QQ)$  
(which can be identified with the Alexander polynomial of
$\PP^2\setminus \cA$) in terms of multiplicities of point of $\cA$ along one of the
lines (cf.\cite{trends}). For example, if $\cA$ has only triple points
along one of the lines, it follows that the characteristic polynomial
of the monodromy of Milnor fiber has form
$(t-1)^{N-1}(t^2+t+1)^{\kappa}, \kappa \ge 0 $.
See \cite{dimcabook}, \cite{dimcaarrangements}, \cite{oka} for other numerous applications.
\end{example}

\subsection{Branched covers}\label{branchedcoverssect} A branched
cover\index{branched cover} of a complex space $Y$ is a finite dominant morphism $f:X
\rightarrow Y$.  We will consider only the case when
$X$ is normal and $Y$ is smooth.  
Ramification locus\index{ramification locus} $R_f\subset X$ is the support of the quasi-coherent
sheaf $\Omega_{X/Y}$ and the branch locus is $f(R_f)\subset Y$. It has
codimension 1 (Nagata-Zariski purity
of the branch locus cf. \cite{zariskipurity}). 

Given an irreducible divisor $D \subset Y$ on a complex manifold
$Y$ one associates to $(Y,D)$ a discrete valuation
$\nu_{D}: \CC(Y) \rightarrow \mathcal{N}_{D}$ of
the field of meromorphic functions on $Y$ given by
$\nu_{D}(\phi))=ord_{D}(\phi), \phi \in \CC(Y)$ (cf.\cite{hartshorne}, p.130). Here $\mathcal{N}_{D}$ is the subgroup of $\ZZ$ generated by the values of
$\nu_{D}(\phi), \phi \in {\CC(Y)}$.
For a branched cover $X \rightarrow Y$ and a pair of irreducible divisors
$D \subset Y,{\Delta}' \subset X$ where ${\Delta}'$ is a component of $f^*(D)^{red}$
the map $f^*: \CC(Y) \rightarrow \CC(X)$ induces the map $f^*_\mathcal{N}:
\mathcal{N}_{D} \rightarrow \mathcal{N}_{{\Delta}'}$  .The index $[\mathcal{N}_{\Delta'}:f^*\mathcal{N}_{D}]$ (cf.\cite{zarsam} cf. Chapter 6, $\S 12$) is the ramification index
$e_{D}$ of $f$ along the
component $\Delta'$. One has $e_{D'}=1$, unless $D'$ is a component of $R_f$. 
Restriction of a branched cover $X \rightarrow Y$ onto the complement 
to the ramification divisor induces \'etale map $X\setminus R_f
\rightarrow Y \setminus D$ where $D=f(R_f)$ is the branch locus. 
In particular given a branched cover $f: X \rightarrow Y$, selection of
base point $p\in Y\setminus D$ allows to construct monodromy: 
\begin{equation}\label{monodromy} \pi_1(Y\setminus D) \rightarrow Sym(f^{-1}(p))
\end{equation}
into permutation group of points in the preimage of $p$ assigning to
each loop and a point  $a \in f^{-1}(p)$ the end of the lift of the loop
starting at $a$.

The set of equivalence classes of unramified covers $f_Z: Z
\rightarrow Y\setminus D$, where $f_{Z_1},f_{Z_2}$ are considered to
be equivalent iff
there exists biholomorphic isomorphism $h: Z_1 \rightarrow Z_2$ such
that $f_{Z_1}=f_{Z_2} \circ h$, is in one to one correspondence 
with the subgroups of $\pi_1(Y\setminus D,p)$ where $p \in Y
\setminus D$ is a
base point. The correspondence is given by assigning to $f_Z$ the
subgroup $(f_Z)_* \pi_1(Z,p') \subset \pi_1(Y\setminus D,p)$.
This correspondence depends on a choice of a base point 
$p' \in f^{-1}(p) \subset X\setminus R_f$, but the subgroups
corresponding to $p',p^{\prime\prime} , p'\ne p^{\prime \prime}$ are conjugate.
A cover $Y \rightarrow X$
is called Galois if the corresponding subgroup is normal. The 
quotient of the fundamental group by this subgroup is the Galois group
of the cover. This group is the image of the monodromy (\ref{monodromy}).

A branched cover is Galois
if ond only if the extension of the fields of meromorphic functions
$\CC(X)/\CC(Y)$ is Galois and the Galois group of the cover is the
Galois group of this field extension.

It follows from above discussion that the Galois group $G$ acts 
freely on $X\setminus R_f$ with the quotient $Y\setminus D$ 
and one has the exact sequence $0 \rightarrow \pi_1(X\setminus R_f,p') \rightarrow
\pi_1(Y\setminus D,p) \rightarrow G \rightarrow 0$. 
Vice versa, given an unramified cover $f: X\setminus R_f \rightarrow
Y\setminus D$, it follows from Riemann Extension Theorem
for normal spaces (cf. \cite{grauertrem} Ch.7, $\S 4$, sect.2) that
this action on $X\setminus R_f$ extends to the $G$-action on $X$
 via biregular automorphisms. 

For an irreducible component $\Delta \subset R_f$ of the ramification divisor, the
subgroup $I(\Delta)$ of  $G$ of automorphisms which fixes all $x \in \Delta$ is called the
decomposition group\index{decomposition group} of $\Delta$ or inertia group of $\Delta$
(cf. \cite{SGA1}, Expose V, sect 2) \footnote{Recall that the ground field here is $\CC$.  For varieties over non-algebraically closed fields, the inertia
  group $I(x)$ of $x \in \Delta$ (which is the subgroup of the decomposition
  group consisting of automorphisms inducing trivial automorphism of
  the extension of the residue fields of $f(x)$ by the residue field
  of $x$ (cf.\cite{SGA1} Expose V, sect. 2) is a proper subgroup of
  the decomposition group.}. Action of inertia group on the tangent
space at a smooth point $x\in \Delta$ which it fixes, induces the action on the 
normal space of $\Delta$. 
The character $\psi_I$ of this 1-dimensional representation of the
cyclic group $I(\Delta)$
generates the group of characters $Char(I(\Delta))$. In particular one has a well defined map
$Char I(\Delta)
\rightarrow \ZZ: \chi \rightarrow i_{\chi}$ where
$\chi=\psi^{i_{\chi}},  0 \le i_{\chi} <{\rm ord} I(x)$.

The extension is called abelian
(resp. cyclic) if it is Galois and the Galois group is abelian
(resp. cyclic).  For a branched Galois cover $X \rightarrow Y$, the
ramification index $e_{\Delta'}$ is the same for all irreducible
components $\Delta'\subset X$, having the same image $D\subset Y$.  Moreover, the order
of the inertia group $\vert I(\Delta')
\vert=e_{\Delta'}$. If $r$ is the number of 
 $f$-preimages of a generic point $D \subset Y$ then one has $\vert G
\vert=re_{\Delta'}$.

The above correspondence between subgroups of the fundamental group and
covers in Galois case becomes the correspondence (for fixed pair $(Y,D)$)
between surjections $P: \pi_1(Y\setminus D) \rightarrow G$ and covers
with Galois group $G$. Given a surjection, one can construct the
corresponding cover, i.e. unique, up to homeomorphism over $Y\setminus D$,  topological space $Y'$ and the map $f: Y'\rightarrow
Y \setminus D$ making $Y'$ into unramified covering space with
group $G$, as follows.  
This space $Y'$ can be viewed as the quotient of the space of paths in $Y-D$
with a fixed initial points $p \in Y \setminus D$ with two paths
$\gamma_1,\gamma_2$ being equivalent iff they have the same end point
and the homotopy class of $\gamma^{-1}_1 \circ \gamma_2 \in
\pi_1(Y\setminus D,p)$ has trivial image in $G$. $Y'$, being an 
unramified cover, inherits from $Y\setminus D$ the structure of complex
manifold so that $f'$ is \'etale. We have the following theorem:

\begin{theorem} (Grauert-Remmert, cf. \cite{grauertrem} Chapter 7,
  Grothendieck, \cite{SGA1}, SGA1, Ch XII, sect. 5) Let $f^*:
  X^*\rightarrow Y \setminus D$ be a finite
  unramified map of complex spaces where $Y$ is a smooth complex
  manifold and $D$ a divisor on $Y$. 
 Then there is a unique normal space $X$ containing $X^*$ as a
 dense subset and morphism
 $f: X\rightarrow Y$ such that $f^*=f\vert_{X^*}$. 
\end{theorem}

The inertia group of a component $\Delta' \subset X$ of ramification
divisor in a Galois cover with  a group $G$ 
is the cyclic subgroup of $G$ generated by the image in $G$
of a representative $g$ in the conjugacy class 
$\gamma \in \pi_1(Y\setminus D)$ of a meridian of $D$.  If the order
of this image is $s$, then $g^s$ can be 
 lifted to $X\setminus \Delta'$ as a closed path (i.e. $g^s$ belongs to the
kernel of surjection $\pi_1(Y\setminus D,p) \rightarrow G$) and is
homotopic in the complement to the ramification divisor in $X$ to  
a meridian of $\Delta'$.

\subsection{Abelian Covers}  We discuss two ways to enumerate 
branched covers with abelian Galois group over a manifold with given
branch locus. One is topological, which  
follows immediately from the discussion of previous section (cf. \cite{mecharvar})
and another is algebro-geometric (cf. \cite{pardini},
\cite{alexeevpardini}). A different important perspective, from a view
point of root-stacks is discussed in \cite{simpsonstacks}.

Since by Hurewicz theorem $H_1(Y\setminus D,\ZZ)$ is the abelianization
of $\pi_1(Y\setminus D)$,  any surjection of the fundamental group
onto an abelian group $G$ factors as $\pi_1(Y\setminus D) \rightarrow H_1(Y\setminus D,\ZZ)
\rightarrow G$. Hence we have the following:

\begin{corollary}\label{abcovviarep} Equivalence class of a branched cover
  of a complex manifold $Y$ and having a divisor $D=\bigcup D_i$
  as its branch locus is determined by the surjection $H_1(Y\setminus
  D)\rightarrow G$ taking neither of meridians of $D$ to identity. 
Vice versa, an abelian branched cover $f: X \rightarrow Y$ determines 
the branch divisor $D\subset Y$ and the above surjection of the
homology group. Moreover, this correspondence induces the map from
the set of irreducible
components of the branch locus to the set of cyclic subgroups of the Galois group
(inertia subgroups of the irreducible components of ramification divisor).
\end{corollary}

An algebro-geometric description of abelian covers (cf. \cite{pardini})
is given in terms of the collections of line bundles labeled by the
characters of $G$.  Given such a cover $f: X \rightarrow Y$,
one obtains the decomposition into eigen-sheaves:
\begin{equation}\label{pardinidata}
     f_*(\mathcal{O}_X)=\bigoplus_{\chi \in Char G} \CaL_{\chi}^{-1}
\end{equation}
The left hand side has the structure of a sheaf of algebras and the work
\cite{pardini} describes the data specifying such structure on the right in
(\ref{pardinidata}). This is done in terms of classes of components of branch locus 
in $Pic(Y)$,  the Galois group $G$, the collection of cyclic subgroups
 $H$ and generators
$\psi_H$ of each group $Char H$ satisfying the following compatibility conditions.
Once the $\mathcal{O}_Y$-algebra structure, say $\cA$, on the right in (\ref{pardinidata}) is
specified, the branched cover is just $Spec\cA\rightarrow Y$.

For a pair $(H,\psi)$,
where $H \in Cyc(G)$ and $\psi$ is a generator of $Char H$, let $D_{H,\psi}$ be the union
of irreducible components $D$ of the branch locus which have $H$ as its
inertia group and $\psi$ as the character of representation of $H$ on
the normal space to a component of the ramification locus over $D$
fixed by $H$. To a character $\chi \in Char G$ corresponds 
$r^{\chi}_{H,\psi} \in \NN$ such that $\chi
\vert_H=\psi^{r^{\chi}_{H,\psi}}, 0 \le r^{\chi}_{H,\psi} <{\rm} ord H$.
For a pair $\chi_1,\chi_2 \in Char
G$, let us set
$\epsilon^{H,\psi}_{\chi_1,\chi_2}$ to be $0$ (resp.1) if 
$\chi_1\vert_H=\psi^{i_{\chi_1}}, \chi_2\vert_H=\psi^{i_{\chi_2}}, 0 \le
i_{\chi_1},i_{\chi_2} <Card H$ and $i_{\chi_1}+i_{\chi_2} <Card H$
(resp. $i_{\chi_1}+i_{\chi_2}  \ge Card H$).

Then the bundles $\CaL_{\chi}$ in (\ref{pardinidata}) satisfy the relations (cf.\cite{pardini}):
\begin{equation}\label{relationspardini}
      \CaL_{\chi_1 \chi_2} =\CaL_{\chi_1} \otimes\CaL_{\chi_2}
     \bigotimes_{H \in Cyc(G),\psi \in CharH}\mathcal{O}(D_{H,\psi})^{\epsilon^{H,\psi}_{\chi_1,\chi_2}}
\end{equation}

In fact, if $\chi_1,..,\chi_s$ are generators of a decomposition of
$\Char(G)$ into a direct sum of cyclic subgroups and $d_j$ is the
order of $\chi_j, j=1,...,s$ then:

\begin{equation}\label{bundlespardini}
  d_{\chi}{\CaL}_{\chi}=\sum_{H,\psi}{\frac{d_{\chi}r^{\chi}_{H,\psi}}{\vert H \vert}}D_{H,\psi}
\end{equation}

Vice versa (cf. \cite{pardini}), given

(a) a finite abelian group $G$, 

(b) a smooth compact complex
manifold $Y$,  

(c) a divisor $D$ on $Y$ with assignment to each irreducible component a
cyclic subgroup $H$ of $G$ and a 
generator 
$\psi_H$ of $\Char(H)$

(d) collection of line bundles $\CaL_{\chi}, \chi \in \Char(G)$ labeled
by the characters of $G$ 

\noindent with (a),(b),(c),(d)  satisfying the relations (\ref{relationspardini}),
there is abelian branched cover $X$ of $Y$ satisfying
(\ref{pardinidata}) (cf. \cite{pardini}. The data (a),(b),(c),(d) subject to
(\ref{relationspardini}) called {\it the building data}.\index{building data}
\smallskip

We will show how to recover from $Y,D$ and surjection
$H_1(Y\setminus D,\ZZ) \rightarrow G$ the parts (c),(d) of the building
data and vice versa, the building data determines the surjection onto
the covering group.

\begin{proposition}\label{datarelation}  Let $Y$ be a smooth projective
  manifold and  let $D=\bigcup_{i=1}^r D_i$ be a divisor with irreducible components $D_i$. 
The surjection $\pi: H_1(Y \setminus D) \rightarrow G$ onto an
abelian group $G$ determines for each character $\chi \in
\Char G$ the bundle $\CaL_{\chi}$ so that the bundles
$\CaL_{\chi},\chi \in \Char G$ satisfy the relations (\ref{relationspardini}). Moreover,
the bundles $\CaL_{\chi}^{-1}$ are the eigenbundles of decomposition (\ref{pardinidata})
for the covering corresponding to $\pi$ (cf. Cor. \ref{abcovviarep}). Vice versa, a building data determines the surjection $H_1(Y \setminus
D,\ZZ)\rightarrow G$.
\end{proposition}

\begin{proof} To a unitary character $\chi \in H^1(Y\setminus D,U(1))$ 
one associates the element in $Pic(X)$ as follows.
One has the following high dimensional version of the exact sequence
(\ref{homologysequence}):
\begin{equation}\label{homseqhighdim}
   H_2(Y,\ZZ) \rightarrow H^{2\dim D}(D,\ZZ) \rightarrow
     H_1(Y\setminus D,\ZZ) \rightarrow H_1(Y,\ZZ)=0
\end{equation}
Applying $Hom(\cdot,\KK), \KK=\ZZ, \RR,U(1)$ to the terms of  
(\ref{homseqhighdim}) we obtain:
\begin{equation}\label{dualhomologysequence}
\begin{matrix}0 & \rightarrow & Hom(H_1(Y\setminus D,\ZZ),\ZZ) & \rightarrow & 
H_{2\dim D}(D,\ZZ) & \buildrel \iota_{\ZZ} \over \rightarrow & Pic(Y)  \subset
H^{2}(Y,\ZZ) \cr
 &  & \downarrow & & 
\downarrow &  & \downarrow \cr
0 & \rightarrow & H^1(Y\setminus D,\RR) & \rightarrow & 
H_{2 \dim D}(D,\RR) & \buildrel {\iota_{\RR}} \over \rightarrow & Pic(Y) \otimes \RR \subset
H^{2}(Y,\RR) \cr
 &  & \downarrow & & 
\downarrow \exp &  & \downarrow \exp \cr
0 & \rightarrow & H^1(Y\setminus D,U(1)) & \buildrel \varrho \over \rightarrow & 
H_{2\dim  D}(D,U(1)) & \buildrel \iota_{U(1)} \over \rightarrow & H^2(Y,\ZZ) \otimes U(1) 
\end{matrix}
\end{equation}
Here $\varrho$ is evaluation of a character on the meridian of the
irreducible component  \footnote{recall that this follows from identification 
$H_{2 \dim D}(D,U(1))=H^{2 \dim D}(Y,Y \setminus D,U(1))$ obtained by excision and Lefschetz duality}, the vertical arrows are
induced by the exponentiation $exp: \RR \rightarrow U(1),
a \rightarrow e^{2 \pi  i a}$ and the map $\iota$ (for each choice of
coefficients) assigns to a
homology class, the class in $H^{2\dim D}(Y)$ which corresponds to the
linear function on $H_{2 \dim D}(X)$ given  by the intersection
index with this class. 
A lift  $exp^{-1}(\varrho)$ of $\chi \in H^1(Y\setminus D,U(1))$
determines uniquely the element $\tilde \chi$ in the
unit cube in $H_{2 \dim D}(D,\RR)$, which is a fundamental domain for
the action of the group $H_{2 \dim D}(D,\ZZ)$ on the latter and which has 
$H_{2 \dim D}(D,U(1))$ as the quotient. Since  $\iota_{U(1)}$ takes $\exp (\tilde \chi)$  
to the trivial  class in $H^2(Y,\ZZ)\otimes U(1)$
we obtain that $\iota_{\RR}(\tilde \chi) \in H^2(Y,\RR)$ 
is an integral class.
Since it has the Hodge type $(1,1)$, this class defines a line bundle.
We shall denote it as $\CaL_{\chi}$.

Let $\chi(\gamma_{D_i})=exp(2 \pi i \alpha_i), \alpha_i \in \QQ, 0 \le
\alpha_i <1$. If $ord \chi=d$ then
$\alpha_i=\frac{\nu_i}{d}, 0 \le \nu_i <d, i\in \NN$. It follows that 
$\iota_{\RR}(exp^{-1}(\varrho(\chi)))=\sum {\nu_i \over d}[D_i]$ defines an integral class in $H^2(Y,\ZZ)$
and $\CaL_{\chi}$ is the bundle with the first Chern class
corresponding to this integral class \footnote{recall that $Y$ is
  simply connected and hence $\CaL_{\chi}$ is well defined.} .  
More directly, integrality can be seen as follows: since
$\chi(D_{\gamma})=1, \forall \gamma \in H_2(Y,\ZZ)$, it follows from (\ref{homologycomplement})
that one has $\prod_i exp(2\pi i
\alpha_i)^{(\gamma,D_i)}=1$. Hence $(\gamma, \sum_i {\nu_i \over
  d}[D_i]) \in \ZZ$ for all $\gamma \in H_2(Y,\ZZ)$ i.e. $\sum_i {\nu_i \over d}D_i$ is 
an integral class. Let $\CaL_{\chi}$ be the bundle with the first Chern class
corresponding to this class.  
The bundle $\CaL_{\chi}^{d}=\sum {{\nu_i}}\mathcal{O}(D_i)$ has a section and it
  follows from the calculation in \cite{esnaultviehweg} Section 3.6 (cf. also
  \cite{lazarsfeld} Remark 4.1.7 and Remark \ref{cycliccover} below)  that for the cyclic cover
  $\pi_{\chi}: Y_{\chi} \rightarrow Y$ corresponding to $Ker(\chi)$
  one has $\pi_{\chi}(\mathcal{O})=\sum_0^{d-1}
  \CaL_{\chi}^{-k}$. Moreover,  $\CaL_{\chi}^{-k}$ is the
  eigensheaf with corresponding character $\chi^k$. 
  Considering the full $G$-cover and factoring it through $Y_{\chi}$, one
  sees that this is also the eigenbundle in the $G$-cover corresponding to
  $\pi$.  
 Now the relations (\ref{relationspardini}) follows from
 \cite{pardini}, Theorem 2.1.

Vice versa, the map $\Char G=Hom(G,U(1))\rightarrow H_{2\dim
  D}(D,U(1))$ sending a component with inertia group $H$ to $exp(2 \pi
i {{r_{H,\psi}^{\chi}} \over {\ord H}})$ due to relations
(\ref{relationspardini}) lifts to the map to $H^1(Y\setminus D,U(1))$
and hence by duality induces the surjection $H_1(Y\setminus D,\ZZ)
\rightarrow G$.
\end{proof}


\begin{example} (cf. \cite{lazarsfeld}, Sect. 4.1.B or \cite{esnaultviehweg})\label{cycliccover} Let $Y$ be a smooth projective variety,  $D\subset Y$  be
a very ample divisor.  Let $\CaL$ be a very ample line bundle such that
$\CaL^d=\mathcal{O}(D)$. Clearly, the bundles $\CaL^i$ form a part of a building
data for $G=\ZZ_d$. The correspnding cover can be obtained as follows.
Let $s \in H^0(Y,\mathcal{O}(D))$ be a section with
zero-scheme $D$ 
and $\nu_d: [\CaL] \rightarrow [\mathcal{O}(D)]$ the map of the total
spaces of the line bundles given by $v \in \CaL \rightarrow v^{\otimes
  d}\in \CaL^d$. Then $\nu_d^{-1}(s(Y))$ is a smooth subvariety $Y_d$ of the
total space of the line bundle $\CaL$ and its projection $\pi$ onto the base endows $Y_d$ with the structure of the branched cover over
$Y$ with branch locus $D$. The divisibility of the fundamental class of
$D$ by $d$, implies that if $H_1(Y,\ZZ)=0$, then there is well defined surjection
$H_1(Y\setminus D,\ZZ) \rightarrow \ZZ_d$. It assignes to a
1-cycle $\delta$ representing a class in $H_1(Y\setminus D,\ZZ)$,
the modulo $d$ intersection index  of a 2-chain in
$Y$ having $\delta$ as its boundary. So $Y_d$ is the cyclic cover of
$Y$ branched over $D$ with Galois group $\ZZ_d$ 
and corresponding to this surjection of $H_1(Y\setminus
D,\ZZ)$. The inertia group of any point of $D$ is $\ZZ_d$. On the other hand 
$\pi_*(\mathcal{O}_{Y_d})=\oplus_{i=0}^{i={d-1}}\CaL^{-i}$ and $\CaL^{-i}$ is
the eigen-bundle corresponding to the character of $\ZZ_d$ given by
$\chi_i: j \rightarrow exp({{2 \pi \sqrt{-1} ij} \over d})$. The relation 
(\ref{relationspardini}) is immediate.

Vice versa, given the surjection $H_1(Y\setminus D,\ZZ) \rightarrow
\ZZ_d$, the diagram (\ref{dualhomologysequence}) shows that 
 the character $\chi_i$, taking value $exp({2 \pi
\sqrt{-1}i \over d})$ on generator of $\ZZ/d\ZZ$, has as the lift
$i_{\RR}(exp^{-1}(\varrho(\chi_i))=c_1(\CaL^i)$ and in this way producing a building data.

In the case $Y=\PP^2$, $D$ is an irreducible curve of degree $d$ with
equation $f(x_0,x_1,x_2)=0$, one
has $H_1(\PP^2\setminus D,\ZZ)=\ZZ_d$ and the cover corresponding to 
this isomorphism is biholomorphic to a hypersurface $V_f: u^d=f(x_0,x_1,x_2)$
in $\PP^3$. Moreover the decomposition (\ref{pardinidata}) becomes
$f_*(\mathcal{O}_{V_f})=\oplus_{i=0}^{d-1} \mathcal{O}_{\PP^2}(-i)$.
\end{example}

\begin{example} (cf. \cite{ishida}, \cite{hirzebruch}) Let $\cA$ be an arrangement of $r$
  lines in $\PP^2$. Then $H_1(\PP^2\setminus
  \cA,\ZZ)=\ZZ^{r}/\{(1,....,1)\}=\ZZ^{r-1}$.  Let $H_1(\PP^2\setminus
  \cA,\ZZ)\rightarrow G=\ZZ_n^r/\{(1,...,1)\}=\ZZ_n^{r-1}$ sending the meridian the $i$-th line to
  $(0,...,0,1,...0) \mod n$. A character of $G$ can be identify
  with a vector $({a_1 \over n},...,{a_r \over n}), 0 \le a_i <n, \sum_1^r
  {a_i \over n} \in \ZZ$. Let us denote this character
  $\chi_{a_1,...,a_r}$. The inertia group $H_i$ of the $i$-th line is the
  subgroup of $G$ isomorphic to $\ZZ_n$ and generated by $(0,...,0,1,...0) \mod n$ (all
  components except the $i$-th are zero) and the character $\psi$ of
  $H_i$ takes the value $exp{{2 \pi i}\over n}$  on the corresponding
  generator. It follows from discussion of Proposition \ref{datarelation}
 that $$\CaL_{\chi_{a_1,...,a_r}}^{-1}=\mathcal{O}_{\PP^2}\left(-({{\sum_1^r a_i}\over
   n})\right)$$ See \cite{ishida} for a direct calculation of the direct image
 of the structure sheaf using that this abelian cover is the
 restriction of the Kummer cover: $\PP^{r-1}\rightarrow \PP^{r-1}$
 given by $(x_1,...,x_r) \rightarrow (x_1^n,...,x_r^n)$.
\end{example}

\begin{example} Let $D$ be the hypersurface in $\CC^n$ given by $f_1(x_1,...,x_n) \cdot...\cdot
  f_r(x_1,....,x_n)=0$ where $f_i\in \CC[x_1,..,x_n]$ are
  irreducible. Using a non-compact version of the calculation (\ref{homologycomplement}) one
  obtains $H_1(\CC^n\setminus D)=\ZZ^r$. Let 
 $p:H_1(\CC^n\setminus D) \rightarrow G$ be a surjection onto an abelian
 group.
Then to $p$ corresponds the cover $P: V_{p,\bar D} \rightarrow
\PP^n$ branched over the projective closure $\bar D$ of $D$ and
possibly over the hyperplane at infinity
with the following properties.
 The order $r_i$ of $p(\gamma_i)\in G$ coincides
 with the ramification index of the branched cover  $P$ at $P^{-1}(s)$
 where $s$ is a generic point in $D_i$. 
 At a generic point 
 $s \in \PP^{n-1}$ the ramification index at $P^{-1}(s)$ is the order
 in $G$ of the class $p(\sum (deg f_i)\gamma_i)\in H_1(\CC^n\setminus D,\ZZ)$. 
 An explicit model of such covering can be obtained as the
 normalization of the projective closure of affine complete
 intersection in $\CC^{n+r}$ given by equations:
\begin{equation}
   z_i^{r_i}=f_i(x_1,...,x_n) \ \ \ \ i=1,...,r.
\end{equation}
\end{example}

\subsection{Characteristic varieties}\label{charvarsection}
\subsubsection{Alexander invariants and jumping loci of local systems}
\label{alexandeinvariantsjump}
A multivariable generalization of Alexander polynomials was proposed
in \cite{mecharvar} as follows. Let $Y$ be a finite CW complex and $\phi:
\pi_1(Y) \rightarrow A$ be a surjection onto a finitely generated
abelian group. The unbranched abelian cover $\pi_{\phi}: Y_{\phi} \rightarrow Y$ corresponding to
$\phi$ comes with a free action of $A$ via cellular maps. Hence the
compact supported homology 
 $H_k(Y_{\phi},\CC)$ and also its exterior powers 
\begin{equation}\Lambda^iH_k(Y_{\phi},\CC) 
  \end{equation}
can be considered as the modules over the group algebra $\CC[A]$ of $A$. 
\begin{definition} \label{defdepth}(cf. \cite{homologyabcov}) The affine subvariety
  $Char_i^k(Y,\phi)$ of the torus $Spec \CC[A]$ defined as support of
  the module $\Lambda^iH_k(Y_{\phi},\CC)$ (cf.\cite{eisenbud-ca}) is called {\it the depth $i$ 
characteristic variety\index{characteristic!variety} of $Y$ in dimension $k$ (corresponding to
surjection $\phi$)}\footnote{the support is assumed to be a reduced variety}
\end{definition}

Standard results from commutative algebra (cf. 
\cite{buchsbaum} or \cite{eisenbud-ca} Ch. 20) show
that $Char_i^k(Y)$ is the zero set of the $i$-th Fitting ideal of
$\CC[A]$-module $H_k(Y_{\phi},\CC)$ i.e. the ideal generated by
 $(n-i+1)\times (n-i+1)$ minors of the matrix of a presentation of this module via
$n$ generators and $m$ relations. Moreover, for $k=1$, which unless otherwise
stated will be our focus for the rest of this section, presentation of
the module $H_1(Y_{\phi},\CC)$ can be studied using the matrix of
Fox derivatives giving presentation of $\CC[A]$-module $H_1(Y_{\phi},\tilde
p,\CC)$ where $\tilde p=\pi_{\phi}^{-1}(p)$ is the preimage of $p\in Y$ (cf. \cite{homologyabcov}). As a
consequence, this implies that for a CW complex
having as its fundamental group a group with deficiency $1$ 
and the homomorphism $\phi$ being the abelianization, the
characteristic variety $Char_1^1(Y)$ of depth 1, has codimension 1 in $Spec \CC[A]$. For example this is
the case for $Y=S^3\setminus L$ where $L$ is a link. In fact,
$Char_1^1(S^3\setminus L)$ is the zero set of 
the multivariable Alexander polynomial of $L$ (\cite{knottheory}).  
In the case of algebraic curves in $\CC^2$, the codimension of
$Char^1_1$ is typically larger than 1 except for the case $A=\ZZ$ in which case 
it is the zero set of the 1-variable Alexander polynomial discussed in section
\ref{alexandercurves} \footnote{However $Char_1(\pi_1(\PP^1\setminus
  [3]),ab)=(\CC^*)^2$ where $[3]$ is a subset containing 3 points and
  $ab$ is the abelianization of the free group}

There is a different interpretation of these subvarieties of  the complex
tori $Spec \CC[A]$ \footnote{$Spec \CC[A]$ is algebraic group with
  $Card Tor A$ connected components with $(\CC^*)^{rk A}$ being the component of identity}. 
Recall (cf. for example \cite{daviskirk}, Ch.5) that a rank $l$ local system\index{local system} on a CW
complex $Y$ is a $l$-dimensional linear 
representation of the fundamental group $\rho: \pi_1(Y)\rightarrow GL(l,\CC)$.
(Co)homology of a local system are obtained as the cohomology of the chain complex:
\begin{equation}\label{chaintwisted}
   ... \rightarrow C_i(\tilde Y,\CC)\otimes_{\pi_1(Y)} \CC^l \rightarrow ...
\end{equation}
where $C_i(Y,\CC)$ are the chains with compact support on the universal cover
considered as a module over the group ring of the fundamental group.
 
An important feature of local systems is the following:
 in the case when
$Y$ is a smooth quasi-projective variety, the cohomology of local systems
admit Hodge-deRham
description given by Deligne (cf. \cite{delignediffeq}).  
First of all representations of fundamental groups can in interpreted
as locally trivial vector bundles with constant transition functions 
(cf. \cite{delignediffeq}, Cor.1.4), which in turn, in the case when
$Y$ is a smooth manifold, can be interpreted 
as flat (integrable) connections $\nabla: \cV \rightarrow
\Omega^1_Y\otimes \cV$ on a holomorphic vector bundle $\cV$ (i.e. a
$\CC$-linear map satisfying Leibnitz rule). This
differential operator can be extended to higher degree forms and
lead to a twisted deRham complex: 
\begin{equation}\label{derhamconnection}
    ...\rightarrow \Omega^p(Y) \otimes \cV \buildrel \nabla \over \rightarrow
    \Omega^{p+1}\otimes \cV \rightarrow ... \ \ \ 
\end{equation}
The (co)homology of the latter are identified with the (co)homology
$H^*(Y,\rho)$ of the complex (\ref{chaintwisted}) since both are derived functors of 
the functor sending a representation to the subspace of invariants 
(cf. \cite{delignediffeq}, Prop. 2.27, 2.28). 

Using (co)homology with twisted coefficients of rank one local systems, one can define
{\it jumping loci}:\index{jumping loci}

\begin{equation}\label{jumpingloci}
   \cV^k_i(Y)=\{\rho \in Hom(\pi_1(Y),\CC) \vert H_k(Y,\rho) \ge i\}
\end{equation} 

There is the canonical identification of $Hom(H_1(Y,\ZZ),\CC^*)$ and 
$Spec \CC[H_1(Y,\ZZ)]$ making two collections of subvarieties of both tori
correspond to each other \footnote{the order of vanishing of the Fitting
  ideal of $H_1(Y_{\phi},\CC)$ at $(1,....,1)$ in general is different than $rk
  H_1(Y,\CC)$ which is the first Betti number of trivial local system}: 
\begin{equation}
    Char^1_i(Y) \setminus \{1\}=\cV_i^1(Y)\setminus \{1\}
\end{equation}
This was shown to be the case for any finite CW complex $Y$ in \cite{ekojump} for 
$k=1$ and arbitrary $i$ (cf. also \cite{mecharvar}, \cite{dimca07})
i.e. when one is interested in invariants of $\pi_1(Y)$ and for
$i=1$ but with arbitrary $k$ (cf. \cite{papadimasuciu}) when one considers
invariants of the homotopy type.

\subsubsection{Homology of abelian covers}\label{homologyabcov} Characteristic varieties
determine the homology of covering spaces as follows.

\begin{proposition}\label{homologyunbranched}
 (cf. \cite{homologyabcov}, \cite{hironaka}, \cite{denhamsuciu})
 Let $\phi: \pi_1(Y) \rightarrow A$ be a surjection onto
  a {\it finite} abelian group $A$ and
  let $Y_{\phi}$ be the corresponding unbranched cover of $Y$ with the
  Galois group $A$. Let $\nu(A)$ be the image of embedding $\phi^*: Spec
  \CC[A] \rightarrow Spec\CC[H_1(Y,\ZZ)]$. 
Then:
 \begin{equation}
    rk H_k(Y_{\phi},\CC)=\sum_iCard (\nu(A) \cap \cV^k_i(Y))
\end{equation}
\end{proposition}

This follows from the definition of jumping loci (\ref{jumpingloci}), 
the extension of classical Shapiro lemma   (cf. \cite{denhamsuciu})
from group
cohomology to arbitrary spaces (cf. \cite{daviskirk}) i.e. in our notations
the identification $H_k(Y_{\phi},\CC)=H_k(Y,\CC[A])$, and the
decomposition $\CC[A]=\oplus_{\chi \in Spec \CC[A]} \CC_{\chi}$
where $\CC_{\chi}$ is the 1-dimensional representation of $A$ 
given by the character $\chi$.

Now consider the case of abelian branched covers which according to 
Cor. \ref{abcovviarep} are specified by the branching locus and
the abelian quotient of the fundamental group of the complement to
the latter. The proof below is a version of the argument
due to M.Sakuma (cf. \cite{sakuma}).

\begin{proposition}\label{homologysakuma}  Let $X$ be a smooth simply-connected projective surface and 
$D=\bigcup D_i$ a reduced divisor. Let $\phi: H_1(X\setminus D,\ZZ)
\rightarrow A$ be a surjection onto a finite abelian group. For a character $\chi \in A^*$ where
$A^*=Hom(A,\CC^*)$, let
$D^{\chi}$ be the union of irreducible components $D_i$ of $D$ 
such that for the meridian $\delta_i \in H_1(X\setminus D,\ZZ)$ of
$D_i$ one has $\chi(\delta_i)\ne 1$.
Denote by $d(D^{\chi},\chi)$ \footnote{The integer $d(D^{\chi},\chi)$
  is called {\it the depth} of the character\index{depth of the character} $\chi$ of the curve
  $D^{\chi}$} 
 the  maximum of the integers $i$ such that $\chi \in
Char^1_i(X\setminus D^{\chi})$, where $Char^1_i(X\setminus D^{\chi})$
is the depth $i$ 
characteristic variety of the curve $D^{\chi}$ as defined in \ref{defdepth}.  Let $\widetilde {\bar {X_{\phi}}}$ be a resolution of
singularities of the branched cover $\bar X_{\phi}$ of $X$ ramified
along $D$ corresponding to above surjection $\phi$. 
Then 
\begin{equation}\label{sakumaequation}
    rk H_1(\widetilde {\bar {X_{\phi}}},\CC)=\sum_{\chi \in A^*} d(D^{\chi},\chi)
\end{equation}
\end{proposition}
\begin{proof} Denote by $Sing \bar X_{\phi}$ the set of singularities
  of $\bar X_{\phi}$. This is a finite set mapped by the covering map
  into the set $Sing(D)$ of singularities
  of $D$. 
We will start by showing that 
\begin{equation}\label{mumfordcons}
    rk H_1(\widetilde {\bar {X_{\phi}}},\CC)=rk H_1( {\bar X_{\phi}}
    \setminus Sing(\bar X_{\phi}),\CC)
\end{equation}
 Indeed, if $E$ is the exceptional set of a resolution 
 $\widetilde {\bar {X_{\phi}}} \rightarrow  {\bar X_{\phi}}$, then 
\begin{equation}
\widetilde {\bar {X_{\phi}}}\setminus E= {\bar X_{\phi}}\setminus
Sing(\bar X_{\phi})
\end{equation}
On the other hand, the exact sequence of the pair $(\widetilde{\bar
  X_{\phi}}, \widetilde{\bar {X_{\phi}}}\setminus E)$ and  the identification
$H^i(\widetilde{\bar {X_{\phi}}},\widetilde{\bar {X_{\phi}}}\setminus
E)=H_{4-i}(E)$ yield:
\begin{equation}\label{mumford1}
 0 \rightarrow  H^1(\widetilde{\bar {X_{\phi}}}) \rightarrow H^1(\widetilde{\bar
    {X_{\phi}}}\setminus E) \rightarrow H_2(E) \rightarrow H^2(\widetilde{\bar {X_{\phi}}})
\end{equation}
Together with injectivity of the right map in (\ref{mumford1}), which is a consequence of
Mumford theorem on non-degeneracy of the intersection form on a
resolution of a surface singularity (cf. \cite{mumford}), we obtain (\ref{mumfordcons}).

By universal coefficients theorem, allowing to switch to cohomology, the identity (\ref{sakumaequation})
will follow from the following calculation of dimensions of $\chi$-eigenspaces 
\begin{equation}\label{sakumaequation1}dim H^1(\bar X_{\phi}\setminus Sing(\bar
X_{\phi}))_{\chi}=d(D^{\chi},\chi)
\end{equation}
for all characters $\chi \in Char(A)$.

To show (\ref{sakumaequation1}), note that the group $A$ acts on $\bar X_{\phi}\setminus Sing(\bar
X_{\phi})$ with the quotient $X\setminus Sing(D)$. For any
character $\chi$ of group $A$ we consider the cyclic branched cover
$(X\setminus Sing(D))_{\chi}$ of 
$X\setminus Sing(D)$ corresponding to composition $\pi_1(X\setminus
D)\rightarrow A \rightarrow Im(\chi)$. One has the biregular isomorphism:
\begin{equation}\label{kernelchi}
       \bar X_{\phi}\setminus Sing(\bar X_{\phi})/Ker \chi=(X\setminus Sing(D))_{\chi}
\end{equation} 
The group $A/Ker(\chi)=Im(\chi)$ acts on the left side of
(\ref{kernelchi}) and the 
identification (\ref{kernelchi}) is $Im(\chi)$-equivariant. The transfer $H^*(\bar
X_{\phi} \setminus  Sing(\bar
X_{\phi})) \rightarrow H^*(\bar X_{\phi}\setminus Sing(\bar
X_{\phi})/Ker)$
(cf. \cite{bredon}, p.118) provides $Im(\chi)$-equivariant
isomorphism $H^*(\bar X_{\phi}\setminus Sing(\bar
X_{\phi}))^{Ker(\chi)}=
H^*(\bar X_{\phi}\setminus Sing(\bar X_{\phi})/Ker(\chi))$
which implies that 
\begin{equation}\label{sakumastep}
H^*(\bar X_{\phi}\setminus Sing(\bar
X_{\phi})_{\chi}=H^*(\bar X_{\phi}\setminus Sing(\bar
X_{\phi})^{Ker(\chi)}_{\chi}=H^*((X\setminus Sing(D))_{\chi})_{\chi}
\end{equation}
(\ref{sakumaequation1}) is obvious for trivial character and the isomorphism (\ref{sakumastep}) shows that 
(\ref{sakumaequation}) follows from the cyclic case of the Proposition for non-trivial $\chi$
\footnote{we also use that removal 0-dimensional set $Sing D$ from a 4
  -dimensional manifold does not change the first Betti number}.

Finally, the cyclic cover $(X\setminus Sing(D))_{\chi}$ is a totally ramified 
cover of $X\setminus Sing D$ branched over $D^{\chi}$ and the cyclic case with
non-trivial $\chi$ follows from the calculation of homology of unbranched covers in Proposition
\ref{homologyunbranched} since the action of $Im(\chi)$ on kernel and
cokernel of the map induced by the embedding of the cyclic unbranched
cover $(X\setminus D^{\chi})_{\chi}$
with Galois group $Im(\chi)$ 
\begin{equation}
  H^1((X-Sing D)_{\chi}) \rightarrow H^1((X\setminus
D^{\chi})_{\chi}) 
\end{equation}
is trivial.
\end{proof}
\subsubsection{Structure of characteristic varieties}

The central result on the structure of characteristic varieties of
quasi-projective manifolds is obtained from their interpretation
(cf. \cite{ekojump}) as the jumping loci of the cohomology of local
systems, which allows to apply deep Hodge theoretical methods \cite{arapura}. It asserts that the irreducible components of characteristic varieties
are a finite order cosets of subtori of $Spec \CC[H_1(X\setminus D,\ZZ)]$
and that these components are the pull backs of the characteristic varieties of
fundamental groups of curves via holomorphic maps. The origins of such 
correspondence are going back to deFranchis (\cite{zariskialgsurf}),
Beauville (\cite{beauville}), Green-Lazarsfeld
(\cite{greenlaz}), Simpson (\cite{simpson}) in projective case with 
quasi-projective case being addressed by
Arapura (\cite{arapura}).  In a very special
case when the quasi-projective manifold is a complement to an irreducible plane
singular curve the assertion of the finitness of the order of cosets
becomes the cyclotomic property of the roots of 
Alexander polynomials (\cite{mealex}, cf.Theorem
\ref{theoremdivisibility}) and does not require Hodge theory (unlike
the Theorem \ref{translated}).
We shall quote an orbifold version (cf. \cite{artalorbifolds}) of this correspondence between the
holomorphic maps and the components of characteristic varieties  

\begin{theorem}\label{translated} (cf.\cite{arapura},\cite{artalorbifolds}). Let
  $\cV^1_i(X\setminus D)^{irr}$ be an irreducible component of jumping locus
  of 1-dimensional cohomology of a smooth quasi-projective variety
  $X\setminus D$. Then $\cV^1_i(X\setminus D)^{irr}$ is a coset of
  finite order of a subtorus of the commutative algebraic group $Spec\CC[H_1(X\setminus D,\ZZ)]$.
 Moreover, there exist an orbifold curve $C^{orb}$, an
  irreducible component $\cV_i(\pi^{orb}_1(C^{orb}))^{irr}$ and
  holomorphic orbifold map $f: X\setminus D \rightarrow C^{orb}$ 
  such $\cV^1_i(X\setminus D)^{irr}=f^*(\cV_i(\pi^{orb}_1(C^{orb}))^{irr})$.
\end{theorem}

\begin{corollary} The array of Betti numbers of finite abelian covers of $X$
  branched over $D$ determines the characteristic varieties of the
  fundamental group of the complement.
\end{corollary}

\begin{proof} Translated subgroups are specified by the points of
  finite order on the torus which they contain. A point $\chi$ of the
  finite order belongs to the $i$-th characteristic variety if and
  only if the multiplicity of $\chi$ in the cyclic cover corresponding
  to the group $Im(\chi)$ is at least $i$. The claim follows.
 \end{proof}

A very important application of translated subgroup property is that
it provides a necessary conditions on a group be quasi-projective
i.e. to be a fundamental group of a 
smooth quasi-projective variety. For application of these and 
ideas from different ideas, not discussed here, to the problem of 
characterisation of quias-projective and quasi-Kahler groups
(in particular the comparison with the fundamental groups of 3-manifolds)
 see: \cite{meabhyankar},
\cite{dimcasuciu3manifolds}, \cite{dimcapapadima},
\cite{artalorbifolds}, \cite{aranori}, \cite{kots}, \cite{friedl},\cite{biswas}

\subsection{Isolated non-normal crossings\index{isolated non-normal crossings}}\label{isolatednonnormalsection} A generalization of results
on Alexander invariants from  sections
\ref{alexandercurves}- \ref{charvarsection} providing invariants of the homotopy
type beyond fundamental groups was proposed in \cite{meample}.
The starting point is the following:
\begin{theorem}\label{amplenc} (cf. \cite{meample}, Th. 2.1) Let $X, dim X> 2$ be a smooth simply connected projective
  variety and let $D$ be a divisor such that all its
  irreducible components are smooth and ample. Then $\pi_1(X\setminus
  D)$ is abelian and $\pi_i(X\setminus D)=0$ for $2 \le i \le dim X-1$. 
\end{theorem}
This theorem and Lefschetz hyperplane sections theorem have the
following as an immediate corollary:
\begin{corollary}\label{cormeample}
(\cite{meample}) Let $X$ be as in Theorem \ref{amplenc}, let $D=\bigcup_1^r D_i$ be a
  divisor with ample irreducible components and let $NNC(D)$ be the subvariety of $X$
  consisting of points $x \in X$ at which $D$ fails to be a normal crossings
  divisor \footnote{We call $D$ a
    divisor with isolated non-normal crossings, if $\dim NNC(D)=0$.}
.  Let $s=dim NNC(D)$ \footnote{the convention is that if $x
    \notin D$ then $D$ does have normal crossing at $x$ and the
    dimension of empty set is $-1$}. Then $\pi_i(X\setminus D)=0$ for $i
 \le d-s-2$. Moreover, if $H=\bigcap_1^s H_i$ is a sufficiently general
 intersection of very ample
 divisors on $X$ then $D \cap H$ is a divisor on $X\cap H$ with
 isolated non normal crossings and
\begin{equation}
   \pi_i(X\setminus D)=\pi_i((X\setminus D) \cap H) \ \ \ i \le d-s-1
\end{equation}  
In particular the first non-trivial homotopy group of the complement
to a divisor with ample components can be calculated using the divisor
$D\cap H$
with isolated non-normal crossings on $H$.
\end{corollary}

The high dimensional analogs of the results on the Alexander
invariants of the complement to curves described in \cite{meample} 
give a similar description of the homotopy group $\pi_{d-1}(X\setminus
D)$ where $X$ as in Theorem \ref{amplenc}, $D$ is a divisor with
ample components but now $D$ is allowed to have {\it isolated} non
normal crossings (INNC) i.e.
$dim NNC(D)=0$. The role of the first homology of the infinite abelian cover in the case of
complements to curves is played by the first non-vanishing homotopy group
$\pi_i(X\setminus D), i>1$. In fact one has the identification:
\begin{equation}
    H_{dim X-1}(\widetilde{X\setminus D},\ZZ)=\pi_{dim X-1}(X\setminus D)
\end{equation}
where $\widetilde{X\setminus D}$ is the universal (hence also universal
abelian, cf. theorem \ref{amplenc}) cover of $X\setminus
D$. The $\ZZ[\pi_1(X\setminus D)]$-module structure equivalently can be obtained 
using the Whitehead product (cf. \cite{meample}) and the characteristic
variety of $X\setminus D$ in dimension $dim X-1$ is the support of 
the module $\pi_{dim X-1}(X\setminus D) \otimes \CC$.
 
 In the case when a point in $NNC(D)$ belongs to only one component, failure to be normal
crossing means that the point is just an isolated singularity of $D$. 
In this case the local information about the Alexander invariants is contained in the Milnor 
fiber of the singularities of $D$ and includes the characteristic
polynomial of the monodromy as well as some the Hodge theoretical
invariants (cf. \cite{meample}). 
Reducible analog of isolated singularities are isolated non-normal
crossings i.e. the intersections of hypersurfaces which are smooth and
transversal everywhere except for a single point. It turns out that many features 
of isolated singularities described by  J.Milnor in \cite{milnor} have
counterparts in INNC case. They include the analogs of high connectivity of Milnor fibers, analogs of
monodromy action on the cohomology, its  cyclotomic properties and others.

To be specific, recall that if $f(x_0,...,x_n)=0$ is a germ of an isolated singularity,
$V_f$ is its zero set, $B_{\epsilon}$ is a small ball about the singular point of $f$,
$\partial B_{\epsilon}$ is the boundary sphere then we have the
following:
\footnote{only the last claim in (ii) requires singularity of $f$ to be isolated.}

\begin{theorem}(cf.\cite{milnor})
(i) The complement $B_{\epsilon} \setminus (V_f \cap B_{\epsilon})$ is homotopy equivalent
to $\partial B_{\epsilon}\setminus (V_f\cap \partial B_{\epsilon})$.

(ii) There is a locally trivial fibration $\partial B_{\epsilon} \cap
(V_f \cap \partial B_{\epsilon}) \rightarrow S^1$, with the fiber (i.e. the
Milnor fiber) being homotopy
equivalent to a wedge of spheres of dimension $n$. 
\end{theorem}

\begin{corollary}  If $f(x_0,...,x_n)$ is a germ of isolated singularity, then the
  universal cyclic cover of $B_{\epsilon}\setminus (B_{\epsilon}\cap
  V_f)$ is homotopy equivalent to a finite wedge of spheres of dimension $n$.
Moreover, the action of the deck transformation on the homology of the
universal cyclic cover coincides with the action
of the monodromy operator on the homology of the Milnor fiber. 
In particular, the characteristic polynomial of the monodromy coincides with the
Alexander polynomial of $\partial   B_{\epsilon}\setminus (V_f\cap \partial B_{\epsilon})$.
\end{corollary}
 
A generalization of these results to multi-component germs is as
follows (which is the local counterpart of Corollary \ref{cormeample}):

\begin{theorem}(cf. \cite{isolatednonnormal}) Let $X_r$ be a germ
  $f_1(x_0,...,x_n)\cdot ... \cdot f_r(x_0,...,x_n)=0$ of a
  hypersurface  which is product of $r$ irreducible germs.
Then for $n>1$
\begin{equation}
   \pi_1(\partial B_{\epsilon}\setminus (X_r\cap \partial
   B_{\epsilon}))=\ZZ^r \ \ \ \pi_i (\partial B_{\epsilon}\setminus (X_r\cap \partial
   B_{\epsilon}))=0 \ \ , 2 \le i \le n-1
\end{equation}
\end{theorem}
The role of the Milnor fiber is now played by the universal abelian
\footnote{which is also the universal cover for $n \ge 2$ since the fundamental
group is abelian in this case} cover of the complement $\widetilde{\partial B_{\epsilon}\setminus (X_r\cap \partial
   B_{\epsilon})}$ 
and the monodromy action is replaced by the 
action of fundamental group of the complement to INNC germ on the
universal abelian cover via deck transformations. 
The homology $H_n(\widetilde{\partial B_{\epsilon}\setminus (X_r\cap \partial
   B_{\epsilon})},\CC)$ is equipped with the action of the group algebra $\CC[H_1(\partial B_{\epsilon}\setminus (X_r\cap \partial
   B_{\epsilon}))]$. The $H_1(\partial B_{\epsilon}\setminus (X_r\cap \partial
   B_{\epsilon}),\ZZ)$-action 
can be, as in global case, identified the 
Whitehead product of the elements of $\pi_n(\partial B_{\epsilon}\setminus (X_r\cap \partial
   B_{\epsilon})$ with the elements of $\pi_1(\partial B_{\epsilon}\setminus (X_r\cap \partial
   B_{\epsilon}))=H_1(\partial B_{\epsilon}\setminus (X_r\cap \partial
   B_{\epsilon}))$.  The group ring of the latter is the ring of
   Laurent polynomials and the role of the characteristic polynomial of the
   monodromy is played by the subvariety of the torus $Spec \CC[H_1(\partial B_{\epsilon}\setminus (X_r\cap \partial
   B_{\epsilon}))]$ which is the support of this module (\cite{eisenbud-ca}).

\begin{example}\label{isolatednonnormalexample}
(cf. \cite{isolatednonnormal}) Let $f_{d_i}(x_0,....,x_n)=0,i=1,...,r$ be the equations of 
  smooth sufficiently general hypersurfaces in $\PP^n$. Then the union in $\CC^{n+1}$ of cones over these
  $r$ hypersurfaces is isolated non-normal crossing. The support of
  this module is the zero set of $t_1^d\cdot...\cdot t_r^{d_r}-1=0$.
\end{example}

\begin{example} Let $f_1(x,y)....f_r(x,y)$ be a germ of reducible
  curve in$\CC^2$. The support of the universal abelian cover of the 
complement to the link coincides the zero set of the multivariable Alexander
polynomial of the link. Further properties of this support and its
Hodge theoretical properties are discussed in \cite{pierretteme11}, \cite{pierretteme14}.
\end{example}

We refer to the work \cite{meample} Sections 5 and 6 for description of the structure of
 characteristic varieties of {\it global} INNC in terms of the homotopy
 groups of local INNC (of germs), as well as the relationship between
 the cohomology of local systems, the homology of branched cover and
 the translated subgroups of $Spec \CC[\pi_1(X\setminus D)]$ which
 are the irreducible components of $Supp \pi_{dim X-1}(X\setminus
 D)\otimes \CC$. The divisibility relations extending the divisibility
 theorem from section \ref{divisibility11} to the case of hypersurfaces in $\PP^n$
 with isolated singularities, the Thom Sebastiani theorems for the
 orders of the homotopy groups and other results on the topology of
 the complements to  are discussed in \cite{meannals}. For some  results in 
 non-isolated case, cf. Section \ref{perverse}.

\subsection{Twisted Alexander Invariants\index{Twisted Alexander Invariants}} A generalization
of Alexander polynomials was proposed in the context of  knot theory which uses as
an additional input a linear representation of the fundamental group (cf. \cite{kirkliv}
for a discussion of this generalization). A twisted version of Alexander
polynomials of algebraic plane curves was considered in \cite{cogoflorens}. 
In \cite{me2009} a multivariable extension of this construction called
{\it a characteristic varieties of a CW complex twisted by a unitary
representation} was defined as follows. 

Let $\pi: \pi_1(X) \rightarrow U(V)$
be a unitary representation of the fundamental
group of a CW complex such that $H_1(X,\ZZ)$ is
a free abelian group of a positive rank \footnote{torsion freeness condition of
$H_1(X,\ZZ)$ is introduced to simplify the exposition.}. Here $V$ is a complex vector space endowed with a Hermitian
bilinear form and viewed 
as a left $\CC[\pi_1(X)]$-module.                                                                                                                                                                                                                                                                                                            
Let $\tilde X$  be the universal
cover of $X$. For a (left or right) module $M$ over the algebra $\CC[\pi_1(X)]$
(which is associative but
possibly non-commutative), we denote by $M^{\flat}$
 the module obtained by restriction of the coefficients to 
 the group algebra $\CC[\pi_1'(X)] \subset \CC[\pi_1(X)]$ of the
 commutator subgroup $\pi_1'(X)$ of $\pi_1(X)$. Let $C_*(\tilde X)$ denotes
 chain complex of $\tilde X$ endowed with the natural structure of (a
 right) $\CC[\pi_1(X)]$-module. Consider the following complex of tensor products of $\CC[\pi_1'(X)]$-modules.
\begin{equation}\label{tensorproduct}
     C_*(\tilde X)^{\flat} \otimes_{\CC[\pi_1'(X)]} V^{\flat}: \ \ \
     g(c\otimes v)=cg^{-1}\otimes gv
\end{equation}
The group $\pi_1(X)$ acts on the module (\ref{tensorproduct}) 
and the restriction of this action to the commutator $\pi_1'(X)$ is
trivial. Hence (\ref{tensorproduct}) obtains the structure of
$\CC[H_1(X,\ZZ)]$-module. It  passes to the homology 
of the complex (\ref{tensorproduct}). We denote the resulting homology modules
as $H_i(X_{ab},V_{ab})$
 
\begin{definition}  The support of $\CC[H_1(X,\ZZ)]$-module $\Lambda^lH_i(\tilde
  X_{ab},V_{ab})$ we call {\it the $\rho$-twisted degree $i$, $l$-th characteristic variety of $X$} 
This is a subvariety of the torus $Spec\CC[H_1(X,\ZZ)]$ which we
denote as $Ch^l_i(X,\rho)$.
\end{definition}
In the case when $rk H_1(X,\ZZ)=1$ the support is a finite subset of
$\CC^*$ and hence the zero set of a unique monic polynomial of degree
$Card (Ch^l_i(X,\rho))$. More generally, a surjection $\epsilon:
H_1(X,\ZZ) \rightarrow \ZZ$ defines the surjection of the group
algebras and hence the embedding $\epsilon^*: \CC^*=Spec\CC[\ZZ]\rightarrow
Spec\CC[H_1(X,\ZZ)]$.
The Alexander polynomial $\Delta^l_n(X,\rho, \epsilon)$  is the unique
monic polynomial of minimal degree having in $\CC^*$ the roots $(Im
\epsilon^*)\cap Ch^l_n(X,\rho)$.

We refer to \cite{me2009} Sect. 5 for other results on the relation between the twisted Alexander
polynomials and characteristic varieties in the context of
complements to plane curves. For example the cyclotomic
property of the roots of Alexander polynomials becomes the following:
the roots of a $\rho$-twisted Alexander polynomial belong to a
cyclotomic extension of the extension of $\QQ$ generated by the eigenvalues of
$\rho$ (cf. \cite{me2009}, Th. 5.4). 

\subsection{Alexander Invariants of the complements without isolatedness properties.}\label{perverse} Investigations of 
 the Alexander invariants of the complements to hypersurfaces with isolated
singularities which were discussed in sections
\ref{alexandercurves}, \ref{divisibility11}, \ref{charvarsection} and \ref{isolatednonnormalsection}
were extended  to the case of hypersurfaces with non-isolated singularities\index{non-isolated singularities}
and further to smooth quasi-projective varieties by L.Maxim and his collaborators
(cf. \cite{maxim06}, \cite{liumaxim},\cite{liunearby}). The
$D$-modules (cf. \cite{saito},\cite{schnell}), the category of 
complexes of constructible sheaves, the perverse sheaves and
peripheral complex (going back to \cite{cappell} and
first studied in this context in \cite{maxim06})  are the key
technical tools used by
these authors. We will review two most
important outcomes of this approach: propagation of characteristic
varieties and extension of divisibility theorem for Alexander polynomials.
The propagation property of characteristic varieties was first noticed
in the context of arrangements of hyperplanes (cf. \cite{epy}, \cite{dsy}) 
 and extended further
 in \cite{liumaximwang},\cite{lmwsurvey}.  
The key observation is pure topological and concerns the spaces 
satisfying a version of cohomological duality.
\begin{definition} (cf. \cite{bieri}, \cite{dsy}, \cite{liumaximwang}) Let $X$ be a finite CW complex and $G=\pi_1(X,x_0), x_0\in
  X$. A topological space is called a duality space of dimension $n$ if
  $H^p(X,\ZZ[G])=0, p\ne n$ and $H^n(X,\ZZ[G])$ is non-zero and
  torsion free. A spaces $X$ is called {\it an abelian} duality space\index{abelian duality space} if for $A=H_1(X,\ZZ)$
  one has $H^p(X,A)=0, p \ne n$ and $H^n(X,A) \ne 0$ is torsion free. 
\end{definition}

\begin{theorem}  (cf. \cite{lmwsurvey}, Theorem 3.16) Let $X$ be an abelian duality
  space of dimension $n$.  Then the cohomology jumping loci of
 the  characters of the fundamental group $\pi_1(X)$ 
satisfy the following properties (the so call propagation package (\cite{liumaximwang},\cite{lmwsurvey}):

(i) {\it Propagation\index{propagation}}: Subvarieties $\mathcal{V}^i(X)$ form descending chain: 
 \begin{equation}
       \mathcal{V}^n(X) \supseteq \cV^{n-1}(X)\supseteq .... \supseteq \cV^0(X)
\end{equation} 

(ii) Codimension bound:
$${\rm codim}\cV^{n-i} \ge i$$ 

(iii) Irreducible components: if $V$ is an irreducible component of
codimension $d$ in $\cV^n(X)$ then $V\subset \cV^{n-d}$

(iv) Generic vanishing: for characters $\rho$ of the fundamental group
$\pi_1(X)$  in a Zariski open set in $Hom(\pi_1(X),\CC^*)$ one has
$H^i(X,L_{\rho})=0$ for $i \ne {\rm dim} X$.

(v) Signed Euler characteristic property: $(-1)^{dim X}{\chi}(X)>0$

(vi) Betti numbers inequality:  $b_i(X)>0$ for $0 \le i \le n$ and
$b_1(X) \ge n$.
\end{theorem}

Use of jump loci of constructible complexes on semi-abelian varieties
(which are Albanese varieties of $X$ is appropriate sense) is the key
step in the proof of this result.

This theorem suggests the problem of identifying the abelian duality
spaces. In this direction one has the following:

\begin{theorem} (cf. \cite{liumaximwang})  (i) Let $X$ a compact Kahler manifold which is abelian 
duality space. Then $X$ is biholomorphic to an abelian variety.

(ii) Let $X$ be quasi-projective manifold, such the 
albanese map $X \rightarrow Alb(X)$ is proper.  Then $X$ is an abelian
duality space. In particular a complement to a union of hypersurfaces
in $\CC^{n}$ satisfies the propagation package.
\end{theorem}

Next we will describe the identity having as a special case 
the divisibility relation between product of local Alexander polynomials of
singularities and the global Alexander polynomial from section
\ref{divisibility11}  in the case of curves and in \cite{meannals} in
higher dimensions. Let $f(z_0,....,z_{n+1})=0$ be a homogeneous polynomial of degree $d$ 
having as its zero set the hypersurface $V_f \subset \PP^{n+1}$ and
let $f_d(z_1,..,z_d)=f(0,z_1,....,z_{n+1})$ be the equation of the
intersection of $V_f$ with the hyperplane $H_{\infty}$ at infinity. 
The map 
\begin{equation}\label{maptocstar} f: \PP^{n+1}\setminus (H_{\infty}\cup V_f) \rightarrow
  \CC^*
\end{equation} 
which is the restriction of $f: \CC^{n+1}=\PP^{n+1}\setminus \PP^n
\rightarrow \CC$ given by $(1,z_1,...,z_{n+1}) \rightarrow f(1,z_1,...,z_{n+1})$
allows to define the infinite cyclic cover
$\widetilde{\PP^{n+1}\setminus (H_{\infty}\cup V_f)}$ corresponding to
the kernel of the map $\pi_1(\PP^{n+1}\setminus (H_{\infty}\cup V_f)
\rightarrow \pi_1(\CC^*)=\ZZ$ induced by (\ref{maptocstar}).
For each $1 \le i \le n$ one has a well defined (up to a unit of
$\CC[t,t^{-1}]$) polynomial $\Delta_i(t)$which is the order of the
$\CC[t,t^{-1}]$-module $H_i(\widetilde{\PP^{n+1}\setminus
  (H_{\infty}\cup V_f)})$ (cf. \cite{maxim06}). Let
$\psi_f\QQ_{\CC^{n+1}}$ denotes Deligne's nearby cycles complex
associated to $f$ and let $\psi_i(t)$ be the order
of $H^{2n+1-i}(V_f \cap \CC^{n+1},\psi_f\QQ_{\CC^{n+1}})$. Note that
$\Delta_i(t)=\psi_i(t) \  for \ i <n$ and $\Delta_n(t)$ divides $\psi_n$.
Let $h(t)$ be the characteristic polynomial of the Milnor fiber $f_d(z_1,..,z_{n+1})=1$.
The final ingredient is the determinant $det \phi$ of the bilinear form:  
\begin{equation}
   H_{n+1}(\widetilde{\PP^{n+1}\setminus (H_{\infty}\cup V_f)},\QQ(t))\otimes
   H_{n+1}(\widetilde{\PP^{n+1}\setminus (H_{\infty}\cup V_f)},\QQ(t))
   \buildrel \phi \over \rightarrow\QQ(t)
\end{equation}
given by $(\alpha,\beta) \rightarrow \alpha \cdot i(\beta)$ where
$i: \partial(\PP^{n+1}\setminus (H_{\infty}\cup V_f)))\rightarrow
\PP^{n+1}\setminus (H_{\infty}\cup V_f)$ is the embedding of the
boundary $\partial(\PP^{n+1}\setminus (H_{\infty}\cup V_f))$
of a small regular neighborhood of $V_f\cup H_{\infty}$, and $"\cdot"$
is the Poincare pairing in the homology of the infinite cyclic cover (cf. \cite{milnorduality})
of the
pair $((\PP^{n+1}\setminus (H_{\infty}\cup V_f),(\partial \PP^{n+1}\setminus (H_{\infty}\cup V_f)))$.
With these definitions one can describe the relation between the
global Alexander invariants and the data of the starta of a singular
hypersurface as follows.
\begin{theorem} (cf. \cite{liumaxim}) Let $f \in
  \CC[x_1,....,x_{n+1}]$ be defining polynomial of an affine hypersurface $F
  \subset \CC^{n+1}$, $f_d$ be the top degree form of homogenization of
  $f$ with corresponding Milnor fiber $F_h$ and $h_i(t)$ be the
  Alexander polynomial associated to $H_i(F_h)$. Let $\psi_n(t)$ denotes 
  the Alexander module associated to
  $H_c^{2n-i}(V_0,\psi_f\QQ_{\CC^{n+1}})$. Finally, let $\phi$
  be the intersection form on the infinite cover associated with $f$.
Then
\begin{equation}
       h_n(t)\psi_n(t)=\delta_n(t)^2det(\phi)
\end{equation} 
\end{theorem} 
In the case when $V_f$ is transversal to $H_{\infty}$ and has only isolated
singularities, this is translated into the following relation (cf. \cite{liumaxim}):
\begin{equation}
    (t-1)^{\vert e(\PP^{n+1}\setminus (H_{\infty}\cup V_f))\vert +(-1)^{n+1}}(t^d-1)^{{1 \over d}((d-1)^{n+1}+(-1)^n)}\prod_{p \in Sing(V_f)} \Delta_p=\Delta_n^2det(\phi)
\end{equation}i.e. one obtains topological interpretation of terms converting
divisibility into equality. In the case of plane curves one recovers
the result in \cite{cogoflorens}.

\section{Ideals of quasiadjunction and multiplier ideals}\label{section4}

Now we will turn to calculation of components of characteristic
varieties in terms of dimensions of linear systems determined by the
singular points of the curve. For earlier expositions of this material
in the case of plane curves 
cf. \cite{me09} or \cite{artal14}.

\subsection{Ideals and polytopes of quasi-adjunction}\label{qasection}

\begin{definition} Let $X$ be a complex $n$-dimensional
  manifold,  $P\in X$ be a point and let $B_P$ be a small ball centered at $P$. Let $D$ be a
reduced divisor on $X$ containing $P$, $f \in
  \mathcal{O}_P$ be a reduced germ of
  a holomorphic function having $D$ as its divisor and let $f(x_1,...,x_n)=f_1(x_1,...,x_n)\cdot ...\cdot f_r(x_1,....,x_n)$ be its
  prime factorization. Let 
\begin{equation}\label{arrays} 
 (j_1,...,j_r),(m_1,..,m_r), 0 \le j_i
  <m_i 
\end{equation} 
be two arrays of integers. Consider abelian branched cover
$V_{m_1,....,m_r}$ of $B$ 
ramified over $D$ and corresponding to the component-wise reduction 
$H_1(B_P\setminus D\cap B_P,\ZZ)=\ZZ^r\rightarrow
\oplus_1^r\ZZ/m_i\ZZ$ 
(cf. Section \ref{branchedcoverssect}). 
  After selecting local coordinates near $P$, this cover can be viewed as a germ at the origin in $\CC^{n+r}$ with
  coordinates $(z_1,....,z_r,x_1,...,x_n)$ given by  
the local equations:
\begin{equation}\label{abelcover}
   z_1^{m_1}=f_1(x_1,...,x_n),.......,z_r^{m_r}=f_r(x_1,...,x_n), 
 \end{equation}
The covering map of subvariety (\ref{abelcover}) onto $B_P$ is given
by projection $$(z_1,...,z_r,x_1,....,x_n)\rightarrow (x_1,...,x_n).$$
The ideal of quasi-adjunction $\cA(f \vert j_1,...,j_r\vert
m_1,...,m_r)$ of $f(x_1,...,x_n)$, corresponding to the array
(\ref{arrays}), is the ideal of
germs $\phi \in \mathcal{O}_P$ in the local ring of $P$, such that the
$n$-form 
\begin{equation}\label{form}
         \omega_{\phi}={{\phi(x_1,...,x_n) dx_1\wedge...\wedge dx_n} \over {z_1^{m_1-j_1-1}.....z_r^{m_r-j_r-1}}}
\end{equation}
defined on the smooth locus of (\ref{abelcover}), 
can be extended over a log resolution of (\ref{abelcover}).  
\end{definition}
One shows that the ideal $\cA(f \vert j_1,..,j_r \vert m_1,...,m_r)$ is
independent of a resolution of (\ref{abelcover}) and that it depends only on the vector:
\begin{equation}\label{vector}
   \left({{j_1+1}\over m_1},....,{{j_r+1}\over m_r}\right) \in [0,1]^r \subset \RR^r
\end{equation} 
rather than on specific values of $j_i,m_i$.
To see that dependence is only on (\ref{vector}), let us consider the following
resolution of the germ (\ref{abelcover}): select a log-resolution 
\begin{equation}\label{logres}
\mu: (\tilde X,\tilde D) \rightarrow (B_P,B_P\cap D)
\end{equation} with
the exceptional set $E=\bigcup_1^K E_k$ i.e. assume that $E\cup \mu^*(D)$ is a normal
crossing divisor. Consider a resolution $\mathcal{X}_m$ of
the {\it normalization} $\tilde X_{m_1,...,m_r}$ of the fiber product $\tilde X\times_X V_{m_1,..,m+r}$ 
\begin{equation}
 \begin{matrix}    \mathcal{X}_{m_1,...,m_r}  & & \cr
      \downarrow  &  & \cr 
{\tilde X}_{m_1,...,m_r} & & \cr
\downarrow & & \cr
\tilde X\times_X V_{m_1,..,m_r} &
   \rightarrow & V_{m_1,...,m_r} \cr
  \tilde \pi_{m_1,...,m_r} \ \  \downarrow & & \pi_{m_1,...,m_r}\downarrow \cr
  \tilde X & \buildrel \mu \over \rightarrow & \CC^n \cr  
    \end{matrix}                            
\end{equation}
The normalization $\tilde X_{m_1,..,m_r}$ has abelian quotient singularities and
hence is $\QQ$-Gorenstein Kawamata log-terminal (cf. \cite{batyrev}). 
In particular an $n$-form on the smooth locus of $\tilde X_{m_1,..,m_r}$ extends over
exceptional locus of $\mathcal{X}_m \rightarrow X_{m_1,..,m_r}$ iff it is
holomorphic on the smooth locus of $\tilde X_{m_1,..,m_r}$ (cf. \cite{GKKP}).  
 Let $a_{k,i}=ord_{E_k}\mu^*(f_i),
c_k=ord_{E_k}\mu^*(dx_1\wedge...\wedge
dx_n),e_k(\phi)=ord_{E_k}\mu^*(\phi)$. 
We obtain that $\tilde \pi_{m_1,..,m_r}(\omega_{\phi})$ extends over the smooth
locus of normalization $\tilde X_{m_1,...,m_r}$ iff for all
irreducible components $E_k \subset X$ one has:
\begin{equation}\label{multiplicities}
   \sum_i a_{k,i}{{j_i+1} \over m_i} > \sum_i a_{k,i}-e_k -c_k-1 \ \ \ e_k=e_k(\phi)
\end{equation}

\begin{definition}\label{polytopeqa} For a choice of collection $\cE=\{e_k\}, k=1,...,K$ of
  non-negative integers labeled by irreducible components of
 resolution  (\ref{logres}) and such that there exist a germ $\phi \in
 \mathcal{O}_P$ such that
$e_k=e_k(\phi)$,  the closure $\cP(\cE)$ in the $r$-cube $[0,1]^r$ of the set of solutions to 
inequalities (\ref{multiplicities}) is called a {\it polytope of quasi-adjunction\index{polytope!of quasi-adjunction}}\footnote{By a polytope we mean a set of solutions to
  a finite collection of inequalities. All polytopes considered here are
  bounded (subsets of a unit cube) and hence are the convex hulls of
  the sets of their vertices. Faces are subsets of the boundary of a
  polytope which are the convex hulls of a subset of the set of
  vertices of the polytope. The dimension of a polytope (including a
  face) is the maximal
dimension of the balls in its interior (with the dimension of a vertex
being zero).} of the germ $f=f_1 \cdot....\cdot f_r$.
\end{definition}

Using polytopes $\cP(\cE)$ or inequalities (\ref{multiplicities}), one can
describe necessary and sufficient conditions on (\ref{arrays})
assuring extendability of the form $\omega_{\phi}$
(cf. (\ref{form})) on resolution of singularities of abelian covers in
the tower (\ref{abelcover}). While the polytopes $\cP(\cE)$ form a set
which is partially
ordered by inclusion and closely related to Alexander invariants,
sometimes it is convenient also to use  
a partition of $[0,1]^r$ into a union of (locally
closed) {\it non-intersecting and possibly non-convex} polytopes $\cQ$ compatible with polytopes $\cP(\cE)$. 

To describe these polytopes consider the hyperplanes in $\RR^r$
\begin{equation}\label{equationfaces}
           \sum_i a_{k,i}(\gamma_i-1) +c_k=\epsilon_k, \epsilon_k \in
           \ZZ_{< o}
\end{equation} 
 labeled by the exceptional
components $E_k, k=1,...,K$ of resolution $\mu$ (cf. (\ref{logres})) 
where collections of integers $\epsilon_k$ are such that there exist $\phi \in \mathcal{O}_P$ satisfying
 $\epsilon_k=e_k(\phi)$ and the remaining coefficients in
 (\ref{equationfaces}) are defined just before (\ref{multiplicities}).
The hyperplanes (\ref{equationfaces}) 
partition the unit cube $[0,1]^r$ into a union of locally closed
polytopes $\cQ$ of various dimensions, such that
$\upsilon=(...,\gamma_i,..)=(...,{{j_i+1} \over m_i},...)$ and
$\upsilon'=(...,\gamma_i',..)=(...,{{j'_i+1} \over m'_i},...)$
belong to the same polytope iff $\upsilon,\upsilon'$ satisfy the same sets
of inequalities (\ref{multiplicities}). Equivalently, 
 each polytope $\cQ$ coincides with the 
polytope formed by an equivalence class of points in
$[0,1]^r$ when one considers points equivalent if they have the same
set of polytopes $\cP(\cE)$ containing each. We wil call the polytopes
$\cQ$, defining the decomposition of the unit cube into disjoint union,
{\it the strict polytopes of quasi-adjunction} to distinguish them
from the ordinary polytopes of quasi-adjunction $\cP(\cE)$.

For a fixed vector $\upsilon$ with coordinates
(\ref{vector}) the germs $\phi$ which satisfy
the inequalities (\ref{multiplicities}) form the ideal $\cA_{\cQ} \subset  \mathcal{O}_P$ depending only on
the polytope $\cQ$ and not on a choice of $\upsilon \in \cQ$.
In particular,  
the ideal of quasi-adjunction 
$\cA(f\vert j_1,...,j_r \vert m_1,...,m_r)$ is the ideal $\cA_{\cQ}$ such that
  $(....,{{j_i+1}\over {m_i}},...) \in \cQ$.

\begin{definition}\label{sheafofquasiad} 1.
 We  define an ideal of quasi-adjunction\index{ideal of quasi-adjunction} as an ideal coinciding with an ideal $\cA_{\cQ}$ for some strict polytope of quasi-adjunction $\cQ \subset
  [0,1]^r$. 

2. For a polytope in $[0,1]^r\subset \RR^r$ as above, a {\it face of
  quasi-adjunction}\index{face of
  quasi-adjunction} $\cF$ of dimension $p$ is the $p$-dimensional
intersection of the polytope of quasi-adjunction with an affine
half-space in $\RR^r$, transversal to all coordianted hyperplanes in
$\RR^r$, such
that each point of this intersection is the boundary point of the
polytope and the half-space.
\footnote{i.e.the set of solutions to a linear inequality}  
\end{definition}

\begin{corollary} The polytopes, ideals and faces of quasi-adjunction depend only on the
germ $f(x_1,...,x_n)$ and not on a resolution. The inequalities (\ref{multiplicities}) show that
there exist firstly, the collection of polytopes $\cP$, which are the
unions of the polytopes $\cQ$, and secondly to each of these
polytopes correspond the ideal $\cA_{\cQ}$
such that for any $\upsilon \in \cQ$ and a germ $\phi \in \mathcal{O}_P$, the
corresponding form 
$\omega_{\phi}$ can be extended to a holomorphic form on
$\cX_{m_1,...,m_r}$ iff $\phi \in \cA_{\cQ}$. The boundary of such a polytope
$\cQ$ is a union of faces each being also a face the boundary of a
polytope $\cP$ and
each such face being a close polytope in the
intersection of hyperplanes given by the equations (\ref{equationfaces})
with 
\begin{equation}\gamma_i=1-{{j_i+1}\over m_i} \ \ \ i=1,...,r.
\end{equation}
\end{corollary}

\subsection{Ideals of quasi-adjunction and multiplier ideals}
For an exposition of the theory of multipler ideals we refer to
\cite{lazarsfeld} Part III.  Here we recall the key definitions and
relate them to the ideals of quasi-adjunction.

\begin{definition}(cf.\cite{lazarsfeld} Def.9.2.1). Let $X$ be a smooth
  complex variety and $D \in Div(X)\otimes \QQ$ an effective $\QQ$-divisor. Let
  $\mu: X'\rightarrow X$ be a log resolution,
  $K_{X'/X}=K_{X'}-\mu^*(K_X)$ is a relative canonical class. The multiplier ideal sheaf\index{multiplier ideal sheaf} 
  $\cJ(X,D)$ of $D$ is the direct image 
\begin{equation}\label{multiplierdef}
   \mu_*(K_{X'/X}-\lfloor \mu^*(D) \rfloor)
\end{equation}
where for a $\QQ$-divisor $D=\sum \gamma_iD_i, a_i\in \QQ, D_i\in Div(X)$,
$\lfloor D \rfloor=\sum \lfloor \gamma_i \rfloor D_i$ and $\lfloor \gamma \rfloor \in \ZZ$ denotes the integral part of
$\gamma\in \QQ$.

For a collection of $\QQ$-divisors $F_1,...,F_r$ (resp. the ideals
$\wp_1,...\wp_r$), the mixed multiplier ideal
$\cJ(c_1F_1,...,c_rF_r)$
(resp, $\cJ(\wp^{c_1}\cdot .... \cdot \wp^{c_r}_r)$) (cf. \cite{lazarsfeld} 9.2.8, \cite{alberichjump})
is defined as the multipler ideal
\begin{equation}
   \mu_*(K_{X'/X}-\lfloor c_1F_1+....+c+rF_r \rfloor) 
\end{equation} 
(in the case of mixed multiplier ideals attached to
$\wp_1,..,\wp_r$, the divisors $F_i$ are determined from $\wp_i\mathcal{O}_{X'}=\mathcal{O}_{X'}(-F_i)$).
\end{definition}

\begin{proposition}\label{quasiadmultiplier} The ideal of quasi-adjunction $\cA(f \vert j_1,...,j_r
  \vert ,m_1,..,m_r)$ coincides with the multiplier ideal 
  $\cJ(\sum ((1-{{j_i+1} \over {m_i}})D_i)$ 
\end{proposition}

\begin{proof} A germ $\phi$ is a section of the sheaf given by (\ref{multiplierdef}) where
  $D=\sum \gamma_iD_i$ if and only if 
  it satisfies the inequality
\begin{equation}
   e_k(\phi)+c_k-\sum \gamma_ia_{k,i} \ge 0 
\end{equation}
This is equivalent to (\ref{multiplicities}).
\end{proof}

Following \cite{mustatame} one can define the LCT polytope, which is 
a multi-divisor analog of the log canonical threshold (cf. \cite{lazarsfeld}):

\begin{definition}
 $$ LCT(D_1,.....,D_r)=\{ (\lambda_1,....,\lambda_r) \in \RR^r_{+} \vert
   (X,\sum_1^r \lambda_iD_i) \ is\ log \ canonical \}$$
\end{definition}

\begin{proposition} (log-canonical polytopes\index{log-canonical polytope} and polytopes of
  quasi-adjunction). Let $I: [0,1]^r\rightarrow [0,1]^r$ be the
involution of the unite cube given by $(\gamma_1,...,\gamma_r)
\rightarrow (1-\gamma_1,...,1-\gamma_r)$.
Then $LCT(D_1,..,D_r)$ is the $I$-image of the part in the interior of
the unite cube of the boundary of the polytope of
quasiadjuction containing the origin. 
\end{proposition}

\begin{proof} Clearly for the vectors (\ref{vector}) sufficiently close
  to zero, the $n$-form (\ref{form}) is extendable for any $\phi$ in the
  local ring of $P$ i.e. the ``ideal'' of quasi-adjunction is not
  proper. The ideal of quasi-adjunction is constant for all vectors
(\ref{vector}) in the polytope bounded by the faces of quasi-adjunction
closest to the origin. Hence the claim follows from the
characterization of log-canonical thresholds in terms of multiplier
ideals  (cf. \cite{lazarsfeld} Section 9.3.B) and Proposition \ref{quasiadmultiplier}). 
\end{proof}

\begin{remark}\label{placeforpolytopes} 1. We often use the following correspondence between
  vectors (\ref{vector}) and the characters of local
  fundamental group of the complement to the germ of $D=\bigcup D_i$.
Consider the embedding of $[0,1)^r$ into $H^1(B_P \setminus B_P\cap
\bigcup D_i,\RR)$ using the basis dual to the meridians $\delta_i$ of
the divisors
$D_i$ i.e. assigning to a vector $(\gamma_1,....,\gamma_r)$ the
cohomology class $h$ such that $h(\delta_i)=\gamma_i$.  

This embedding induces the identification of the cube $[0,1)^r$ with
the group of characters of the local fundamental group. Indeed, any point in  $H^1(B_P \setminus B_P\cap
\bigcup D_i,\RR)$ via exponential map $t
\rightarrow exp(2 \pi i t)$ determines an element in $H^1(B_P \setminus B_P\cap
\bigcup D_i,U(1))$ i.e. an unitary character of the local fundamental
group \footnote{this is a local version of the global construction
  used in Prop. \ref{datarelation}. Here the rank of the fundamental
  group coincides with the number of irreducible components of the divisor.}
.  Vice versa, to a character in $H^1(B_P \setminus B_P\cap
\bigcup D_i,U(1))$  we can assign its unique preimage
belonging to the fundamental domain of the action of $H^1(B_P \setminus B_P\cap
\bigcup D_i,\ZZ)$ on $H^1(B_P \setminus B_P\cap
\bigcup D_i,\RR)$ which is the unit cube in the coordinates of the
above basis.

\par 2. Exact sequence of a pair $(B_P,B_P\setminus B_P\cap D)$ gives
the identification:
\begin{equation}\label{eqplaceforpolytopes}
        H^1(B_P\setminus B_P\cap D,\RR)=H^2(B_P,B_P\setminus D,\RR)=H_2(B_P\cap D,\partial
        (B_P\cap D),\RR)
\end{equation}
and similarly for the coefficients $\ZZ$ and $U(1)$. $B_P\cap D$ is
homeomoprhic to a bouquet of disks and $\partial(B_P\cap D)$ is a disjoint union of
circles, both having the cardinality coinciding with the number of
branches of $D$ at $P$. In particular (\ref{eqplaceforpolytopes}) is a
vector space with canonical direct sum decomposition with summands
corresponding to the branches of $D$ at $P$.
\end{remark}

\subsection{Local polytopes of quasi-adjunction and spectrum of singularities}\label{spectrum11}

\subsubsection{Cyclic Theory} The relation between the constants of quasi-adjunction (i.e. the
faces of quasi-adjunction in cyclic case) and the Hodge theory
was first discussed in \cite{loeservaq} in the case of
isolated singularities. This was continued in \cite{budurspectrum},
\cite{pierretteme11}.

Let $f(x_0,x_1,..x_n)$ be a germ of an isolated singularity at the
origin. Recall that the cohomology of the Milnor fiber $F_f$ support
a mixed Hodge structure defined by Steenbrink and Varchenko
(cf. \cite{steenbrink},\cite{peterssteenbrink}, \cite{varchenko}, \cite{arnoldetal}).
This mixed Hodge structure  provides the cohomology $H^n(F_f,\CC)$
with a decreasing filtration $F^p \cap H_{\lambda} \subseteq F^{p-1}
\cap H_{\lambda}$ on each summand of the direct sum decomposition of $H^n(F_f,\CC)$ into the eigenspaces of the
monodromy operator. A rational number $\alpha$ is an element of the spectrum\index{spectrum} of
$f$ of multiplicity $k$ if there is an eigenvalue of the monodromy 
$\lambda$ such that 
\begin{equation}\label{spectrum}
  n-p-1 <\alpha \le n-p \ \ exp(2 \pi i \alpha)=\lambda  \ \ \ dim
  F^p\cap H_{\lambda}/F^{p+1} \cap H_{\lambda}=k
\end{equation}(cf. \cite{loeservaq})
\begin{theorem} A rational number $\alpha$ belongs to the spectrum of
  $f(x_0,...x_n)=0$ and satisfies $0 < \alpha < 1$ if and only iff
$-\alpha$ is a  face of quasi-adjunction
of $f$ (i.e. is a constant of quasi-adjunction 
in terms of  \cite{arcata81};  the definition given there is in terms of
the adjoint ideals and is a special
case of  Def. \ref{sheafofquasiad} corresponding to cycilc
covers of $\CC^2$.)
\end{theorem}

There are many cases when spectrum can be easily calculated explicitly. In the case of quasi-homogeneous singularities (with weights
$w_0,...,w_n$ i.e. when defining polynomial is a sum of monomials 
$ax_0^{m_0}....x_n^{m_n}$ such that $\sum_{i=0}^{i=n}w_im_i=1$) the generating
function i.e. $\sum_{\alpha} t^{\alpha}$ where
$\alpha$ runs through the spectrum of the singularity is given by 
\begin{equation}
   {1 \over t}\prod_{i=0}^{n} {{t^{w_i}-t}\over {1-t^{w_i}}}
\end{equation}
cf. \cite{saitospectrum}; we included the factor ${1 \over t}$ since we use 
the same normalization of the spectrum as in \cite{loeservaq} where the
left end of the spectrum is $-1$, cf. (\ref{spectrum}). 
In particular for an irreducible germ with one characteristic pair
$x^p=y^q$ we obtain $(t^{-{1 \over p}}+.....+t^{-{{p-1} \over
    p}})(t^{1\over q}+....+t^{{q-1}\over q})$
which for the cusp $x^2=y^3$ gives $t^{-{1 \over 6}}+t^{1 \over 6}$.
 Explicite calculations
of the spectrum and hence the constants of quasi-adjunction of
irreducible plane curve
singularities in terms of Puiseux pairs was made in \cite{saitospectrum} 
For related calculations see  \cite{budurspectrum}, \cite{galindo}. 

\subsubsection{Local abelian theory}

Calculations of 
the polytopes of quasi-adjunction of singularities of plane curves
were made in \cite{me2002},\cite{pierretteme11},\cite{pierretteme14}. Curves on
surfaces with rational singularities were considered in
\cite{alberichjump}. 

Several examples of calculation of polytopes and ideals of
quasi-adjunction for isolated non-normal crossings were considered in
 \cite{isolatednonnormal} \cite{me2009}  which lead to expressions for characteristic
 varieties mentioned in Example \ref{isolatednonnormalexample}.
Related results are presented in \cite{mustatame}.

\subsection{Ideals of quasi-adjunction and homology of
  branched covers} 

One of the first applications of ideals of quasi-adjunction was a
procedure that allowed to express the Hodge numbers of abelian covers in terms of 
dimensions of linear systems determined by the branch locus and 
 the data of singularities. 
The relation between the constants and ideals of
quasi-adjunction and Hodge numbers of the cyclic covers in
\cite{arcata81} followed by
numerous works (cf. for example \cite{loeservaq}, \cite{vaquie},
\cite{naie1}, \cite{naie2}, \cite{budur}), many using the terminology of multiplier ideals. \footnote{the term introduced by A.Nadel in 1990, cf. \cite{nadel}}

In the case of curves on
surfaces one has the following result. For a high dimensional 
extension leading to calculations of the dimensions of the space of holomorphic
forms on the resolutions of branched covers (the only Hodge
numbers which are independent of resolutions) see \cite{esnault}, \cite{meannals},
\cite{budur} etc.
\begin{proposition}\label{homologyquasiadjunction}
 Let $f: \bar X_{\phi} \rightarrow X$ be an abelian branched cover of a
 smooth projective surface $X$
  ramified over a divisor $D$ with $r$ irreducible components and let
  $f_*(\mathcal{O}_{\bar X_{\phi}})=\oplus
  L^{-1}_{\chi}$ be the decomposition (\ref{pardinidata}). For a character
  $\chi$, let $\cA_{\chi}\subset \mathcal{O}_X, \chi=exp (... .{2 \pi i
    \alpha(\chi)},...), \alpha \in [0,1)^r$ be the ideal sheaf having as stalk
  at $p \in Sing(D)$ the
  ideal of quasi-adjunction of  singularity at $p$ corresponding to the
  polytope of quasi-adjunction containing $\alpha(\chi)$

 Then the dimension of the $\chi$-eigenspace of the covering group
 acting on the space of holomorphic 1-forms on a resolution
 $\widetilde{X_{\phi}}$
of singularities of $\bar X_{\phi}$ is given by 
\begin{equation}\label{homologyquasiadjunctionformula}
   dim H^0(\widetilde{X_{\phi}},\Omega^1_{\widetilde{X_{\phi}}})_{\chi}= H^1(X,K_X\otimes
   L_{\chi} \otimes \cA_{Sing D}(\chi))
\end{equation}
\end{proposition}

\subsection{Hodge decomposition of characteristic varieties.}

\subsubsection{Calculation of Characteristic varieties: 
  Deligne extensions} 

In this section we sketch a method for calculation of the variety of
unitary characters with corresponding local systems having 
positive first Betti number and which is based on the Hodge theoretical description of
the cohomology of local systems due to Deligne and Timmersheidt
(\cite{timmer},\cite{delignediffeq},\cite{delignehodgeI},\cite{delignehodgeII},
related works include works of S.Zucker, M.Saito, El Zein and Illusie
cf. discussion in H.Esnault review of \cite{timmer} in
Math. Reviews). For
details we refer to \cite{me2009},
\cite{budur},\cite{esnaultviehweg}, \cite{artalorbifolds}.

The starting point is deRham type description of the
cohomology of local systems on smooth quasi-projective varieties
using logarithmic forms already mentioned in section
\ref{alexandeinvariantsjump}. 
With notations used in this section, we assume now
that $Y$ is smooth and quasi-projective and that $\bar Y$ is
a smooth projective compactification, such that $\bar Y\setminus Y$ is a
normal crossing divisor  $Y^{\infty}=\bigcup_{i \ge 1} Y_i$. The connection
$\nabla$ in (\ref{derhamconnection}) can be selected to be holomorphic
on $Y$ and meromorphic on $\bar Y$ i.e. in having a
matrix given in local coordinates by meromorphic functions with poles along $\bar Y
\setminus Y$. Moreover, this selection can be made so that the matrix
of connection has as its entries the 1-forms 
having logarithmic poles along $Y^{\infty}$ (i.e. linear combinations:
 $\omega=\sum \alpha_i{{dz_i}\over {z_i}}$ where $z_i$ are
local equations of irreducible components of $Y^{\infty}$ and
$\alpha_i$ are holomorphic in a chart in $\bar Y$ cf. \cite{delignediffeq}, Prop.3.2). Globally, logarithmic
connection can be viewed as a $\CC$-linear map $E \rightarrow
\Omega^1(log Y^{\infty}) \otimes E$ where $E$ is a vector bundle on
$\bar Y$satisfying Leibnitz rule.
Note that the matrix of connection, in the rank one case, is just a logarithmic
1-form. One has a well defined 
Poincare residue map $\Omega^1(log Y^{\infty})\otimes E
\rightarrow \mathcal{O}_{Y_i}\otimes E$  along each irreducible
component $Y_i$. Locally, residue depends on trivialization and globally on the
bundle $E$. 
\begin{definition} {\it Deligne's extension\index{Deligne P.!Deligne's extension} of a flat connection
    $\nabla$ on $Y$}  is
  a logarithmic connection on a bundle $E$ on $\bar Y$ such that its residues satisfy
  the inequality
\begin{equation}\label{residueinequality}
   0\le   Res_{Y_i}(\nabla)<1  
\end{equation}
for any $i$. 
\end{definition}
The description of the jumping loci of local systems in terms of
Deligne's extensions, is based on 
the degeneration of the Hodge-DeRham spectral sequence
\begin{equation}
  E_1^{p,q} = H^p(\Omega^q(log Y^{\infty}) \otimes V_{\rho}) \rightarrow  H^{p+q}(\cV_{\rho}). 
\end{equation}
in term $E_1$.  Here $V_{\rho}$ is a Deligne's extensions of a unitary
connection corresponding to the local system $\cV_{\rho}$ where
$\rho:\pi_1(Y) \rightarrow U(n)$ is a unitary representation (cf. \cite{timmer}). 
For a rank one local system $\cV_{\chi}$, corresponding to a character 
$\chi=exp(2 \pi i u)$ and the Deligne extension $L_{\chi}$ of the corresponding
connection, this degeneration implies $rk
H^1(Y,\cV_{\chi})=rk H^0(\Omega^1(logY^{\infty} \otimes L_{\chi}))+rk
H^1(L_{\chi})$ and the following:
\begin{proposition} A subset of $[0,1)^{b_1(X\setminus D)} \subset H^1(X\setminus D,\RR)$ such that the
 Deligne extension of the connection corresponding to a character
 $\chi=exp(2\pi i u)$ coincide with a fixed line bundle $L$ is a
 polytope $\Delta_L$ in $H^1(X\setminus
 D,\RR)$. The rank $rk H^1(Y,\cV_{\chi})$ is constant when $u$ varies
 within $\Delta_L$.
The subset of the torus ${\CC^*}^{b_1(X\setminus D)}=\Char
H_1(X\setminus D,\ZZ)$ consisting of the 
characters $exp(2 \pi i u), u\in \Delta_L$ is a jumping set of the
Hodge numbers of unitary local systems.  The Zariski closure in
$(\CC^*)^{b_1(X\setminus D)}$ of this subset 
is a translated by a finite order character a connected subgroup of
${\CC^*}^{b_1(X\setminus D)}$ 
which is an irreducible component of the characteristic variety of the
fundamental group. Vice versa, any irreducible component of
characteristic variety of $\pi_1(X\setminus D)$, is Zariski closure of a set $exp(\Delta_L)$.
\end{proposition}

Indeed, after selecting a basis in $H^1(X\setminus D,\RR)$, one readily sees
that inequality (\ref{residueinequality}) for each component
translates into a linear inequality on components of logarithm of $\chi$ in
coordinates in this basis.
It is not hard to see that there are only finitely many bundles on $X$
which are the Deligne extensions of a connection corresponding to a 
character in $[0,1)^{b_1(X\setminus D)}$ can occur (cf. \cite{me2009}). 
This provides an explicite description of the unitary part of the components of
characteristic varieties. Since by \cite{arapura} all irreducible components of
the characteristic variety are translated subtori of
${\CC^*}^{b_1(X\setminus D)}$, it follows that in this way we obtain
all the components as the Zariski closures of the exponential images
of the polytopes
$\Delta_L$. 

\subsubsection{Calculation of Characteristic Varieties: quasi-adjunction} We will focus on the
case of characteristic varieties of curves on surfaces. Similar
results are expected for characteristic varieties associated with
higher homotopy groups. We refer to  \cite{mecharvar}, \cite{meample} for some of the
results in this direction. 

Calculation in terms of ideal of quasi-adjunction is based on
comparison of topological and algebro-geometric calculation of 
the dimensions of eigenspaces of the action of the Galois group on the
homology of abelian covers given respectively by Propositions \ref{homologysakuma}
and \ref{homologyquasiadjunction}. However, the Proposition
\ref{homologysakuma} considers only the characters of
$\pi_1(X\setminus D)$ which values on the meridians of all irreducible
components of $D$ are non-trivial. This motivates the following
definition.

\begin{definition}\label{essentialcomp} Let $D=D_1\cup D_2$ be a decomposition of a reduced divisor on 
a smooth projective simply connected surface $X$. 
Let $s_{D/D_1}: \pi_1(X\setminus D) \rightarrow \pi_1(X\setminus D_1)$ be the
surjection of the fundamental groups induced by inclusion $X\setminus D
\subset X\setminus D_1$ and let 
 $s_{D/D_1}^{H_1}, s_{D/D_1}^{\pi_1'},s_{D/D_1}^{\pi_1'/\pi_1"}$
be the corresponding surjections respectively on the homology, commutator of the fundamental group and
the abelianization of the commutator. Let $s_{D,D_1}^{Char_i}:
Char_i^1(X\setminus D_1) \rightarrow Char_i^1(X\setminus D)$ be induced map
of supports of the $i$-th exterior powers of the homology modules
(over $\CC[H_1(X\setminus D_1)]$ and $\CC[H_1(X\setminus D)]$
respectively, cf. Definition \ref{defdepth}).   
An irreducible component $\mathcal{C}_D$ of the characteristic
  variety of $\pi_1(X\setminus D)$ is called {\it non-essential} if 
there is a decomposition $D=D_1\cup D_2$ and a component $\mathcal{C}_{D_1}$ of the characteristic variety of $\pi_1(X\setminus D_1)$
such that $s_{D/D_1}^{Char_i}(\mathcal{C}_{ D_1})=\mathcal{C}_{D}$. {\it An
essential component}\index{essential component} is an irreducible component of $Char_i(X\setminus
D)$ which fails to be non-essential.
\end{definition}

Below we describe a calculation of only essential components,
for simplicity assuming that $H_1(X \setminus D,\ZZ)$ has no torsion.
The results certainly can be described without this assumption. In fact, the first examples
of calculations of characteristic varieties (the roots of the Alexander
polynomials) in terms of ideals 
of quasi-adjunction, were made in \cite{arcata81} in the cases when $D$ is a
an irreducible plane curve i.e. when $H_1(\PP^2\setminus D,\ZZ)$ is a cyclic
group of order $deg D$. See also example (\ref{sectionsofbundles}) where the complement is
torsion as well. Non
essential components may exhibit a subtle behavior: the depth of
a component may increase considered as component of $\mathcal{C}_D$
instead of $\mathcal{C}_{D_1}$. We refer to discussion of this phenomenon
to \cite{ajcarmona2},\cite{artalorbifolds}, \cite{acme2}. 

\smallskip

Recall (cf. section \ref{qasection} and Remark \ref{placeforpolytopes}) that with each singular point $P$
of the reduced divisor $D=\bigcup_1^N D_i$ on a surface $X$ we
associated 
a collection of
polytopes of quasi-adjunction $\cQ_j(P), j=1,...,n(P)$ in
$U_{s(P)}=[0,1)^{s(P)} \subset H^1(B_P\setminus 
D \cap B_P,\RR)=H_2(B_P\cap D,\partial B_P\cap D,\RR)$ (recall that
$B_P$ is a small ball in $X$ centered at $P$). 
Note that the latter locally homology groups can be endowed with the maps to the corresponding global
groups for each group of coefficients $\KK=\ZZ,\RR,U(1)$ leading to  the diagram:
\begin{equation}\label{polytopesfrompoints}
\begin{matrix}   H^1(B_P\setminus B_P\cap D,\KK) & \buildrel \delta_P
  \over \rightarrow &
   H_2(B_P\cap D, \partial (B_P\cap D),\KK)
 \cr
\uparrow i_P & &\uparrow   \epsilon_P \cr
   H^1(X\setminus D,\KK) & \rightarrow   &  H_2(D,\KK) \cr
\end{matrix}
\end{equation} 
Here the top horizontal map $\delta_P$ is the isomorphisms
(\ref{eqplaceforpolytopes}),
the left vertical map $i_P$ is induced by embedding 
and the right vertical map $\epsilon_P$ is the homology boundary map:
$H_2(D,\KK) \rightarrow H_2(D,D\setminus B_P\cap D,\KK)=
 H_2(B_P\cap D, \partial (B_P\cap
D,\KK)$).

For each polytope $\cQ_j(P) \subset 
 H^1(B_P\setminus D \cap B_P ,\RR)$ we consider preimages
$$\cQ_j^X(P)=(i_P^*)^{-1}(\cQ_j(P)) \subset H^1(X\setminus D,\RR) \ \
{\rm and} \ \ \cQ_j^D(P)=\epsilon_P^{-1}(\delta_P(\cQ_j(P)))\subset H_2(D,\RR)$$

In  (\ref{polytopesfrompoints}),
each group in the top row for $\KK=\RR$ and  the group  $H_2(D,\RR)$ contains
the canonical fundamental domain for the action of respective lattice
which one obtains taking for each group $\KK=\ZZ$. 
These fundamental domains are the unit cubes in the bases
corresponding to the fundamental classes of appropriate irreducible component of $D$.
The image of each such fundamental domain in $H^1(B_P,B \setminus
B_P\cap D,\RR)$, induced by embedding $H^1(B_P,B \setminus
B_P\cap D,\RR) \rightarrow \oplus_{P \in Sing(D)} H_2(B_P\cap D, \partial (B_P\cap D),\RR)$,
is a face of the unit cube in the latter. The intersection of the
image of $H_2(D,\RR)$ in $\oplus_P H_2(B_P\cap D, \partial (B_P\cap D),\RR)$
 is either such a face, if the branches of $D$ at $P$  belong
to different irreducible components of $D$, or a diagonal in such a
face, if different branches at $P$ belong to the same irreducible
component of $D$. We denote by $U_{X,D}$ the unit cube 
in $H_2(D,\RR)$. From now on we will use the same nation
$\cQ_j^X(P),\cQ_j^D(P)$ for their intersections with the respective
unite cubes: these and only parts of respective polytopes contain 
the information we need below.

The unit cubes in $H_2(D,\RR)$ and $H_2(D \cap B_P,\partial (D\cap B_P),\RR)$,
have canonical
involution corresponding to the lift of the conjugation of characters 
via inverse of the map induced on cohomology by $\exp: \RR \rightarrow
U(1)$.  For example on each unit
cube in $H^1(B_P\setminus B_P\cap D,\RR)$ this involution is given by $(u_1,..,u_r)
\rightarrow (1-u_1,....,1-u_r)$ and similarly in other cases. For a
subset $\cU$ in such a cube we
denote the image of this involution as$-\cU$.

\begin{definition} Let $Sing(D)$ be the set of singularities of $D$ 
and let $\cQ$ be the set of collections $Q=\{ \cQ_j(P_k) \vert P_k \in Sing(D),
j=1,...J(P_k)\}$ (here $J(P_k)$ is the cardinality of the set of local 
polytopes of quasi-adjunciton of singularty $P_k$)
of (strict \footnote{strict
  polytopes of quasiadjunction were described
  just before Def. \ref{sheafofquasiad}}) local
polytopes of quasi-adjunction $\cQ_j(P_k)$,
one for each singularity $P_k \in Sing(D)$.

a. {\it The divisorial global polytope of quasi-adjunction\index{divisorial global polytope of quasi-adjunction}} is the intersection
  \begin{equation}\label{global} \cG_Q=\bigcap_{P_k\in Sing(D),Q_j(P_k) \in Q} {pr_{P_k}}^{-1} \cQ_j(P_k)
       \subset U_{X,D} \subset H_2(D,\RR)
\end{equation} 
of preimages of polytopes of quasiadjunction, one for each singular
point of $D$.
\footnote{the number of global polytopes of quasi-adjunction is at
  most $\prod_{k \in Sing(D)} n(P_k)$ where $n(P_k)$ is the number of local
  polytopes of quasi-adjunction at singular point $P_k$}

b. A {\it global divisorial face of quasi-adjunction\index{global divisorial face of quasi-adjunction}} is a face $\cF$  of a polytope
$\cG_Q$ (cf. (\ref{global})) corresponding to a collection $Q$ of local polytopes of
quasi-adjunction. We say that a face $\cF$ of a global polytope of
quasi-adjunction {\it correspond to a subset}
$\cS \subset Sing(D)$ if $\cF$ is a face of a polytope determined
already by the local
polytopes of singularities only from $\cS$: $\bigcup_{P_k\in
  \cS,Q_j(P_k) \in \cQ} {pr_{P_k}}^{-1} \cQ_j(P_k)$. 

c. {\it The sheaf of quasi-adjunction\index{sheaf!of quasi-adjunction} $\cA_{Q}$ (or $\cA_{\cG(\cQ)}$) corresponding to a choice
collection $Q$ of local polytopes of
quasi-adjunction}, is the ideal
sheaf in $\mathcal{O}_X$ 
having as the stalk at $P \notin \cS$ the local ring of $P \in X$ and the
ideal of quasi-adjunction $\cA_{\cQ_j(P)}$ 
corresponding to selected local polytope of quasi-adjunction for
singularity  $P \in \cS$ 

d. {\it The homological global polytope (resp. face) of quasi-adjunction\index{homological global polytope of quasi-adjunction}\index{homological global face of quasi-adjunction}} 
is a polytope in $H^1(X\setminus D,\RR)$, viewed as the trivial coset
of $H_2(D,\RR)/H^1(X\setminus D,\RR)$ ,
 which is the translation to
this trivial coset 
\footnote{cf. construction described in Proposition \ref{datarelation}} 
 of the intersection of divisorial
global polytope (resp. face) of quasiadjunction described in a. 
(resp. b.) of this definition with a coset in
$H_2(D,\RR)/H^1(X\setminus D,\RR)$ which image in $H^2(X,\RR)$
(i.e. the image via the map $i_{\RR}$ in (\ref{dualhomologysequence})) is an
integral cohomology class (i.e. the first Chern class of a line bundle). 
\end{definition}

The following Proposition shows that in the case when irreducible
components of $D$ are big and nef, only characters of the
fundamental group, which after lift to $H^1(X\setminus D,\RR)$ give 
classes belonging to the faces of quasi-adjunction, can have non-trivial eigenspaces for
the action of Galois groups on the abelian covers of $X$ ramified
along $D$. 

\begin{proposition}\label{kvnprop}  Assume that irreducible components of $D$
  are big and nef. Let $u \in U_{X,D}$ be such that $u$ is in interior of
  all global polytopes of quasi-adjunction $\cG_Q$ of divisor $D$ or their
  images $\bar \cP_Q$ for the involution $u \rightarrow \bar u$ sending
  $(x_1,..,x_k)\in U_{D,X}$ to $(1-x_1,...,1-x_k)$. If $\tilde X_G$ is
  a resolution of singularities of a cover $X_G$ of $X$ with abelian
  Galois group $G$ and $\chi=exp (2 \pi i u)$, then the eigenspace $H^1(\tilde X_G,\CC)_{\chi}=\{
  v \in H^1(\tilde X_G,\CC)\vert g\cdot v=\chi(g)v, \forall g \in G\}$ is trivial.
\end{proposition}

\begin{proof} Since the characters of $H_1(X\setminus D,\ZZ)$ having
  a finite order are the characters $\chi=exp(2\pi i u)$ with $u\in \QQ$,
the density of those in $U_{X,D}$ implies that we can assume that $\chi$ 
 is a character of a finite
  abelian group $G$. Let $\CaL^{-1}_{\chi}$ be the corresponding line
  bundle (cf. Prop. \ref{datarelation}). Since the action of $G$ is holomorphic and hence preserves
  the Hodge decomposition of $H^1(\tilde X_G,\CC)$, after possibly replacing the
  character $\chi$ by the conjugate $\bar \chi$, we can assume that
  $\chi$ has non-trivial eigenspace for $G$ acting on the holomorphic forms
  of the cover.  
Then one has (cf. Prop. \ref{homologyquasiadjunction})
\begin{equation}\label{kvn}
     \dim H^0(\Omega^1_{\tilde X_G})_{\chi}=dim
     H^1(X,\Omega^2_X\otimes \CaL_{\chi}\otimes \cA_{\cG(Q)})
\end{equation}
 where $\cP(Q)$ is the global polytope of quasi-adjunction containing $u,
 \chi=exp(2 \pi i u)$.

Since $u$ is an interior point of $\cG(Q)$ one can take a small
perturbation of it along the intersection with $U_{X,D}$ with the coset 
of $H^1(X\setminus D,\RR)$ corresponding to $\CaL_{\chi}$  (i.e. an affine 
subspace) so that it remains inside $\cG(Q)$.  The corresponding 
line bundle $L_{\chi}$ is unchanged in this deformation of $u$. Using multiplier ideal
interpretation of ideals of quasi-adjunction and
Kawamata-Viehweg-Nadel vanishing (cf.\cite{lazarsfeld} sect. 9.4B)
we obtain that the terms in
(\ref{kvn})
are zeros. 
\end{proof}

\begin{definition} A global divisorial face of quasi-adjunction $\cF \subset \cG(\cF)$ is called {\it contributing} if for
  $u \in \cF$  and the resolution of singularities of the cyclic cover $\pi_*
   X_{\chi} \rightarrow X$ corresponding to the
  surjection $\chi: H_1(X\setminus D,\ZZ) \rightarrow Im(\chi) \subset \CC^*$
 one has $H^1(X,\Omega^2_X\otimes
 \CaL_{\chi}\otimes \cA_{\cG(F))})\ne 0$ (here $\CaL_{\chi}$ is the dual 
 $\chi$-eigenbundle of $\pi_*(\mathcal{O}_{\tilde X_{\chi}})$.  A
 homological face of quasi-adjunction is called contributing if its
 translation to a coset $H_2(D,\RR)/H^1(X\setminus D,\RR)$ 
is a contributing divisorial face. 
\end{definition}

\begin{theorem}\label{charvarposition} Let $X$ be a simply connected smooth projective
  surface and let $D=\sum D_i$ be a reduced divisor with irreducible
  components $D_i$ which are big and nef. Assuming as above that
  $H_1(X\setminus D,\ZZ)$ is torsion free, let $r=rk Coker
  H_2(X,\ZZ) \rightarrow H^2(D,\ZZ)$ denote its rank (cf. (\ref{homologycomplement})).
For any essential component of characteristic variety $\cV_i$ having 
  positive dimension i.e. a
coset of the r-dimensional torus $H^1(X\setminus D,U(1))$ such that
$dim \cV_i \ge 1$ there is: 

a) a collection of
  singularities $\cS$ of $D$ 

b) a contributing face $\cF$ of a global polytope of quasi-adjunction $\cG(\cS)$ which is determined by the collection
$\cS$ and a collection of the polytopes of
  quasi-adjunction $\cQ(P)$, one at each of singularities $P\in \cS$ 

c) a line bundle $\CaL_{\cG(\cS)}$

\noindent such that $\cV_i$ is the Zariski
closure of  $exp (\pm 2\pi i \cF)$ in the maximal compact subgroup
of the r-dimensional torus $Char H_1(X\setminus D,\ZZ)$ and $$dim \cV_i=dim \cF=dim
H^1(X,\Omega^2_X \otimes \CaL^{-1}_{\pm \cQ(\cS)}\otimes \cA_{\cG(\cS)})+1$$

Moreover, $\CaL_{\pm \cG(\cS)}$ is the line bundle which is part of the building
data of the cyclic cover corresponding to surjection $\chi: \pi_1(X\setminus D)
\rightarrow \ZZ_{ord(\chi)}$ for a character $\chi$ which is generic
in the component $\cV_i$.

Vice versa, given a maximal \footnote{i.e. not contained properly in a
  contributing face of the same strict global polytope of quasi-adjunction.} 
contributing face  $\cF\subset \cG(S)$ of a global polytope of quasiadjunction,
with the ideal of quasiadjunction $A_{\cG(\cS)}$,
such that the line bundle corresponding to the characters $exp(2 \pi
i u), u\in \cF$ is a $\CaL$ (satisfying $H^1(X,\Omega^2_X\otimes \CaL\otimes
\cA_{\cG(S)}) \ne 0$)  then the Zariski closure of the set of characters $exp(2 \pi
i u), u\in \cF$ is a component of characteristic variety of
$\pi_1(X\setminus D)$.

\end{theorem}

\begin{proof}  Let $Q$ be a maximal contributing face of quasi-adjunction.
The Zariski closure in $H^1(X\setminus D,\CC^*)$ of the set $exp(2 \pi
i u), u \in Q$, belongs to 
a component of characteristic variety, as follows from the
assumptions. If this Zariski closure is a proper subset of a component, 
then preimage of the unitary part (i.e. the intersection with
$H^1(X,U(1)) \subset H^1(X,\CC^*)$) of the full component must belong to the same $H^1(X\setminus
D,\RR)$ coset in
$H_2(D,\RR)$ as $Q$ and,  as follows from Prop. \ref{kvnprop}, 
its preimage in $H_2(D,\RR)$ 
 must be a face of the same polytope as $\cF$ i.e. coincide
with $\cF$ due to maximality assumption. 

Now, let $\cV_i$ be an irreducible component of characteristic
  variety.  A theorem of D.Arapura implies that $\cV_i$ is a translated
  subtorus of the torus  $H^1(X\setminus D, \CC^*)$. The subset
  $\exp^{-1}(\cV_i \cap H^1(X\setminus D,U(1))  \subset H^1(X\setminus
  D,\RR)$ consist of a set of $H^1(X\setminus D,\ZZ)$-translates of
  a linear subspace of $H^1(X\setminus D,\RR)$. The eigenbundles of
  the characters in $\cV_i$, for the push forward of the structure sheaf
   of a cyclic cover of $X$ corresponding to characters from $\cV_i$, define a collection of translates of
  $H^1(X\setminus D,\RR)\subset H_2(D,\RR)$ (cf. Prop.
  \ref{datarelation}) 
which intersect the fundamental domain (i.e. the unit cube) 
  $U_{X,D}$ for the action of $H_2(D,\ZZ)$ on
  $H_2(D,\RR)$. Due to identification in
  Prop. \ref{homologyquasiadjunction} of cohomology of the local
  systems and the cohomology of sheaves of quasi-adjunction,
  one obtains that at least one of translates belongs to a contributing
  face of quasi-adjunction. It is maximal since otherwise $\cV_i$ will
  be a proper subset of a component of larger dimension.

\end{proof}

\begin{corollary}\label{genericline} Let $X,D$ be as in theorem \ref{charvarposition}
and let $C$ be a smooth big and nef curve intersecting  
all irreducible components of $D$ at smooth points transversally.
Then $H_2(D,\RR) \subset H_2(D+C,\RR)$ has codimension one and 
divisorial contributing faces of quasi-adjunction of $D$ coincide with
those of $D+C$.  
\end{corollary}

\begin{proof} Since polytope quasiadjunction of ordinary node coincides
  with the unit square (node does not impose conditions of quasi-adjunction)
it follows that the global polytopes in $H_2(D+C,\RR)$ are the
cylinders over the global polytopes of $H_2(D,C)$ (preimages of
projection of $H_2(D+C,\RR)$ onto the later).
Kawamata-Viehweg-Nadel vanishing implies that the characters in
a contributing faces of the
eigenbundles $\CaL_{\chi}$
must have trivial ramification along $C$ i.e. belong to $H_2(D,\RR)$
(triviality of ramification also follows from Divisibility Theorem
\ref{theoremdivisibility}). 
\end{proof}

\begin{remark} The removal a line at infinity, transversal to a curve, 
was used extensively in \cite{mealex}, \cite{mecharvar}. 
The main theorem in \cite{mecharvar} follows immediately from 
Theorem \ref{charvarposition} and Corollary \ref{genericline}.
\end{remark}

Numerous examples to the Theorem \ref{charvarposition} can be found in the paper
\cite{mecharvar} in the case of line arrangements is a plane and in
(\cite{arcata81}) in the case of irreducible curves. The local
counterpart of the Theorem \ref{charvarposition} and many  
examples of calculations of multivariable Alexander polynomials of the links
(i.e. the characteristic varieties, cf. discussion after Def. \ref{defdepth}) of
singularities in terms of polytopes and ideals of quasi-adjunction are
given in \cite{pierretteme11}. For  results on zero dimensional components of
characteristic varieties we refer to \cite{artalorbifolds}, \cite{artalcogolludome}.
We will finish this section with an example of calculation on a large
class of surfaces generalizing 6-cuspidal sextic of Zariski.

\begin{example}\label{sectionsofbundles} 
Let $X$ be a smooth projective simply connected
  surface and let $L$ be a very ample line bundle on $X$. Let $s_2 \in
  H^0(X,L^2), s_3\in H^0(X,L^3)$ be generic sections of the
  corresponding tensor powers 
of $L$. Let $D$ be the zero set of $s=s_2^3+s_3^2 \in H^0(X,L^6)$. 
Then the Alexander polynomial of this curve with $6L^2$ cusps, 
corresponding to the surjection $H_1(X\setminus D,\ZZ)\rightarrow \ZZ_6$, 
is $t^2-t+1$. 

To see this, first note that the existence of the 
surjection follows from (\ref{homologycomplement})  since the class of $D$ in $H_2(X,\ZZ)$ is
divisible by $6$. Using (\ref{homologyquasiadjunctionformula}),
the eigenspace of the generator of $\ZZ_6$ acting on $H^{1,0}$
of the 6-fold cyclic can be identified with $H^1(X,K_x\otimes L^5
\otimes \mathcal{I}_{Sing})$ where $\mathcal{I}_{sing}$ is the ideal sheaf such that
$\mathcal{O}_X/\mathcal{I}_{sing}$ is the reduced 0-dimensional subscheme of $X$ with
support at the set of cusps of $D$. One has the following Koszul resolution\index{Koszul resolution} of $\mathcal{I}_{sing}$:
$$0 \rightarrow L^{-5} \rightarrow L^{-2}\oplus L^{-3} \rightarrow
\mathcal{I}_{sing} \rightarrow 0$$
After taking the tensor product of this sequence with $K\otimes L^5$
and considering the
corresponding cohomology sequence: 
\begin{equation}\label{kodairavanishing}
H^1(X,K_X \otimes L^2)\oplus H^1(X,K_X \otimes L^3) \rightarrow 
H^1(X,K_x\otimes L^5
\otimes \mathcal{I}_{Sing}) \rightarrow H^2(X,K_X) \rightarrow 0
\end{equation}
we see that Kodaira vanishing implies that
 $dim H^1(X,K_x\otimes L^5
\otimes \mathcal{I}_{Sing})=1$. 
This shows that ${1 \over 6}
\in [0,1]$ is the contributing face of quasi-adjunction 
and now the claim about the Alexander polynomial follows from the 
Theorem \ref{charvarposition}. Note that this example also can be
analyzed using methods of orbifold pencils discussed in \cite{artalorbifolds},\cite{artalcogolludome},\cite{acme2}.
\end{example}.

\subsection{Bernstein-Sato ideals\index{Bernstein-Sato ideals} and polytopes of quasi-adjunction}
Let $f_1,..,f_r$ be germs of holomorphic functions in $n$ variables. 
The Bernstein-Sato ideal $\cB(f_1,...,f_r)$ is the ideal generated by polynomials 
$b(s_1,....,s_r)$ such that there exist a differential operator $P \in
\CC[x_1,...,x_n,{\partial \over {\partial x_1}},....,{\partial \over {\partial x_1}},s_1,...,s_r]$
satisfying the identity:
\begin{equation}
  b(s_1,....,s_r)f^{s_1}....f^{s_r}=Pf_1^{s_1+1}....f_r^{s_r+1}
\end{equation}
(cf. \cite{sabbah} \cite{bahlul}, \cite{maynadier}, \cite{gyoga}, in 1-dimensional
case cf. \cite{malgrange}, \cite{kashiwara}).
In the case of plane curves singularities one has the following: 
\begin{theorem} Let $f_1,..,f_r$ be the germs of holomorphic functions in two variables. 
Let $P$ be the product of the linear forms $L_i(s_1 + 1,...,s_r + 1)$ where
$L_i$ runs through linear forms vanishing on $r-1$-dimensional faces of polytopes of
quasi-adjunction corresponding to a germ with r irreducible components $f_1,...,f_r$.
Then any $b \in \cB(f_1,...,f_r)$ is divisible by $P$.
\end{theorem}

The same argument as used in \cite{pierretteme11}, provides extension
to isolated non-normal crossings (cf. \cite{isolatednonnormal}).  
For a general conjecture of the structure of Bernstein ideals we refer
to \cite{budur15} and for a discussion of the case of arrangements,
other related problems and references 
cf. \cite{walter}


\section{Asymptotic of invariants of fundamental groups}\label{asymptsection}

The problem of characterization of fundamental groups of smooth
quasi-projective varieties is intractable at the moment. Nevertheless
some questions about distribution of Alexander invariants can be addressed. We will see below that one can make some conclusions about
distribution of dimensions of characteristic
varieties of such fundamental groups. A different type of asymptotics,
is suggested by the relation between the degrees
of Alexander polynomials and Mordell-Weil ranks of isotrivial
families of abelian varieties
(cf. \cite{jcogomecrelle},\cite{mathannalen}) since it allows to restate the
problem of asymptotic behavior of such degrees in terms of the conjectures on 
distribution of Mordell Weil ranks of curves over the function fileds. 
In this section we shall survey the
results in \cite{asymptotics} concerning distribution of the dimensions of
characteristic varieties \footnote{Such circle of problems is inspired by conjectural
asymptotic of number fields extensions having a given group as the
Galois group or the group of its Galois closure, which are unramified outside an
arbitrary subset of primes while the size of the norm of discriminant 
grows (\cite{malle}): Malle conjectures implies a positive
  answer to the inverse problem of the Galois with little hope for
  solution in near future (as is obtaining a characterization of quasi-projective group)}.   

Let $X$ be a smooth simply connected projective variety, 
$D$ a reduced divisor and let $\Delta$ be a subset of the effective cone
$Eff(X) \subset NS(X)$ in the Neron Severi group of $X$. We shall call
the set $\Delta$ {\it saturated}\index{saturated set} if $d_1 \in \Delta$ and $d_2\in Eff(X)$ are such that
$d_1-d_2 \in Eff(X)$ implies that $d_2\in \Delta$ and $d_1-d_2\in \Delta$. We are
interested in  distribution of invariants of $\pi_1(X\setminus D)$
when the class of $D \subset Eff(X)$ is a linear combination of classes
in $\Delta$ with non-negative coefficients. We are specifically
interested in curves $D$ with large dimension of a component of characteristic variety
of $\pi_1(X\setminus D)$ and $D$ being a curve with all its irreducible
components having classes in $\Delta$. It follows from \cite{arapura}, that
existence of a component of
dimension $r$ implies existence of surjection $\pi_1(X\setminus
D)\rightarrow F_r$. Vice versa, existence of the latter implies
that the characteristic variety of $\pi_1(X\setminus D)$ contains a
component of dimension not smaller than $r$. Note right away that 
for the purpose of enumeration of reduced divisors $D$ for which one has 
a surjection $\pi_1(X\setminus D) \rightarrow F_r$ we must impose 
some conditions on such surjections. For example, given any $D$ with such
property and any reduced divisor $D'$ one has 

\begin{equation}\label{essential}\pi_1(X\setminus D\bigcup D')
\rightarrow \pi_1(X\setminus D) \rightarrow F_r
\end{equation}
and hence, given a curve admitting a surjection of its fundamental group onto
$F_r$, there are enlargements of this curve with the same
property parametrized by all the curves on the surface. This
motivates the following: 

\begin{definition}\label{esssurj} (cf. \cite{mecharvar})  Let $\CD$ be a reduced divisor on a smooth projective surface $X$. 
A surjection $\pi_1(X\setminus \CD)\rightarrow F_r$ is called {\it essential}
if $\CD$ does not admit split $\CD=D\bigcup D'$ for which one has
factorization (\ref{essential}).

A surjection $\pi_1(X\setminus \CD) \rightarrow F_r$ is called {\it reduced}
if there exist a choice of ordered system of generators
$\{x_1,...,x_{r+1} \vert x_1 \cdot....\cdot x_{r+1}=1\}$ of $F_r$ such that
this surjection takes meridian of each irreducible component of
$\CD$ to a conjugate of a generator.

We also will say that singularities of $\CD$ satisfy condition (*) if 
all singular points belonging to more than one irreducible component
are ordinary singularities i.e. are intersections of smooth
transversal branches \footnote{ The results in this section make this
  assumption. It should be possible to eliminate it with essential
  conclusions remaining intact.}.
\end{definition}

A rather detailed information about such curves was obtained
in \cite{mesergey} in the case $X=\PP^2, \Delta=\{[1]\} \in
\ZZ=Pic(\PP^2)$ i.e. the fundamental 
groups of the complements to arrangements of lines in a plane 
(see \cite{falkyuz},\cite{marco} for related results).
                    
\begin{theorem}\label{concurrent} \cite{mesergey} \cite{pereira} Let $\cA$ be an
  arrangement of lines in $\PP^2$. If there exist an essential
  surjection $\pi_1(\PP^2\setminus \cA)\rightarrow F_r, r \ge 4$
then $\cA$ is a union of concurrent lines, in which case the last
surjection is an isomorphism. 
\end{theorem}

Moreover, there is only one known example of essential surjections of 
the complements to an arrangment line which admits surjection onto
$F_3$ \footnote{i.e. the Hesse arrangement of $12$ lines formed by 
lines containing triples of inflection points of plane smooth cubic
cf. \cite{mecharvar}} and for any $d$ there exist an arrangement of non-concurrent lines
admitting essential surjection onto $F_2$ (e.g. $3d$ lines forming the
zero set of $(x^d-y^d)(y^d-z^d)(x^d-z^d)=0$).

Work \cite{asymptotics} contains an extension of this theorem to
reduced divisors on arbitrary simply connected surfaces. Before stating the
main result, let us describe the analog of the case of concurrent lines 
in Theorem \ref{concurrent}, which is a family of the curves with irreducible
components in $\Delta$ and for which the fundamental group of the
complements may have a free quotient of arbitrary large rank. 
For this family of curves, the fundamental groups of the complement
form {\it a finite} set of groups, having cardinality depending on $\Delta$ and, 
moreover, a presentation of each group in this set can be described in terms of
geometric data we specify. However, unlike the case of Theorem \ref{concurrent}, the
problem of characterizing which specific data is realizable by
curves
in this class remains open in general.
Enumeration of
fundamental groups of such curves for a class $\delta \in
\Delta \subset Pic(X)$ can be made as follows.

\begin{proposition}\label{groupsinpencils} For any $r \ge 1$ and a movable divisor\index{movable divisor} $\delta \in Pic(X)$, \footnote{i.e. such that the
codimension of the base locus of the linear system it defines is at
least 2} there is a divisor $D$ with classes of components in the linear
system of $\delta$
and such that $\pi_1(X\setminus D)$ admits essential surjection\index{essential surjection} onto
$F_r$. Vice versa, if $D$ has all its irreducible
components being members of a pencil of curves in complete linear
system of $\delta$ (i.e. a line in $\PP(H^0(X,\mathcal{O}_X(\delta)))$),
 and $\pi_1(X\setminus D)$ admits surjection onto $F_r, r\ge
 2$  this group is an amalgamated product
$\cG*_{F_a}\CH$ with $\cG$ belonging to a finite collection of groups depending on
$\delta$, obtained by a construction below and
$\CH$ is an extension:
\begin{equation}
 0 \rightarrow F_a \rightarrow \CH \rightarrow F_{r'} \rightarrow 0 \ \
 \ r'\le r
\end{equation}
defined by a homomorphism $F_{r'} \rightarrow Aut(F_a)$ coming from a
finite set cardinality depending only on $\delta$. The number of
isomorphism classes of such groups $\pi_1(X\setminus D)$ with a fixed
class $\delta$, stabilises for large $r$.
\end{proposition}

\begin{proof} Indeed, for any pencil in the linear system containing
  $\delta$, a union on its $r+1$ members yields a divisor $D \in
\PP(H^0(X,(r+1)\delta))$
with $\pi_1(X\setminus D)$ admitting a surjection onto $F_r$ since 
such a pencil induces a dominant map onto the complement in $\PP^1$ to
$r+1$ points. 

To enumerate all possible fundamental groups of the
complements to the curves with all irreducible components belonging to
a pencil let us consider the discriminant $Disc(\PP(H^0(X,\mathcal{O}_X(\delta))))$ of the complete linear  system
$\PP(H^0(X,\mathcal{O}_X(\delta))$ i.e. the subvariety consisting of 
the divisors having singularities worse than singularities of a generic element
in $\PP(H^0(X,\mathcal{O}_X(\delta))$. Consider also the stratification of 
the discriminant into connected components of equisingularity strata, 
adding to this stratification the complement to the discriminant as
a codimension zero stratum (cf. \cite{aluffi} on some information about
geometry of these strata).

We will use finite sets of collections of such equisingularity strata\index{equisingularity!strata}
$\cS_1,...\cS_t$ for which there  exists a
pencil $\cP$ in $\PP(H^0(X,\mathcal{O}_X(\delta)))$ with the following property: there exists a union $D$
of members of $\cP$ such
that the curve $D$ satisfies condition (*) (cf. Def. \ref{esssurj}). Let $N(\delta,t)$ be the
number of isotopy classes of pencils in $\PP(H^0(X,\mathcal{O}_X(\delta))$ such that the number of the strata of this stratification 
intersected by the pencil is $t$ and let $T$ be the least upper bound for integers $t$ for all pencils in $\delta$.
Finiteness of these numbers is a consequence of the finiteness of
the number of strata of stratifications since those are an algebraic subsets of discriminant
(cf. \cite{goresky}).

Let $D$ is a curve having $r+1$ irreducible components belonging to a pencil $\cP$ in $\PP(H^0(X,\mathcal{O}_X(\delta))$ in which the members of $\cP$ have $t$ (where $t\le
r+1$) equisingularity types. We claim that $\pi_1(X\setminus D)$,
can have at most $2^t$ isomorphism types. More precisely for each
subset $\cT$ of the set of strata $\cS_1,..,\cS_t$ there is at most one
isomorphism type of the fundamental groups $\pi_1(X\setminus D)$ where
the set of equisingularity starta of components of $D \in
\PP(H^0(X,\mathcal{O}_X((r+1)\delta)))$
 having non-generic equisingularity type 
in the pencil coincides with $\cT$. This is the case when $D$ is a
union of $\vert \cT \vert$ curves from the strata $\cS_1,...,\cS_t$ and
$r'=r+1-\vert \cT \vert$ curves
from codimension zero stratum and none of remaining $t-\vert \cT \vert$ singular
members of the pencil are not components of $D$.
In particular, for $r>t$ there are at
most $\sum_{t=0}^T 2^tN(\delta(t))$ isomorphism classes of the
fundamental groups and for $r>T$ the number of isomorphism classes of fundamental groups
of curves with components in the linear system of $\delta$ and admitting surjections onto
$F_r$ is bounded, with bound depending only on $\delta \in \Delta$.

To describe the structure of the fundamental groups of the complement
to a union $D$ of several members of a pencil $\cP$ of curves in $\delta$, with
the set equisingularity types of singular members of $\cP$ consisting
of equisingularity strata $\cS_1,..,\cS_t$, such that
non-generic types of components of $D$ are exactly those in $\cT$, and also to
enumerate such fundamental groups,   
consider the blow up $\tilde X$ of $X$ at the
base points of the pencil.  We obtain a regular map $\pi: \tilde X \setminus
\tilde D \rightarrow \PP^1\setminus S_{r+1}$ where $S_{r+1}$ a finite
subset of $\PP^1$ with cardinality $r+1$. 

Let $\PP^1=B_1 \cup B_2$ be partition into union of two disks intersecting along their common boundary
and having the following properties: $B_1$ contains all $t-\vert
\cT \vert$ fibers of $\pi$ which do not have
generic equisingularity type in $\PP(H^0(X,\mathcal{O}_X(\delta))$ and are not components of $D$,
while $B_2=\PP^1\setminus B_1$ contains $\vert \cT \vert$ non-generic fibers if $\pi$
which are components of $D$ and remaining $r'=r+1-(t-\vert T \vert)$ fibers of
$\pi$ which all are generic in the latter
linear system. Over the complement in $B_2$ to the subset over which
the fibers of $\pi$ are the components of $D$, the map $\pi$ is a
locally trivial fibration which global type is determined by $\delta$.  
Van Kampen theorem \ref{vankampen} implies the following: 
if $\Sigma$ is generic fiber of $\pi$, $\cG=\pi_1(\pi^{-1}(B_1)),
\CH=\pi_1(\pi^{-1}(B_2))$ then
\begin{equation}\label{amalgam}
  \pi_1(X\setminus D)=\cG*_{\pi_1(\Sigma)}\CH, \ \ \ \ and \ \  1\rightarrow
  \pi_1(\Sigma) \rightarrow \CH \rightarrow F_{r'} \rightarrow 1
\end{equation}
$\Sigma$ is complement in the generic fiber of the pencil to the set
of base points of the pencil i.e. $\pi_1(\Sigma)$ is a free group
$F_a$ for some $a$. The group $\cG$ belongs to a collection having at
most $2^{t}$ elements (i.e. the number of subsets in $\cS_1,....,\cS_t$). The claim follows.
\end{proof}
\begin{example}\label{quadrics} Let us enumerate the 
  fundamental groups of the complements to conic-line arrangements 
which admit  a surjection onto a free group of rank greater than 5.
The starting point is that a conic-line arrangement (satisfying
condition (*)) having such fundamental group is a union of $r+1$ (possibly
reducible) quadrics belonging to a pencil. This is content of
improvement for conic-line arrangements of the general bound in
Theorem \ref{maintheoremasympt} below (cf. Example
\ref{asymptoticsexample} 2.) 
Equisingular stratification of $\PP(H^0(\PP^2,\mathcal{O}_{\PP^2}(2))$
consists of 3 strata: smooth quadrics, reduced and reducible quadrics
i.e. a union of two transversal lines
and non-reduced quadrics i.e. the double lines. The degree of discriminant is
$3$. We denote these equisingular strata respectively as
$\cS_0,\cS_1,\cS_2$. 

Any pencil of quadrics containing as generic element a smooth
quadric in $\cS_0$, has at most 3 singular fibers which
are either 3 reducible quadrics or contains 2 singular fibers one
reduced and one non reduced. In the latter case, the condition (*)
on $D$ fails. Moreover, there are pencils with generic element inside the
stratum $\cS_1$. For such a pencil, the divisor $D \in \PP(H^0(\PP^2,\mathcal{O}(2(r+1)))$
is a union of $2r+2$ concurrent lines and hence $\pi_1(\PP^2\setminus
D)=F_{2r+1}$. 

There are $4$ equisingular classes of divisors $D \in
\PP(H^0(\PP^2,\mathcal{O}((r+1)2)))$
with components formed by curves in a pencil $\delta$, corresponding to the cases when 
the number of quadrics which are the singular elements of the pencil
and formed by  components of $D$, is either $0$ (i.e. all components of $D$ are
smooth quadrics), or is $1,2$ or $3$. Respectively, there are 4
corresponding types of fundamental groups.

For example, let us take as 
$D$ a union of $r+1$ quadrics belonging to a pencil, one of which is
reducible. Let $B_1$ be a disk containing remaining 2
reducible fibers of the pencil and let $B_2=\PP^1 \setminus B_1$. Then
$B_2$ is a disk containing the points 
corresponding to the fibers containing the components of the pencils
comprising $D$.
 Over $B_1$, the map
$\pi$ is a fibration with generic fiber being a smooth quadrics and
which has two special fibers which are the union of lines and 
therefore can be calculated using van Kampen theorem \ref{vankampen}.
Over the complement in $B_2$ to the points corresponding to the components of
$D$ one has a locally trivial fibration with the fiber being the
complement in a smooth
quadric to $4$ base points of the pencil.  Hence $\CH=\pi_1(B_2)$ is an
extension of free group $F_3$ by the free group $F_r$
with only one type of extension possible since there is only one
isotopy class of generic pencils of quadrics.
\end{example}

Now we turn to the main result of \cite{asymptotics} which can be stated as follows:

\begin{theorem}\label{maintheoremasympt} 
 Given a saturated set\index{saturated set} $\Delta$ of
  classes in $NS(V)$ consider the following trichotomy for the
  distribution of the curves $\CD$ with classes of irreducible components in
  $\Delta$ having a free essential reduced quotient of a fixed rank $r$ and satisfying
  conditions (*)
\begin{enumerate}[label=\arabic*)]
 \item\label{corolmain0:A}
 There exist infinitely many isotopy classes of curves $\CD$ admitting surjections $\pi_1(V\setminus \CD)                                
\rightarrow F_r, , r>1$.
\item\label{corolmain0:B} There are finitely many isotopy classes of curves $\CD$ admitting
surjections $\pi_1(V\setminus \CD) \rightarrow F_r, r>1$.
\item\label{corolmain0:C}
 $\CD$ admitting
a surjection $\pi_1(V\setminus \CD) \rightarrow F_r,$ is composed of
curves of a pencil. There are finitely many isotopy classes of such
$\CD$ for given $\Delta$.
\end{enumerate}
All three cases are realizable at least for some $(V,\Delta)$.  Case
\eqref{corolmain0:B}
takes place for $r \ge 10$. There
exists a constant $M(V,\Delta)$ such that for $r>M(V,\Delta)$ one has case \eqref{corolmain0:C}.
In the latter case, $\pi_1(V\setminus                                                                                             
\CD)$ splits as an amalgamated product $H*_{\pi_1(\Sigma)}G$ where
$\Sigma$ is an open  Riemann surface which is a smooth member of the pencil,
 $H$ is coming from a finite set of
groups associated with the linear system $H^0(V,\mathcal{O}(D))$,
$D$ is a divisor having class $\delta \in \Delta$ and $G$ is an extension:
\begin{equation}\label{extension}0 \rightarrow \pi_1(\Sigma) \rightarrow G \rightarrow F_r
\rightarrow 0
\end{equation}
\end{theorem}


In specific cases of $(X,\Delta)$ information about the constants $10$
and $M(X\Delta)$ can be improved. 
\begin{example}\label{asymptoticsexample} 
\begin{enumerate}
\item  Above results for the 
  arrangements of lines
shows that in this case one can replace $10$ by $2$ and
$M(\PP^2,[1])=3$. 

\item Again in the case $X=\PP^2$ but $\Delta=\{[1],[2]\}$, the
  curves $\CD$ for which there exist a surjection $\pi_1(\PP^2\setminus
  \CD)\rightarrow F_r$ must  have the type only as
  described in  Example \ref{quadrics}, provided $r>5$.  However, a generic pencil in the 
linear system: 
\begin{equation}\label{ruppert2}
 \lambda_0x_0(x_1^2-x_2^2)+\lambda_1 x_1(x_2^2-x_0^2)+\lambda_2x_2(x_0^2-x_1^2)=0
\end{equation} 
has 6 members which are unions of lines and quadrics. This gives a
curve $\CD$ of degree 18 for which $\pi_1(\PP^2\setminus \CD)$  admits a
surjection onto
$F_5$ and is not isotopic to a curve as in \ref{quadrics}.
\end{enumerate}
\end{example}

The Theorem \ref{maintheoremasympt} can be restated as follows: if $N(X,\Delta,r)$ denotes the number of
equisingular isotopy classes of curves on $X$ with irreducible
components having numerical classes in $\Delta$ and fundamental groups
admitting a surjection 
onto a free group $F_r$ then for $r >M(X,\Delta)$,  $N(X,\Delta,r)$ is
finite and all curves have special type as in the case ({\rm C}) of
the trichotomy.  For $10 < r \le M(X,\Delta)$,
$N(X,\Delta,r)$ is also finite but the type of the curves may vary. Finally, for
$r<10$ the number of isotopy classes $N(X,\Delta,r)$ may be infinite. 

Some information on dependence of the constant $M(X,\Delta)$ on $\Delta$ and $X$ is also
available. For example if $X=\PP^2,\Delta_d=\{[1],...,[d]\}$ then 
$M(\PP^2,\Delta_d)\ge 3d$. Indeed, Ruppert (cf.\cite{ruppert})
found a pencil of curves of degree $d+1$ with $3d$ fibers 
being a union of a line an a curve of degree $d$. In particular a
union of these $3d$ fibers yields a curve of degree $3d(d+1)$ with
irreducible components in $\Delta_d$ and having surjection on the free
group of rank $3d-1$. In particular the sequence $M(X,\Delta_d)$ is
unbounded. The Ruppert pencil is a generic pencil in 2-dimension
linear system of curves given by equation (which for $d=2$ it is given
in Example \ref{asymptoticsexample}):
\begin{equation}\label{ruppertd}
  \lambda_0x_0(x_1^d-x_2^d)+\lambda_1 x_1(x_2^d-x_0^d)+\lambda_2x_2(x_0^d-x_1^d)=0
\end{equation}
More precisely, the curve (\ref{ruppertd}) is singular if and only if 
$$(\lambda_0^d-\lambda_1^d)(\lambda_1^d-\lambda_2^d)(\lambda_2^d-\lambda_0^d)=0$$
and all reducible fibers are unions of a line and a curve of degree $d$.
Hence generic line in variables $\lambda_i$ is a pencil with $3d$
reducible members 
as described.

We refer to \cite{asymptotics} for examples of surjections onto free
groups of the
fundamental groups of the complements to curves on surfaces besides
$\PP^2$. 

This discussion suggests the following problems:

\begin{problem} 
\begin{enumerate}
\item Determine the rate of growth of $N(X,\Delta,r)$ 
for various $X$ and $\Delta$ when $r \rightarrow \infty$, i.e. how many
types of reducible curves admitting surjections onto $F_r$, which $r$
large (i.e. $r>M(X,\Delta)$) exist?

\item  Find a bound on $M(X,\Delta)$ in terms of invariants of $X,\Delta$
i.e. how large should be $r$ such that there exist curves admitting 
surjection onto $F_r$ and which are not the unions of the fibers of
a pencil.

\item  For $n \in \NN$ let $\Delta_n=\{ \sum n_i\delta_i \vert \delta_i\in
\Delta, n_i \le n\} \subset NS(X)$.
Determine the asymptotic of the number of curves admitting
surjection onto $F_r, r<10$ with the classes in $\Delta_n$ when
$n\rightarrow \infty$. 

\item Determine algebraic properties of the fundamental groups described
in Proposition \ref{groupsinpencils}. 
\end{enumerate}  
\end{problem} 

Some partial results, mainly in the case of plane, are discussed above
and in \cite{asymptotics}: for example  the curves $(x^n-y^n)(y^n-z^n)(x^n-z^n)=0$
formed by $3n$ lines show that the growth in Problem 3 for $\PP^2,[1]$
for $r=2$ is at least linear. The growth of $N(X,\Delta)$ appears to be related
to the asymptotic of the number of strata (cf. the proof of
Prop. \ref{groupsinpencils}) and possibly is exponential.

\section{Special curves}

This section surveys examples of calculations of the fundamental
groups and other topological information about the
complements, the properties of fundamental groups and applications. 
An important step in each such inquiry is finding a class 
of curves with interesting topology of the complements.
Most examples in this section are plane curves.

\subsection{Arrangements of lines, hyperplanes and plane curves}. 

There are many calculations of the fundamental groups of the
complements to arrangements
of lines. The braid monodromy can be calculated algorithmically. In the
case of real arrangements finding the braid monodromy and
the presentation are particularly simple: see \cite{salvetti},
\cite{hironaka}. In some instances this leads to presentations allowing
a more intrinsic characterization: for example in \cite{fan} conditions on
arrangement were found for the fundamental groups
to be products of free groups.  

The fundamental groups and more subtle questions on the topology of 
the complements to arrangements formed by hyperplanes fixed by the groups generated by
reflections were very actively studied in many case. In case of real
reflection groups, the fundamental groups of the complements to
corresponding complexified real arrangements were found \cite{brieskorn} with
presentations closely related to the Dynkin diagrams of the corresponding Coxeter
groups. The topology of the complements to hyperplanes corresponding to
the complex reflection groups also were actively studied with many
deep results.  The number of striking
results is too large to survey here and we refer for example to
\cite{broue} and \cite{bessis} for some particularly important ones and for further references.

Several calculations were made for the fundamental groups of the complements to
unions of lines and quadrics. Work \cite{quadricline} includes the
arrangements formed by unions a quadric and lines with various
tangency conditions. Few example of such type of arrangements, more specifically those real arrangements of
quadrics and lines which admit projections to  
a line with all critical points being real, were considered in \cite{namba}.
Here the standard methods of calculation of the braid monodromy are
almost as simple as in the case of real arrangements of lines and lead quickly to presentations in terms generators and relators. 

Cardinality of the set of connected components of the equisingular
families of reducible curves with fixed 
combinatorial type (cf. Definition \ref{pairs}) was 
investigated in several cases of plane curves of small degree.
In particular the classification for curves of degree 5 was carried out in \cite{degt90}.
The case of arrangements of  small cardinality and irreducibility of equisingular component was studied for
arrangements up to 9 lines as well as arrangements of 10 and 11 lines
with many different types of combinatorics with some results in the case
of arrangements of 12 lines (cf. \cite{moduli10lines},
\cite{fei}, \cite{yoshinaga12},\cite{artalconjugate}, \cite{benoit1}
 the latter are in 
connection with Rybnikov's example of combinatorially equivalent
arrangements with distinct homotopy types). Specific types of
presentations of the fundamental groups of arrangements were studied
in \cite{conjfreepresent}.



\subsection{Generic Projections\index{Generic Projections}} Study of the fundamental groups of
the complements to the branching curves of
generic projections \footnote{important results on geometry of such curves were
  obtained much earlier by italian school, notably B.Segre, Chisini and
  his school cf. \cite{segre}}
 was initiated by B.Moishezon in work \cite{moishezon} and
continued jointly with M.Teicher and later by M.Teicher and
her collaborators. Given a smooth surface $X \subset \PP^N$, a
projection from a generic $\PP^{N-3}$ gives a generic branched cover
ramified along a curve  $R\subset X$. The image of $R$ is the
branching curve $B \subset \PP^2$ 
of this projection.  
If the center of projection $\PP^{n-3}$ is sufficiently generic, then $B$ has nodes and cusps as the only singularities. 
The number of cusps and nodes can be found in terms of intersection indices of Chern
classes of $X$ and the class of hyperplane section (cf. \cite{bowdin}).
Work \cite{moishezon} 
considers the case when $X$ is a smooth surface in $\PP^3$. Then the branching curve $B$ has degree $n(n-1)$, $n(n-1)(n-2)$ cusps and 
${1 \over 2}n(n-1)(n-2)(n-3)$ nodes (for $n=3$ one obtains sextic with
six cusps). The fundamental group of the
complement is isomorphic to the quotient of the braid group on $n$
strings by its center (cf. \cite{moishezon}). 

Works \cite{moishteicher} consider generic projections of quadrics
$X=\PP^1\times \PP^1$ using a family of embeddings $i_{a.b},a,b \in \ZZ$ corresponding to
various ample divisors in $NS(X)$. Interest in this class stems form
the fact that Galois covers of $\PP^2$ with branching curve of generic
projections of these surfaces provide examples of simply connected surface of general
type for which $c_1^2>2c_2$.  The key step in the showing the simply
connectedness is the calculation of the fundamental group of the
complement to the branching curve. The relation between the
fundamental groups of the complements to the branching curves of
generic projections and the fundamental groups of smooth models of
Galois closures of these projections is discussed in
\cite{moishteicher}, \cite{liedtke}. 

Since then, the class of surfaces which generic projections produces the
curves for which one has a presentation of the fundamental groups of
the complements was greatly increased. Calculations produced over the
span of more than 30
years include complete intersections in projective spaces (\cite{robb}), very ample embeddings of Hirzebruch surfaces, embeddings
of K3 surfaces, very ample embeddings of ruled surfaces which are the products of $\PP^1$ and
smooth curves of positive genus and others. In many instances a quite
different than in the case of surfaces in $\PP^3$ pattern emerged for
the fundamental groups (cf. \cite{minageneric} for references to these
calculations).
One has to mention that
the main technical tool in such calculation is appropriate degeneration
of the surface resulting in degeneration of the branching curve. Steps
of calculation include calculation of the braid monodromy of
degenerate curve (which may be reducible) and then applying rules of
regeneration i.e. relating the braid monodromy of degenerate curve to
the braid monodromy of the curve prior to degeneration.
We refer to a survey article  \cite{teicher11}
which has useful references to these numerous calculations. 

An interesting property of branching curves of generic projections
was discovered by Chisini: (with a small number of exceptions)
the cover given by generic projection is determined by the 
curve alone, i.e. no subgroup of the fundamental group to specify 
the cover (cf. section \ref{branchedcoverssect}) is needed. A proof 
of this result was found in \cite{kulikov} (cf. also,
\cite{catanese}).

\subsection{Complements to discriminants\index{complements to discriminants} of universal unfoldings}
With a germ of isolated hypersurface singularity $f(x_1,...,x_n)=0$ one
associates the germ of the universal unfolding $\CC^N$, 
$N=dim \CC[x_1,....,x_n]/(f,{{\partial f}\over {\partial x_1}},....,
{{\partial f}\over {\partial x_i}})$ which comes with  the germ of discriminantal hypersurface
$Disc$ (corresponding to the germs having a critical point (cf. \cite{greuel})). 

The fundamental groups of
the complements to the germs $Disc$ have appearance in a variety of
questions  spreading from singularity theory and topology to representation theory
and beyond. An important feature of the fundamental groups of such
complements (as well as complements to other discriminants) is that
they come
 endowed with geometric monodromy i.e. the homomorphism
to the mapping class group of the Milnor fiber, i.e. the
group of diffeomorphisms of the Milnor fiber constant on its boundary modulo
isotopy.  This induces the homological monodromy via the action of the
mapping class group on the homology of the Milnor fiber. For $ADE$
singularities one obtains the corresponding Coxeter groups
(cf. \cite{ebeling}).  Moreover, these complements to germs often can be identified with the 
complement to the whole affine hypersurfaces in $\CC^N$, so these
local fundamental groups are quasi-projective. 
In the case of simple $ADE$ surface singularities, 
the fundamental groups of the complement were identified by Brieskorn 
(cf. \cite{brieskorn}) with the braid groups corresponding to the
respective Coxeter systems.

Calculations for several more complicated
classes of singularities were made also. An important case of Brieskorn-Pham polynomials
$f(x_1,...,x_n)=x_1^d+...+x_n^d$ was considered by M.Lonne (cf. \cite{lonne09}
and references there). Generators and relations of the fundamental group of the
complement to discriminant are described in terms of combinatorial data
given by the graph associated to singularity, analogous to Dynkin
diagram or, equivalently, in terms of the corresponding bilinear form. Vertices correspond to the integer points in the
interior of the cube $I_{d,n}=\{{\bf i}=(i_1,....,i_n) \vert 1 \le i_k \le
d-1\}$. Edges described in terms of bilinear form on the vector space
with basis $v_{\bf i}, {\bf i} \in I_{d,n}$ given by 
\begin{equation} \langle v_{\bf i},v_{\bf j} \rangle=
\begin{cases} 
 0 \ if \ \vert i_{\nu}-j_{\nu} \vert
   \ge 2 \ for \ some \ \nu \\
   0 \ if \ (i_{\nu}-j_{\nu})(i_{\mu}-j_{\mu})<1 \ for \ some
  \  \mu,\nu \\
  -2 \ if \ i=j \\
  -1 \ otherwise \\
\end{cases} 
\end{equation}
The edges of the graph connect the pairs of vertices $\bf i,j$ such
that $<v_{\bf i},v_{\bf j}>\ne 0$. In terms of this bilinear form or
the graph the
fundamental group of the complement to discriminant has generators
$t_{\bf i}$ corresponding to the vertices and the relations as follows
\begin{equation}\label{pham}
\begin{matrix}
 t_{\bf i}t_{\bf j}=t_{\bf j}t_{\bf i} & \ if  \ <v_{\bf i},v_{\bf j}>=
 0, \\
t_{\bf i}t_{\bf j}t_{\bf i}=t_{\bf j}t_{\bf i}t_{\bf j} & \ <v_{\bf
  i},v_{\bf j}> \ne 0 \\
t_{\bf i}t_{\bf j}t_{\bf k}t_{\bf i}=t_{\bf j}t_{\bf i}t_{\bf k}t_{\bf
  j} & \ <v_{\bf
  i},v_{\bf j}> <v_{\bf j},v_{\bf k}><v_{\bf k},v_{\bf i}>\ne 0 \\
  & i_{\nu} \le j_{\nu} \le k_{\nu} \ for  \ all \ \nu
\end{matrix}
\end{equation}

\subsection{Complements to discriminants of complete linear
  systems} 

This class of singular curves comprised of the curves where
the fundamental groups come endowed with the homomorphisms into 
non-abelian groups given by either geometric monodromy i.e. with
values in a mapping class group or (co)homological monodromy (with
values in the linear group of automorphisms of the
homology). Homological monodromies 
often are surjective or are close to such (i.e. the fundamental
group itself is non-abelian). The construction of these curve is as
follows. Let $X$
be a smooth projective variety and let $\CaL$ be a line bundle. The linear
system $\PP(H^0(X,\CaL))$ 
contains the discriminant consisting of the elements having
singularities worse than singularities of its generic element. 
With rare exceptions the discriminant has codimension 1 (identifying
varieties with a small dual is an interesting problem). 
Its intersection 
with a generic plane \footnote{generic choice assures that the fundamental
  group of the complement to the intersection with the plane inside
  this plane is
isomorphic to the fundamental group of the complement to the discriminant
of the complete linear system. Non-generic section were studies in
very special cases. For a recent study cf. \cite{lang}}  
 in $\PP(H^0(X,\CaL))$ produces a plane curve which
fundamental group of the complement has monodromy map into the mapping
class group of the generic fiber of the universal element of this
linear system i.e. the group of diffeomorphisms modulo isotopy of the fiber of the incidence correspondence
$I_{\CaL}\subset X \times \PP(H^0(X,\CaL)$ set theoretically consisting 
of pairs $\{(x,C) \vert x\in X, C\in \PP(H^0(X,\CaL)), x \in C\}$. In
\cite{dolgme} was considered the case $X=\PP^2$ (resp. $X=V_2$ the
quadric in $\PP^3$) and $\CaL=\mathcal{O}_{\PP^2}(3)$ (resp. $\CaL=\mathcal{O}_{V_2}(2)$) when one obtains as the fundamental group of
the complement to discriminant the extension of $SL(2,\ZZ)$ by the
Heisenberg group over the field with 3 elements (resp. the ring
$\ZZ_4$). The surjection onto $SL_2(\ZZ)$ is the monodromy (the mapping
class of 2-dimensional torus coincides with $SL_2(\ZZ)$) and the kernel
is the Heisenberg group. Recently, a progress was made in
understanding the kernel of the monodromy in the case 
$\CaL=\mathcal{O}_{\PP^2}(4)$ cf. \cite{harrisreid}.

A much more difficult case  $X=\PP^n,\CaL=\mathcal{O}_{\PP^n}(d)$,  including the case of
discriminant of the family of cubic curves just described, was addressed
by M.Lonne in \cite{lonne09}. It also
includes apparently the only other known case of this construction
i.e. $X=\PP^1,\CaL=\mathcal{O}(d)$ considered by Zariski (and mentioned in \cite{dolgme}) when the
corresponding fundamental group is the braid group of
two dimensional sphere. The fundamental group
$\pi_1(\PP(H^0(\PP^n,\mathcal{O}_{\PP^n}(d))\setminus Disc))$ is the
quotient of the group with generators and relations (\ref{pham}) by
the normal subgroup generated by additional
relations which we now shall describe. They are defined in terms of  
enumeration functions: ${\Upsilon}_k, k=0,....,n:
\{1,...,(d-1)^n\}\rightarrow I_{n,d}$ or equivalently the orderings of
the integral points of the cube $I_{n,d}$. Among them,  ${\Upsilon}_0$
considered as the ordering of the integral points in $I_{n,d}$ 
according to the reverse lexicographic order: $(i_1,....,i_n)<(i_1',....,i_n')$
iff the for the smallest subscript $k$ for which $i_k \ne i_k'$ one
has $i_k>i_{k'}$
(e.g. $(d-1,d-1,d-1)<(d-1,d-1,d-2)<(d-1,d-1,d-2)<...<(d-1,d-1,1)<
(d-1,d-2,d-1)<(d-1,d-2,d-2)<......$).
The order $<_k$ obtained from this one as follows:
\begin{equation}
  (i_1,...,i_n)<_k (j_1,...j_n) \Longleftrightarrow i_k<j_k \ or \
  i_k=j_k , (i_1,...,i_n)<_0(j_1,...j_n)
\end{equation}
With this notations a presentation of $\pi_1(\PP(H^0(\PP^n,\mathcal{O}_{\PP^n}(d))\setminus Disc))$ is given by generators and relators
(\ref{pham}) and 
\begin{equation}\label{lonnepresentation2}
 (t_i\delta_0)^{d-1}=(\delta_0t_i^{-1})^{d-1},   \delta_0 \cdot
 ....\cdot \delta_n=1 \text{where} \delta_k=\prod_{m=1}^{(d-1)^n}t_{\Upsilon_k(m)} \ k=0,....,n
\end{equation} 
It would be interesting to understand the algebraic structure of such
groups and their relation to other geometrically defined group but see 
\cite{lonne09} for discussion of the relation of this presentation
with those in cases known earlier. For results on the monodromy
representations of the groups of the complements to discriminant
using presentation (\ref{pham}), (\ref{lonnepresentation2}) we refer to 
\cite{salter2} and for the case of monodromy of complements to
discriminants on toric surfaces to \cite{salter1} and \cite{lang1}.

\subsection{Plane sextics\index{plane sextics} and trigonal curves}

In the last 10-20 years, many important results were obtained in the
study of  equisingular 
families of curves of degree 6 (and less; cf. \cite{degt99}, \cite{degt13},
\cite{degt12}, \cite{degt11}, \cite{degt11b}, \cite{degt10b} and
references below). 
The number of equisingular families of plane sextics
 measures in thousands and hence listing of possible cases
is not a reasonable approach. Several classes of sextics were
identified and we will describe some of them below. The methods
include the use of Alexander invariants, connection with 
K3 surfaces and relation with the class trigonal curves on ruled rational
surfaces. An interesting study of the moduli of sextics with six
cusps i.e. the locus in the moduli space $\cM_g$ 
given by the curves in distinct equisingular families was done in \cite{galati}. 
Several good surveys of the subject are already available
(cf. \cite{degt11c}, Preface and section 7.2 in \cite{degtbook} and \cite{degt15}).

\bigskip

A. {\it Simple and non-simple  sextics}. A sextic is called {\it simple} if its
only singularities are ADE singularities. Otherwise, a sextic is
called {\it non-simple}.
For irreducible non-simple sextics the type of equisingular deformation
type is determined by  the combinatorial type i.e. the collection
of the local types of all singularities (cf. \cite{degt11c}
Theorem 3.2.1). 
The key to a classification of simple sextics is the relation with the
theory of K3 surfaces. Consider a double cover $X_C$ of $\PP^2$ branched
over a sextic $C$. Singularities of this surface, correspond to
the singularities of the branching curve and are simple of the same ADE type
as the singularity of the curve. Moreover, the minimal resolution
$\tilde X_C$ comes with the following data associated with the
intersection form on $H_2(\tilde X_C,\ZZ)$. Recall that as a lattice
with bilinear 
form the latter is isomorphic to 
$\LL=2E_8 \oplus U^3$ where $U$ is the intersection form of quadric
surface. The data associated with the minimal resolution $\tilde X_C$
consists of the sublattice of
$H_2(\tilde X_C,\ZZ)$ spanned by the classes of exceptional curves of 
the resolution. These curves form a root system $\sigma$ in
this sublattice. Let $\tilde S_C$ be the primitive hull in $H_2(\tilde X_C,\ZZ)$
of these sublattices and the pull back to $\tilde X_C$ of the class of
a line in $\PP^2$.  An abstract oriented homological type of a K3
surface is a sublattice in $H_2(\tilde X_C,\ZZ)$, which
in the case of a double cover over a ADE sextics is the image of
$\tilde S_C$, plus the orientation of the positive definite plane in real subspace spanned by transcendental
lattice given by the holomorphic
2-form on $\tilde X_C$ (cf. \cite{degt10c} p.214).
\begin{theorem} (cf.\cite{degt08},\cite{urabe},\cite{yang})
There is one to one correspondence between oriented abstract
homological types arising from sextics and the set of equisingular deformations of sextic
curves with simple singularities. Moreover, the moduli space of
sextics in each equisingular component (i.e. its quotient by the group
of projective isomorphisms) is isomorphic to the moduli space of K3
surfaces with such abstract homological type. 
\end{theorem}

Particularly well understood class of such sextics is the class of
maximizing ones i.e. for which the sum of Milnor numbers is 19 (i.e. 
the maximal possible). However, there is no classification
of the fundamental groups for the curves of this type though very 
large number of cases was made explicite.

\bigskip 

B. Sextics of torus type. Those are sextics given by an equation of
the form $f_2^3+f_3^2=0$ where $f_i$ denotes a form of degree $i$.

If $f_i$ generic than for such $C$, $\pi_1(\PP^2\setminus C)$ is
the quotient of the braid group $B_3$ by its center
(\cite{zariski29}). The fundamental group varies 
when $f_i=0$ become singular or tangent to each other and there are
many explicite calculations. For curves with simple singularities, 
having such type, the commutativity of the fundamental group of the
complement is detected by the Alexander polynomial
 (Oka conjecture cf. \cite{degtoka}, \cite{jcogomecrelle})

\bigskip

C. Sextics with triple points. Blow up of the plane at a triple point of a sextic
results in Hirzebruch surface $F_1$ and a cover of degree 3 of 
projective line induced by projection from the triple point.
Such curves and their braid monodromy were studied extensively by 
Degtyarev in his book \cite{degt11c} in a more general framework 
of trigonal curves on arbitrary Hirzebruch surfaces $F_d$. Relation
between the braid monodromy and the graphs in 2-spheres leads to
enumeration of extremal irreducible trigonal curves which shows that 
their number grows exponentially (as a function of appropriate parameter).

\subsection{Zariski pairs\index{Zariski pairs}} 

One of applications of the fundamental
groups of the complements (as was envisioned and implemented in some
cases by Zariski cf. \cite{zariski29},\cite{zariskialgsurf}) 
is detecting the existence of different connected components of the 
space of equisingular deformations of curves on the surface. Indeed,   
those deformations do not alter the fundamental group. In fact there
are several natural topological equivalence relations of curves on
surfaces interrelationship between which is a natural question.

\begin{definition}\label{pairs}
Let $X$ be an algebraic surface and let $D_1,D_2$ be 
divisors on $X$.  Pairs $(X,D_1)$ and $(X,D_2)$ are equivalent if one
of the following conditions is satisfied:

\smallskip

(A) There exist an irreducible variety $T$, a holomorphic map $\Phi: \mathcal{X} \rightarrow T$ with a fiber biholomorphic to $X$, 
 a divisor $\CD \subset \cX$ such that  $\Phi$ is a locally trivial
 fibration of the pair $(\cX,\CD)$ \footnote{i.e. for any $t\in T$
   there is a neighborhood $U\subset T$ such that $\Phi^{-1}(U)$ and
   $T\times \Phi^{-1}(t)$ are equivalent as stratified spaces}
 and such that
there exist a 
pair of points $t_1,t_2 \in T$ the fibers of $\Phi\vert_{\CD}$ over $t_1,t_2$ are
$D_1,D_2$ respectively.

\smallskip
(${\rm A}^{\prime}$) There is a symplectic isotopy of pairs $(X,D_i)$ i.e. $(\cX,\CD)$ in (A)
is a pair of symplectic spaces with symplectic $\Phi$ with fiber being
sympletomorphic to $D_1,D_2$ respectively.  

\smallskip
(B) There exists a diffeomorphism (resp. PL equivalence,
reps. homeomorphism, resp. a homotopy equivalence, resp. proper
homotopy equivalence of the complements) of pairs $(X,D_1)=(X,D_2)$
i.e. one selects the corresponding type of a continuous map $X\rightarrow X$ taking subcomplex $D_1$ to
$D_2$.

\smallskip

(C) There exists an isomorphism of fundamental groups $\pi_1(X\setminus
D_1)=\pi_1(X\setminus D_2)$ (or sometimes just equality of the Alexander
polynomials).

\smallskip

(D) There exist the following:

(i) a one to one correspondence between irreducible components of
$D_i$ such that corresponding components are homeomorphic and

(ii) a one to one correspondence between singularities of $D_i$
preserving the local type in $X$ compatible with correspondence (i)
between the components.

\smallskip

(E) There exist an automorphism of fileds $\CC/\QQ$ which takes (a
deformation as in (a)) of the pair
$(X,D_1)$ to the pair $(X,D_2)$.

\end{definition}

The names used in literature are respectively, equisingular
deformation equivalence for (A), Zariski pairs \footnote{ the term was
  coined in \cite{artalpairs} in reference to first example found by
  Zariski in 1930s} for (D)-equivalent but not (B)-equivalent pairs ,
$\pi_1$-equivalent for (C), combinatorially equivalent for (D) and
conjugation equivalent for (E).

The relation between these conditions is as follows:
(A) implies (B) (Thom isotopy theorem), (B) implies (D) and also (C) by topological
invariance of the fundamental groups. Relation between equivalences in
(B) corresponding to different types of homeomorphisms of pairs are unknown in
real dimension 4 and finally (E) implies (D). 

Large and continuing to increase volume of papers
deals with finding examples confirming that these implications cannot
be reversed, though until 80s connected components of the strata
were viewed as an aberration. The conditions found in \cite{gruel} 
delineate the range of combinatorial data for which the strata 
are connected but numerous examples found up to date outside of this range,  suggest that disconnectedness
of equisingular families is a widespread occurence. At the same time no systematic theory of Zariski pairs as to 
classification or distribution did emerge. A good survey of
this vast subject is given in \cite{zariskipairssurvey}. 
Further non-trivial results on the relations between above equivalences are as follows.

(C) or (D) does not imply (A): Shirane \cite{shirane} showed that
curves in  equisingular families constructed earlier by Shimada 
(cf. \cite{shimada03}) cannot be transformed by a homeomorphism 
of $\PP^2$ though fundamental groups are isomorphic. 
Work \cite{degt14} contains examples of such type in the case of sextics.

(D) does not imply (C) for arrangements of lines defined over $\QQ$:
cf. \cite{benoit1} and references there for other numerous examples 
found by those authors giving arrangements of lines for which (D) does not
imply (C). (D) does not imply (C) for reducible curves with 
 components being a smooth curve and a union of certain 3
tangent lines
cf. \cite{shirane19}. $k$-tuples of pairwise distinct reducible
curves with one component of degree 4 and several quadrics were
considered in \cite{banaitokunaga} (also, see there the references to the works 
of these two authors presenting many other examples of failure of this 
implication). 

(E) does not imply (C): \cite{artalconjugate} gives examples of
conjugate line arrangements with non-isomorphic fundamental groups.
See also \cite{artal17} where one has conjugacy over $\QQ$ and isomorphism of the
fundamental groups and even homeomorphism of the complements but 
there is no homeomorphism of pairs. Examples are the appropriately chosen unions of
sextics and lines. Moreover, (E) and (C) do not imply (B) (cf. \cite{triangularcurves})
 
Distinct connected components often even contain curves conjugate over
$\QQ$ (arithmetic Zariski pairs cf. \cite{shimada08}).

The examples of Zariski pairs or multiplets \footnote{Zariski k-tuples
  are sets of $k$ curves in distinct classes of equivalence (B) but
  in the same class (D); sometimes, in a more loose usage, the reference is to sets of $k$
  curves in distinct classes for some equivalences (A)-(E) but not another.} 
fall in the following
groups

A.Arrangements of lines and conics. \cite{benoit}

B.Curves of degree 6 and trigonal curves (cf \cite{degt09})

C.Other sporadic examples such as reducible curves with components of
low degree (cf. \cite{ozguner}).

Methods employed in these works include study of the  Alexander invariants, Hurwitz equivalence classes of  
braid monodromy and more ad hoc invariants of the fundamental groups
(e.g. existence of dihedral cover of the complement to a curve  is an
invariant of the fundamental
group and hence can be used to distinguish classes of equisingular
deformations) and other sporadic methods (cf. \cite{shimada10},
\cite{shimada101}, \cite{shirane17}). Problems here include the
question 
of combinatorial invariance of the Alexander polynomials and more
generally the characteristic varieties or existence of Zariski pairs
defined over $\QQ$.

Many examples of fundamental groups of Zariski pairs were
computed in \cite{degt10} 

An interesting problem about Zariski multiplets is understanding the asymptotic of the 
number of connected components of equisingular families when the
number of classes of the curves 
grows. Consideration of families of trigonal curves on Hirzebruch surfaces,
shows that the number of equisingular components grows exponentially. 
One can show exponential growth of the number of connected components
of equisingular families of plane curves with nodes and cusps when
degree grows by considering generic
projections of surfaces of general type in a families which have
exponentially large growth of the number of connected components of the
moduli spaces (cf. \cite{lonnemoduli}). This follows from the
explicite formulas in terms of Chern numbers of the surfaces for the numbers of cusps, nodes and the degree of
the branching curves of generic projection (cf. \cite{bowdin}).

\printindex

\begin{thebibliography}{99}

\bibitem{degt15} A. Akyol, A.Degtyarev, Geography of irreducible plane sextics. 
Proc. Lond. Math. Soc. (3) 111 (2015), no. 6, 1307-1337. 

\bibitem{alberichjump}  M. Alberich-Carrami\~nana, J. Alvarez Montaner,
  F. Dachs-Cadefau, V. Gonz\'alez-Alonso,  Poincar\'e series of multiplier ideals in two-dimensional local rings with rational singularities. Adv. Math. 304 (2017), 769-792. 

\bibitem{alexander} J. W. Alexander,
Topological invariants of knots and links.
Trans. Amer. Math. Soc. 30 (1928), no. 2, 275-306.

\bibitem{alexeevpardini} V. Alexeev, V.; R. Pardini, On the existence of ramified abelian covers. Rend. Semin. Mat. Univ. Politec. Torino 71 (2013), no. 3-4, 307-315.

\bibitem{aluffi} P. Aluffi,
Characteristic classes of discriminants and enumerative geometry. (English summary)
Comm. Algebra 26 (1998), no. 10, 3165-3193.

\bibitem{moduli10lines}  M. Amram, M. Teicher; Fei Ye, 
Moduli spaces of arrangements of 10 projective lines with quadruple points. 
Adv. in Appl. Math. 51 (2013), no. 3, 392-418.

\bibitem{teicher11} M. Amram, R. Lehman, R. Shwartz, M. Teicher, 
Classification of fundamental groups of Galois covers of surfaces of small degree degenerating to nice plane arrangements. Topology of algebraic varieties and singularities, 65-94,
Contemp. Math., 538,

\bibitem{quadricline} M. Amram, M. Teicher, M. A. Uludag,
Fundamental groups of some quadric-line arrangements.
Topology Appl. 130 (2003), no. 2, 159-173, 


\bibitem{arapura}  D. Arapura, Geometry of cohomology support loci for local systems. I. J. Algebraic Geom. 6 (1997), no. 3, 563-597. 

\bibitem{aranori} D. Arapura, M. Nori, Solvable fundamental groups of algebraic varieties and K\"ahler manifolds.
Compositio Math. 116 (1999), no. 2, 173-188.
 

\bibitem{arnoldetal}  V. Arnold,; S. Gusein-Zade, A. Varchenko, Singularities of differentiable maps. Vol. II. Monodromy and asymptotics of integrals. Monographs in Mathematics, 83. Birkh\"auser Boston, Inc., Boston, MA, 1988.

\bibitem{triangularcurves} E. Artal Bartolo, J. I. Cogolludo-Agust\'in, J. Mart\'in-Morales, Triangular curves and cyclotomic Zariski tuples. Collect. Math. 71 (2020), no. 3, 427-441. 

\bibitem{artal17} E. Artal Bartolo,  J. I. Cogolludo-Agust\'in,
Some open questions on arithmetic Zariski pairs.Singularities in geometry, topology, foliations and dynamics, 31-54,
Trends Math., Birkh\"auser/Springer,  2017

\bibitem{artalconjugate} E. Artal Bartolo, J. I. Cogolludo-Agust\'in, B.Guerville-Ball\'e, M.Marco-Buzun\'ariz, An arithmetic Zariski pair of line arrangements with non-isomorphic fundamental group. Rev. R. Acad. Cienc. Exactas Fis. Nat. Ser. A Mat. RACSAM 111 (2017), no. 2, 377-402. 

\bibitem{zariskipairssurvey}  E. Artal Bartolo, J. I. Cogolludo,
  H.Tokunaga, A survey on Zariski pairs. Algebraic geometry in East Asia-Hanoi, 2005, 1-100, Adv. Stud. Pure Math., 50, Math. Soc. Japan, Tokyo, 2008.

\bibitem{ajcarmona}  E. Artal Bartolo, R. Carmona Ruber,  J. I. Cogolludo-Agust\'in, Braid monodromy and topology of plane curves. Duke Math. J. 118 (2003), no. 2, 261-278.

\bibitem{ajcarmona2} E. Artal Bartolo, Enrique, R. Carmona Ruber, J. I. Cogolludo Agust\'in,
Essential coordinate components of characteristic varieties. 
Math. Proc. Cambridge Philos. Soc. 136 (2004), no. 2, 287-299.

\bibitem{artal14} E. Artal Bartolo,
Topology of arrangements and position of singularities. 
Ann. Fac. Sci. Toulouse Math. (6) 23 (2014), no. 2, 223-265.

\bibitem{artalorbifolds}  E. Artal Bartolo,  J.I.Cogolludo-Agust\'in,
  D.Matei, Characteristic varieties of quasi-projective manifolds and
  orbifolds. Geom. Topol. 17 (2013), no. 1, 273-309. 

\bibitem{artalcogolludome} E. Artal Bartolo, J. I. Cogolludo-Agust\'in,
  Jose Ignacio, A. Libgober, Characters of fundamental groups of curve
  complements and orbifold pencils. Configuration spaces, 81–109,
CRM Series, 14, Ed. Norm., Pisa, 2012. 


\bibitem{acme2} E. Artal Bartolo, J. I. Cogolludo-Agust\'in, A. Libgober, Depth of cohomology support loci for quasi-projective varieties via orbifold pencils. Rev. Mat. Iberoam. 30 (2014), no. 2, 373-404.

\bibitem{artalpairs} E. Artal-Bartolo, Sur les couples de Zariski.  J. Algebraic Geom. 3 (1994), no. 2, 223-247. 



\bibitem{ludmil2} D. Auroux, L. Katzarkov, Branched coverings of ${\bf CP}^2$ and invariants of symplectic 4-manifolds. Invent. Math. 142 (2000), no. 3, 631-673. 

\bibitem{ludmil1} D. Auroux, S. Donaldson, L. Katzarkov, M. Yotov,  Fundamental groups of complements of plane curves and symplectic invariants. Topology 43 (2004), no. 6, 1285-1318.


\bibitem{bahlul} R. Bahloul, Demonstration constructive de l'existence de polynomes de Bernstein-Sato pour plusieurs fonctions analytiques, Compos. Math. 141 (2005), no. 1, p. 175-191.

\bibitem{banaitokunaga} S. Bannai, H. Tokunaga,
Geometry of bisections of elliptic surfaces and Zariski N-plets II.
Topology Appl. 231 (2017), 10-25.

\bibitem{batyrev} V. Batyrev, Dual polyhedra and mirror symmetry for Calabi-Yau hypersurfaces in toric varieties. J. Algebraic Geom. 3 (1994), no. 3, 493-535. 

\bibitem{bessis}  D. Bessis, Finite complex reflection arrangements are $K(\pi,1)$. Ann. of Math. (2) 181 (2015), no. 3, 809-904.

\bibitem{brieskorn}  E. Brieskorn, Die Fundamentalgruppe des Raumes der regularen Orbits einer endlichen komplexen Spiegelungsgruppe. Invent. Math. 12 (1971), 57-61.

\bibitem{beauville} A. Beauville, Annulation du $H^1$ pour les fibres en droites plats, Complex algebraic varieties (Bayreuth, 1990), 1-15, Lecture Notes in Math., 1507, Springer, Berlin, 1992.
 
\bibitem{bieri}   R. Bieri, B. Eckmann, Groups with homological duality
  generalizing Poincar\'e duality. Invent. Math. 20 (1973), 103-124. 

\bibitem{biswas}  I. Biswas, M. Mj, Quasiprojective three-manifold groups and complexification of three-manifolds. Int. Math. Res. Not. IMRN 2015, no. 20, 10041-10068.

\bibitem{bredon} G. Bredon, Introduction to compact transformation groups. Pure and Applied Mathematics, Vol. 46. Academic Press, New York-London, 1972. 

\bibitem{broue} M. Brou\'e, G. Malle, R. Rouquier, Complex reflection groups, braid groups, Hecke algebras. J. Reine Angew. Math. 500 (1998), 127-190.

\bibitem{buchsbaum} D. Buchsbaum, D. Eisenbud, What annihilates a module? J. Algebra 47 (1977), no. 2, 231-243.

\bibitem{budurspectrum} N. Budur, On Hodge spectrum and multiplier ideals.
Math. Ann. 327 (2003), no. 2, 257-270.

\bibitem{budur} N. Budur, Unitary local systems, multiplier ideals, and
  polynomial periodicity of Hodge numbers. Adv. Math. 221 (2009),
  no. 1, 217-250. 

\bibitem{budur15} N. Budur, Bernstein-Sato ideals and local systems. Ann. Inst. Fourier (Grenoble) 65 (2015), no. 2, 549-603. 


\bibitem{pierretteme11} P. Cassou-Nogu\'es, A. Libgober, Multivariable Hodge theoretical invariants of germs of plane curves. J. Knot Theory Ramifications 20 (2011), no. 6, 787-805.

\bibitem{pierretteme14} P. Cassou-Nogu\'es, A. Libgober, Multivariable Hodge theoretical invariants of germs of plane curves. II. Valuation theory in interaction, 82-135, EMS Ser. Congr. Rep., Eur. Math. Soc., Z\"urich, 2014. 

\bibitem{cappell} S. E. Cappell, J. L. Shaneson, Singular spaces, characteristic classes, and intersection
homology, Ann. of Math. 134 (1991)

\bibitem{catanese} F. Catanese,
On a problem of Chisini.
Duke Math. J. 53 (1986), no. 1, 33-42.

\bibitem{ciliberto} F. Catanese, C. Ciliberto,
On the irregularity of cyclic coverings of algebraic surfaces. Geometry of complex projective varieties (Cetraro, 1990), 89-115,
Sem. Conf., 9, Mediterranean, Rende, 1993.

\bibitem{cogoflorens} J. I. Cogolludo Agust\'in, V. Florens,
Twisted Alexander polynomials of plane algebraic curves. 
J. Lond. Math. Soc. (2) 76 (2007), no. 1, 105-121.

\bibitem{jaca} Topology of algebraic varieties and singularities.
Papers from the Conference on Topology of Algebraic Varieties, in honor of Anatoly Libgober's 60th birthday, held in Jaca, June 22-26, 2009. Edited by Jos\'e Ignacio Cogolludo-Agust\'in and Eriko Hironaka. Contemporary Mathematics, 538. American Mathematical Society, Providence, RI; Real Sociedad Matematica Espanola, Madrid, 2011. 

\bibitem{jcogomecrelle} J. I. Cogolludo and A. Libgober, Mordell-Weil groups of elliptic threefolds and the Alexander module of plane curves. J. Reine Angew. Math. 697 (2014), 15-55. 


\bibitem{asymptotics} J. I. Cogolludo, A. Libgober, Free quotients of
  fundamental groups of smooth quasi-projective varieties, arxiv 1904.10852

\bibitem{lang1}  R. Cretois, L. Lang, The vanishing cycles of curves in toric surfaces I. Compos. Math. 154 (2018), no. 8, 1659-1697.

\bibitem{daviskirk} J. Davis, F. P. Kirk, 
Lecture notes in algebraic topology.
Graduate Studies in Mathematics, 35. American Mathematical Society, Providence, RI, 2001.

\bibitem{degt08} A. Degtyarev, On deformations of singular plane sextics. J. Algebraic Geom. 17 (2008), no. 1, 101-135. 

\bibitem{degtoka}  A. Degtyarev,  Oka's conjecture on irreducible plane sextics. J. Lond. Math. Soc. (2) 78 (2008), no. 2, 329-351.



\bibitem{degt13} A. Degtyarev,
On plane sextics with double singular points. 
Pacific J. Math. 265 (2013), no. 2, 327-348. 
 
\bibitem{degt12} A. Degtyarev,
Dihedral coverings of trigonal curves.  
Indiana Univ. Math. J. 61 (2012), no. 3, 901-938. 

\bibitem{degtbook} A. Degtyarev, Topology of algebraic curves. An
  approach via dessins d'enfants. De Gruyter Studies in Mathematics,
  44. Walter de Gruyter and Co., Berlin, 2012.

\bibitem{degt11} A. Degtyarev, Hurwitz equivalence of braid monodromies and extremal elliptic surfaces. Proc. Lond. Math. Soc. (3) 103 (2011), no. 6, 1083-1120.

\bibitem{degt11b} A. Degtyarev, 
The fundamental group of a generalized trigonal curve. 
Osaka J. Math. 48 (2011), no. 3, 749-782.
 
\bibitem{degt11c} A. Degtyarev, Topology of plane algebraic curves: the algebraic approach. Topology of algebraic varieties and singularities, 137-161, Contemp. Math., 538, Amer. Math. Soc., Providence, RI, 2011.

\bibitem{degt10} A. Degtyarev, 
Classical Zariski pairs.
J. Singul. 2 (2010), 51-55. 

\bibitem{degt10b} A. Degtyarev Plane sextics with a type E8 singular point. Tohoku Math. J. (2) 62 (2010), no. 3, 329-355. 

\bibitem{degt10c} A. Degtyarev, Plane sextics via dessins d'enfants. Geom. Topol. 14 (2010), no. 1, 393-433.

\bibitem{degt09} A. Degtyarev, 
Zariski k-plets via dessins d'enfants. 
Comment. Math. Helv. 84 (2009), no. 3, 639-671.

\bibitem{degt2001} A. Degtyarev, A divisibility theorem for the Alexander polynomial of a plane algebraic curve. Zap. Nauchn. Sem. S.-Peterburg. Otdel. Mat. Inst. Steklov. (POMI) 280 (2001), Geom. i Topol. 7, 146-156, 300; translation in J. Math. Sci. (N.Y.) 119 (2004), no. 2, 205-210. 

\bibitem{degt99} A. Degtyarev, Quintics in ${\bf CP}^2$ with nonabelian fundamental group. (Russian) Algebra i Analiz 11 (1999), no. 5, 130-151; translation in St. Petersburg Math. J. 11 (2000), no. 5, 809-826. 

\bibitem{degt90} A. Degtyarev, I. Isotopic classification of complex
  plane projective curves of degree 5. Algebra i Analiz 1 (1989), no. 4, 78-101; translation in Leningrad Math. J. 1 (1990), no. 4, 881-904. 


\bibitem{degt14} A. Degtyarev, On the Artal-Carmona-Cogolludo construction, J. Knot Theory Ramifications 23 (2014), no. 5.

\bibitem{deligne} P. Deligne, Le groupe fondamental du compl\'ement d'une
  courbe plane n'ayant que des points doubles ordinaires est abelien
  (d'apres W. Fulton).  Bourbaki Seminar, Vol. 1979/80, pp. 1-10, Lecture Notes in Math., 842, Springer, Berlin-New York, 1981.

\bibitem{delignediffeq} P. Deligne, Equations differentielles a points singuliers reguliers. Lecture Notes in Mathematics, Vol. 163. Springer-Verlag, Berlin-New York, 1970. 

\bibitem{delignehodgeI} P. Deligne,  Theorie de Hodge. II.
  Inst. Hautes Etudes Sci. Publ. Math. No. 40 (1971), 5-57

\bibitem{delignehodgeII} P. Deligne, Theorie de Hodge. III.  Inst. Hautes Etudes Sci. Publ. Math. No. 44 (1974), 5-77


\bibitem{denhamsuciu}  G. Denham, A.Suciu, Multinets, parallel connections, and Milnor fibrations of arrangements. 
Proc. Lond. Math. Soc. (3) 108 (2014), no. 6, 1435-1470.


\bibitem{dsy} G. Denham, A. Suciu, S. Yuzvinsky, Abelian duality and propagation of resonance. Selecta Math. (N.S.) 23 (2017), no. 4, 2331-2367.

\bibitem{zaidenberg}  G. Dethloff, S. Orevkov, M. Zaidenberg, 
Plane curves with a big fundamental group of the complement,  Voronezh Winter Mathematical Schools, 63-84,
Amer. Math. Soc. Transl. Ser. 2, 184, Adv. Math. Sci., 37, Amer. Math. Soc., Providence, RI, 1998.

\bibitem{dimcabook} A. Dimca, Singularities and topology of hypersurfaces. Universitext. Springer-Verlag, New York, 1992.

\bibitem{dimcaarrangements} A. Dimca, Hyperplane arrangements. An introduction. Universitext. Springer, 2017. 

\bibitem{dimca07} A. Dimca, L. Maxim, 
Multivariable Alexander invariants of hypersurface complements.
Trans. Amer. Math. Soc. 359 (2007), no. 7, 3505-3528.

\bibitem{dimcasuciu3manifolds} Dimca, A., Suciu, A.: Which 3-manifold groups are K\"ahler groups? J. Eur. Math. Soc. 11(3), 521-528 (2009). 

\bibitem{dimcapapadima}  A. Dimca, S. Papadima, A. Suciu, Alexander I
Quasi-K\"ahler groups, 3-manifold groups, and formality. 
Math. Z. 268 (2011), no. 1-2, 169-186.



\bibitem{dolgme} I.Dolgachev, A. Libgober, On the fundamental group of the complement to a discriminant variety. Algebraic geometry (Chicago, Ill., 1980), pp. 1-25, Lecture Notes in Math., 862, Springer, Berlin-New York, 1981.

\bibitem{ebeling}  W. Ebeling, Distinguished bases and monodromy of complex
  hypersurface singularities. Handbook of Geometry and Topology of Singularities, 2020.
p.427-466. 


\bibitem{conjfreepresent} M. Eliyahu, D. Garber, M. Teicher,
A conjugation-free geometric presentation of fundamental groups of arrangements. 
Manuscripta Math. 133 (2010), no. 1-2, 247-271.

\bibitem{eisenbud-ca} D. Eisenbud, Commutative algebra. With a view toward algebraic geometry. Graduate Texts in Mathematics, 150. Springer-Verlag, New York, 1995.

\bibitem{epy} D. Eisenbud, S. Popescu, S. Yuzvinsky, Hyperplane arrangement cohomology and monomials in the
exterior algebra. Trans. Am. Math. Soc. 355(11), 4365-4383 (2003)


\bibitem{esnault} H. Esnault, Fibre de Milnor d'un cone sur une courbe plane singuli\'ere.  Invent. Math. 68 (1982), no. 3, 477-496.

\bibitem{esnaultviehweg} H. Esnault, E. Viehweg, Lectures on vanishing theorems. DMV Seminar, 20. Birkh\"auser Verlag, Basel, 1992


\bibitem{lang} A. Esterov, L. Lang, Braid monodromy of univariate   fewnomials, arxiv:2001.01634

\bibitem{falkyuz} M. Falk, S. Yuzvinsky, Multinets, resonance varieties, and pencils of plane curves. Compos. Math. 143 (4), 1069-1088 (2007) 

\bibitem{fan} K. M. Fan, Direct product of free groups as the fundamental group of the complement of a union of lines.
Michigan Math. J. 44 (1997), no. 2, 283-291.

\bibitem{farb} B. Farb, D. Margalit,
A primer on mapping class groups.
Princeton Mathematical Series, 49. Princeton University Press, Princeton, NJ, 2012.

\bibitem{fei}  Fei Ye, Classification of moduli spaces of arrangements of nine projective lines. Pacific J. Math. 265 (2013), no. 1, 243-256.

\bibitem{fox}  R. H. Fox, A quick trip through knot theory. 1962 Topology of 3-manifolds and related topics (Proc. The Univ. of Georgia Institute, 1961) pp. 120-167 Prentice-Hall, Englewood Cliffs, N.J. 

\bibitem{friedl} S. Friedl, A. Suciu,
K\"ahler groups, quasi-projective groups and 3-manifold groups.
J. Lond. Math. Soc. (2) 89 (2014), no. 1, 151-168.

\bibitem{fujita}  T. Fujita,  On the topology of noncomplete algebraic surfaces. J. Fac. Sci. Univ. Tokyo Sect. IA Math. 29 (1982), no. 3, 503-566.

\bibitem{fulton-80} W. Fulton, 
On the fundamental group of the complement of a node curve.
Ann. of Math. (2) 111 (1980), no. 2, 407-409.

\bibitem{fulton87} W. Fulton, On the topology of algebraic varieties. Algebraic geometry, Bowdoin, 1985 (Brunswick, Maine, 1985), 15-46, Proc. Sympos. Pure Math., 46, Part 1, Amer. Math. Soc., Providence, RI, 1987. 

\bibitem{galati} C. Galati, On the number of moduli of plane sextics with six cusps. Ann. Mat. Pura Appl. (4) 188 (2009), no. 2, 359-368.

\bibitem{galindo} C. Galindo, F.Hernando, F.Monserrat,  The log-canonical threshold of a plane curve. Math. Proc. Cambridge Philos. Soc. 160 (2016), no. 3, 513-535.

\bibitem{golla} M. Golla, L. Starkston, The symplectic isotopy problem for rational cuspidal curves. arXiv:1907.06787

\bibitem{gonzalez} J. Gonz\'alez-Meneses,
Basic results on braid groups, 
Annales Math\'ematiques Blaise Pascal, Volume 18 (2011) no. 1, p. 15-59

\bibitem{manuel} M. Gonz\'alez-Villa, A. Libgober, L.Maxim, Motivic infinite cyclic covers. Adv. Math. 298 (2016), 413-447.

\bibitem{goresky}  M. Goresky, R. MacPherson, Stratified Morse theory. Ergebnisse der Mathematik und ihrer Grenzgebiete (3) 14. Springer-Verlag, Berlin, 1988. 

\bibitem{grauertrem}  H. Grauert, R. Remmert, Coherent analytic sheaves. Grundlehren der Mathematischen Wissenschaften, 265. Springer-Verlag, Berlin, 1984.


\bibitem{GKKP} D. Greb, S. Kebekus, S. J. Kovacs, T.Peternell. Differential forms ´
on log canonical spaces. Inst. Hautes Etudes Sci. Publ. Math, 114(1):87-169, November 2011.

\bibitem{greenlaz} M. Green, R. Lazarsfeld, Higher obstructions to deforming cohomology groups of line bundles. J. Amer. Math. Soc. 4 (1991), no. 1, 87-103. 

\bibitem{SGA1}  Rev\^etements \^Etales et Groupe Fondamental, Lecture Notes in Mathematics, 224, Springer-Verlag, 1971, by Alexander Grothendieck et al. Updating remarks by Michel Raynaud. 

\bibitem{gruel}  G. M. Greuel, C. Lossen, E. Shustin,  Singular algebraic curves. With an appendix by Oleg Viro. Springer Monographs in Mathematics. Springer, Cham, 2018.

\bibitem{greuel} G. M. Greuel, Deformations and Smoothings of
Singularities, Handbook of Geometry and Topology of Singularities, 2020.p.369-426.

\bibitem{benoit1} B. Guerville-Ball\'e, 
An arithmetic Zariski 4-tuple of twelve lines. 
Geom. Topol. 20 (2016), no. 1, 537-553.

\bibitem{benoit} B. Guerville-Ball\'e, J. Viu-Sos,  Configurations of points and topology of real line arrangements. Math. Ann. 374 (2019), no. 1-2, 1-35.

\bibitem{gyoga} A. Gyoja, Bernstein-Sato's polynomial for several analytic functions, J. Math. Kyoto Univ. 33(2) (1993) 399-411. 

\bibitem{hain09} R. Hain, Monodromy of codimension 1 subfamilies of universal curves. Duke Math. J. 161 (2012), no. 7, 1351-1378.

\bibitem{harris-86} J. Harris, On the Severi problem, Invent. Math. 84 (1986), 445-461. 

\bibitem{harrisreid} R. Harris, The kernel of the monodromy of the universal
  family of degree d smooth plane curves. arxiv 1904.10355.

\bibitem{hartshorne} R. Hartshone, Algebraic geometry. Graduate Texts in Mathematics, No. 52. Springer-Verlag, New York-Heidelberg, 1977.

\bibitem{hamm} H. Hamm, Le Dung Trang, Un theoreme du type de Lefschetz.  C. R. Acad. Sci. Paris Ser. A-B 272 (1971), A946-A949. 


\bibitem{hatcher} A. Hatcher, Algebraic topology. Cambridge University Press, Cambridge, 2002. 

\bibitem{hironaka} E. Hironaka, Abelian coverings of the complex projective plane branched along configurations of real lines. Mem. Amer. Math. Soc. 105 (1993), no. 502,

\bibitem{ekojump}  E. Hironaka,  Alexander stratifications of character varieties. Ann. Inst. Fourier (Grenoble) 47 (1997), no. 2, 555-583. 

\bibitem{hirzebruch} F. Hirzebruch,
Arrangements of lines and algebraic surfaces. Arithmetic and geometry, Vol. II, 113-140,
Progr. Math., 36, Birkh\"auser, Boston, Mass., 1983.

\bibitem{ishida} M. Ishida,
The irregularities of Hirzebruch's examples of surfaces of general type with $c^2_1=3c_2$.
Math. Ann. 262 (1983), no. 3, 407-420.

\bibitem{millson} M. Kapovich and J. J. Millson, Inst. Hautes \'Etudes Sci. Publ. Math. No. 88 (1998), 5-95 (1999); C. R. Acad. Sci. Paris S\'er. I Math. 325 (1997), no. 8, 871-876; 

\bibitem{kashiwara} M. Kashiwara, "B-functions and holonomic systems. Rationality of roots of B-functions", Invent. Math. 38 (1976/77), no. 1, p. 33-53. 

\bibitem{knottheory} A. Kawauchi, A survey of knot theory. Burkhauser, 1990.


\bibitem{kirkliv}  P. Kirk, C. Livingston, Twisted Alexander invariants, Reidemeister torsion, and Casson-Gordon
invariants, Topology 38 (1999), no. 3, 635-661.

\bibitem{kots} D. Kotschick, K\"ahlerian three-manifold groups, Math. Res. Lett. 20 (2013), no. 3, 521-525.

\bibitem{kulikov} Vik. S. Kulikov, On Chisini's conjecture, Izv. Ross. Akad. Nauk Ser. Mat. 63:6 (1999), 83-116. 

\bibitem{lasell} B. Lasell, 
Complex local systems and morphisms of varieties. 
Compositio Math. 98 (1995), no. 2, 141-166.

\bibitem{lazarsfeld} R. Lazarsfeld,  Positivity in algebraic
  geometry. vol I and vol. II. Classical setting: line bundles and
  linear series. Ergebnisse der Mathematik und ihrer
  Grenzgebiete. 3. Folge. A Series of Modern Surveys in Mathematics
48. Springer-Verlag, Berlin, 2004.

\bibitem{le} L\^e D\~ung Tr\'ang,
Sur les noeuds alg\'ebriques. 
Compositio Math. 25 (1972), 281–321.

\bibitem{mealex} A. Libgober, Alexander polynomial of plane algebraic curves and cyclic multiple planes. Duke Math. J. 49 (1982), no. 4, 833-851.


\bibitem{arcata81} A. Libgober, Alexander invariants of plane algebraic
  curves. Singularities, Part 2 (Arcata, Calif., 1981), 135-143,
  Proc. Sympos. Pure Math., 40, Amer. Math. Soc., Providence, RI,
  1983.


\bibitem{bowdin} A. Libgober, Fundamental groups of the complements to plane singular curves. Algebraic geometry, Bowdoin, 1985 (Brunswick, Maine, 1985), 29-45, Proc. Sympos. Pure Math., 46, Part 2, Amer. Math. Soc., Providence, RI, 1987. 
 
\bibitem{homologyabcov}  A. Libgober,  On the homology of finite abelian coverings. Topology Appl. 43 (1992), no. 2, 157-166. 

\bibitem{meabhyankar}   A. Libgober, Groups which cannot be realized as
  fundamental groups of the complements to hypersurfaces in ${\bf C}^N$. Algebraic geometry and its applications (West Lafayette, IN, 1990), 203-207, Springer, New York, 1994.

\bibitem{meannals} A. Libgober, Homotopy groups of the complements to singular hypersurfaces. II. Ann. of Math. (2) 139 (1994), no. 1, 117-144. 

\bibitem{me2complex} A. Libgober, On the homotopy type of the complement to plane algebraic curves. J. Reine Angew. Math. 367 (1986), 103-114.

\bibitem{mesergey} A. Libgober,  S. Yuzvinsky, 
Cohomology of the Orlik-Solomon algebras and local systems. 
Compositio Math. 121 (2000), no. 3, 337-361.

\bibitem{mecharvar} A. Libgober, Characteristic varieties of algebraic curves. Applications of algebraic geometry to coding theory, physics and computation (Eilat, 2001), 215-254, NATO Sci. Ser. II Math. Phys. Chem., 36, Kluwer Acad. Publ., Dordrecht, 2001. 


\bibitem{me2002}  A. Libgober, Hodge decomposition of Alexander invariants. Manuscripta Math. 107 (2002), no. 2, 251-269.


\bibitem{trends} A. Libgober, Eigenvalues for the monodromy of the Milnor fibers of arrangements. Trends in singularities, 141-150, Trends Math., Birkh\"auser, Basel, 2002.


\bibitem{isolatednonnormal}  A. Libgober, Isolated non-normal
  crossings. Real and complex singularities, 145-160, Contemp. Math.,
  354, Amer. Math. Soc., Providence, RI, 2004.

\bibitem{meample}  A. Libgober,
Homotopy groups of complements to ample divisors. (English summary) Singularity theory and its applications, 179-204,
Adv. Stud. Pure Math., 43, Math. Soc. Japan, Tokyo, 2006.


\bibitem{me09} A. Libgober, Lectures on topology of complements and fundamental groups. Singularity theory, 71-137, World Sci. Publ., Hackensack, NJ, 2007. 

\bibitem{me2009} A. Libgober, Non vanishing loci of Hodge numbers of local systems. Manuscripta Math. 128 (2009), no. 1, 1-31.

\bibitem{medevelopment} A. Libgober, Development of the theory of Alexander invariants in algebraic geometry. Topology of algebraic varieties and singularities, 3-17, Contemp. Math., 538, Amer. Math. Soc., Providence, RI, 2011. 

\bibitem{mustatame} A. Libgober, M.Mustata, Sequences of LCT-polytopes. Math. Res. Lett. 18 (2011), no. 4, 733-746. 

\bibitem{mathannalen} A. Libgober,  On Mordell-Weil groups of isotrivial abelian varieties over function fields. Math. Ann. 357 (2013), no. 2, 605-629.


\bibitem{liedtke} C. Liedtke, 
Fundamental groups of Galois closures of generic projections. 
Trans. Amer. Math. Soc. 362 (2010), no. 4, 2167-2188.

\bibitem{liunearby}  Y. Liu, Nearby cycles and Alexander modules of hypersurface complements. Adv. Math. 291 (2016), 330-361.

\bibitem{liumaximwang} Y. Liu, L. Maxim, B. Wang, Mellin transformation, propagation, and abelian duality spaces. Adv. Math. 335 (2018), 231-260.

\bibitem{lmwsurvey} Y. Liu, L. Maxim, B. Wang, Perverse sheaves on
  semi-abelian varieties -- a survey of properties and applications, arXiv:1902.05430.

\bibitem{liumaxim} Y. Liu and L. Maxim, 
Reidemeister torsion, peripheral complex and Alexander polynomials of hypersurface complements. 
Algebr. Geom. Topol. 15 (2015), no. 5, 2757-2787.
 
\bibitem{loeservaq} F. Loeser, M. Vaqui\'e, 
Le polyn\^ome d'Alexander d'une courbe plane projective.
Topology 29 (1990), no. 2, 163-173.

\bibitem{lonne09} M. Lonne, Fundamental groups of projective discriminant complements. Duke Math. J. 150 (2009), no. 2, 357-405.

\bibitem{lonnemoduli} M. Lonne, M. Penegini,
On asymptotic bounds for the number of irreducible components of the moduli space of surfaces of general type II. 
Doc. Math. 21 (2016), 197-204.


\bibitem{malgrange} B. Malgrange, Polyn\^omes de Bernstein-Sato et cohomologie evanescente, in Analysis and topology on singular spaces, II, III (Luminy, 1981), Asterisque, vol. 101, Soc. Math. France, Paris, 1983, p. 243-267.


\bibitem{malle} G. Malle, On the distribution of Galois groups. J. Number Theory 92 (2002), no. 2, 315-329. 

\bibitem{marco} M. A. Marco Buzun\'ariz, A description of the resonance variety of a line combinatorics via combinatorial pencils. Graphs Combin. 25 (2009), no. 4, 469-488. 

\bibitem{maxim06} L. Maxim, 
Intersection homology and Alexander modules of hypersurface complements. 
Comment. Math. Helv. 81 (2006), no. 1, 123-155.

\bibitem{maximbook} L. Maxim, Intersection Homology and Perverse
  Sheaves with Applications to Singularity Theory. Graduate Text in
  Mathenmatics, 218, Springer, 2019.

\bibitem{maynadier} H. Maynadier,
Polynomes de Bernstein-Sato associes a une intersection complète quasi-homogene à singularite isolee.
Bull. Soc. Math. France 125 (1997), no. 4, 547-571.


\bibitem{milnoralex} J. Milnor, Infinite cyclic coverings. 1968
  Conference on the Topology of Manifolds (Michigan State Univ.,
  E. Lansing, Mich., 1967) pp. 115-133 Prindle, Weber
and Schmidt, Boston, Mass. 

\bibitem{milnorduality} J. Milnor, 
A duality theorem for Reidemeister torsion.
Ann. of Math. (2) 76 (1962), 137-147.

\bibitem{milnor} J. Milnor, Singular points of complex hypersurfaces. Annals of Mathematics Studies, No. 61 Princeton University Press, Princeton, N.J.; University of Tokyo Press, Tokyo 1968

\bibitem{moishezon} B. Moishezon, Stable branch curves and braid monodromies. Algebraic geometry (Chicago, Ill., 1980), pp. 107-192, Lecture Notes in Math., 862, Springer, Berlin-New York, 1981. 

\bibitem{moishezonchisini} B. Moishezon. The arithmetic of braids and a statement of Chisini. Geometric topology (Haifa, 1992), 151-175, Contemp. Math., 164, Amer. Math. Soc., Providence, RI, 1994.

\bibitem{moishteicher} B. Moishezon, M. Teicher,
Simply-connected algebraic surfaces of positive index.
Invent. Math. 89 (1987), no. 3, 601-643 and also B.Moishezon, M.Teicher, Galois coverings in the theory of algebraic surfaces, Proc.
Symp. Pure Math. 46, 47-65 (1987).

\bibitem{mumford} D. Mumford, The topology of normal singularities of an algebraic surface and a criterion for simplicity. Inst. Hautes Etudes Sci. Publ. Math. No. 9 (1961), 5-22.

\bibitem{mumfordalex} D. Mumford,  Appendix 2 to Chapter VIII in \cite{zariskialgsurf}.

\bibitem{nadel} A. Nadel,
Multiplier ideal sheaves and Kahler-Einstein metrics of positive scalar curvature.
Ann. of Math. (2) 132 (1990), no. 3, 549-596.

\bibitem{naie1}  D. Naie, The irregularity of cyclic multiple planes after Zariski. 
Enseign. Math. (2) 53 (2007), no. 3-4, 265-305.

\bibitem{naie2}  D. Naie, Mixed multiplier ideals and the irregularity of abelian coverings of smooth projective surfaces. Expo. Math. 31 (2013), no. 1, 40-72.

\bibitem{namba} M. Namba, H.Tsuchihashi,
On the fundamental groups of Galois covering spaces of the projective plane.
Geom. Dedicata 105 (2004), 85-105.

\bibitem{yoshinaga12} S. Nazir, M. Yoshinaga, On the connectivity of the realization spaces of line arrangements. Ann. Sc. Norm. Super. Pisa Cl. Sci. (5) 11 (2012), no. 4, 921-937. 

\bibitem{nori} M. Nori, Zariski's conjecture and related problems.
Ann. Sci. \'Ecole Norm. Sup. (4) 16 (1983), no. 2, 305-344.

\bibitem{oka} M. Oka, A survey on Alexander polynomials of plane curves. Singularit\'es Franco-Japonaises, 209-232, S\'emin. Congr., 10, Soc. Math. France, Paris, 2005.

\bibitem{ozguner} A. Ozguner, Classical Zariski Pairs with nodes,
  Master Thesis, Bilkent University, 2007.

\bibitem{papadimasuciu} S. Papadima, A. Suciu, 
Bieri-Neumann-Strebel-Renz invariants and homology jumping loci. 
Proc. Lond. Math. Soc. (3) 100 (2010), no. 3, 795-834.

\bibitem{pardini}  R. Pardini, Abelian covers of algebraic varieties. J. Reine Angew. Math. 417 (1991), 191-213. 

\bibitem{pereira} J.V. Pereira, S. Yuzvinsky, Completely reducible hypersurfaces in a pencil. Adv. Math. 219 (2008), no. 2, 672-688. 

\bibitem{peterssteenbrink} C. Peters, J. Steenbrink,  Mixed Hodge
  structures. Ergebnisse der Mathematik und ihrer Grenzgebiete.  52. Springer-Verlag, Berlin, 2008. 

\bibitem{popp} H. Popp, 
Fundamentalgruppen algebraischer Mannigfaltigkeiten. 
Lecture Notes in Mathematics, Vol. 176 Springer-Verlag, Berlin-New York 1970 

\bibitem{robb} A. Robb, M. Teicher, Applications of braid group techniques to the decomposition of moduli spaces, new examples. Special issue on braid groups and related topics (Jerusalem, 1995). Topology Appl. 78 (1997), no. 1-2, 143-151. 
 
\bibitem{urzua} X. Roulleau, G. Urz\'ua, Chern slopes of simply connected complex surfaces of general type are dense in [2,3]. Ann. of Math. (2) 182 (2015), no. 1, 287-306. 

\bibitem{ruppert} W. Ruppert, 
Reduzibilitat ebener Kurven. J. Reine Angew. Math. 369 (1986), 167-191.

\bibitem{sabbah} C. Sabbah, Proximite evanescente. I. La structure
  polaire d'un $\mathcal{D}$-module, Compositio Math. 62 (1987), no. 3, p. 283-328.

\bibitem{saito}  M. Saito, Mixed Hodge modules. Publ. Res. Inst. Math. Sci. 26 (1990), no. 2, 221-333. 
`
\bibitem{saitospectrum} M. Saito,
Exponents of an irreducible plane curve singularity, arXiv:math.0009133.

\bibitem{sakuma} M. Sakuma, Homology of abelian coverings of links and spatial graphs. Canad. J. Math. 47 (1995), no. 1, 201-224.

\bibitem{salter1} N. Salter, Monodromy and vanishing cycles in toric surfaces. Invent. Math. 216 (2019), no. 1, 153-213.

\bibitem{salter2} N. Salter. On the monodromy group of the family of smooth plane curves.  arXiv:1610.04920.

\bibitem{salvetti} M. Salvetti,
Arrangements of lines and monodromy of plane curves.
Compositio Math. 68 (1988), no. 1, 103-122.

\bibitem{schnell} C. Schnell, An overview of Morihiko Saito's theory of
  mixed Hodge modules, arxiv 1405.3096

\bibitem{segre}  B. Segre, Sulla Caratterizzazione delle curve di
  diramazione dei piani multipli generali Mem. R. Acc. d'Italia, I
4 (1930), 531.

\bibitem{severi} F. Severi, Vorlesungen uber algebraische Geometrie. Johnson Pub. reprinted,
1968; 1st. ed., Leipzig 1921.

\bibitem{simpson}  C. Simpson, Subspaces of moduli spaces of rank one local systems. Ann. Sci. Ecole Norm. Sup. (4) 26 (1993), no. 3, 361-401.

\bibitem{simpsonstacks} C. Simpson,
Local systems on proper algebraic V-manifolds. 
Pure Appl. Math. Q. 7 (2011), no. 4, Special Issue: In memory of Eckart Viehweg, 1675-1759.


\bibitem{shimada97} I. Shimada, Fundamental groups of complements to singular plane curves. Amer. J. Math. 119 (1997), no. 1, 127-157. 

\bibitem{shimada98} I. Shimada, On the commutativity of fundamental groups of complements to plane curves.
Math. Proc. Cambridge Philos. Soc. 123 (1998), no. 1, 49-52.

\bibitem{shimada03} I. Shimada, 
Fundamental groups of algebraic fiber spaces. 
Comment. Math. Helv. 78 (2003), no. 2, 335-362.

\bibitem{shimada3}  I. Shimada, Equisingular families of plane curves with many connected components, Vietnam J. Math. 31 (2003), no. 2

\bibitem{shimada08} I. Shimada, 
On arithmetic Zariski pairs in degree 6.
Adv. Geom. 8 (2008), no. 2, 205-225. 

\bibitem{shimada10} I. Shimada,
Topology of curves on a surface and lattice-theoretic invariants of
coverings of the surface. Algebraic geometry in East Asia, Seoul 2008, 361-382,
Adv. Stud. Pure Math., 60, Math. Soc. Japan, Tokyo, 2010.

\bibitem{shimada101} I. Shimada, 
Lattice Zariski k-ples of plane sextic curves and Z-splitting curves for double plane sextics.
Michigan Math. J. 59 (2010), no. 3, 621-665.

\bibitem{shirane19} T. Shirane, 
Galois covers of graphs and embedded topology of plane curves. 
Topology Appl. 257 (2019), 122-143.

\bibitem{shirane} T. Shirane, 
Connected numbers and the embedded topology of plane curves.
Canad. Math. Bull. 61 (2018), no. 3, 650-658. 

\bibitem{shirane17} T. Shirane, 
A note on splitting numbers for Galois covers and $\pi_1$-equivalent Zariski k-plets.
Proc. Amer. Math. Soc. 145 (2017), no. 3, 1009-1017.

\bibitem{steenbrink}  J. Steenbrink, Mixed Hodge structure on the vanishing cohomology. Real and complex singularities (Proc. Ninth Nordic Summer School/NAVF Sympos. Math., Oslo, 1976), pp. 525-563. Sijthoff and Noordhoff, Alphen aan den Rijn, 1977. 


\bibitem{minageneric} M. Teicher, 
The fundamental group of a CP2 complement of a branch curve as an extension of a solvable group by a symmetric group.
Math. Ann. 314 (1999), no. 1, 19-38.

\bibitem{teicher} M. Teicher, 
Braid monodromy type invariants of surfaces and 4-manifolds. Trends in singularities, 215-222,
Trends Math., Birkh\"auser, Basel, 2002.

\bibitem{timmer} K. Timmerscheidt,
Mixed Hodge theory for unitary local systems.
J. Reine Angew. Math. 379 (1987), 152-171.

\bibitem{tokunaga00} H. Tokunaga,
Dihedral coverings of algebraic surfaces and their application. 
Trans. Amer. Math. Soc. 352 (2000), no. 9, 4007-4017.

\bibitem{tokunaga03} H. Tokunaga, 
Galois covers for S4 and U4 and their applications.
Osaka J. Math. 39 (2002), no. 3, 621-645.

\bibitem{urabe} T. Urabe,  Combinations of rational singularities on plane sextic curves with the sum of Milnor numbers less than sixteen. Singularities (Warsaw, 1985), 429-456, Banach Center Publ., 20, PWN, Warsaw, 1988. 

\bibitem{vankampen} E. Van Kampen,  On the Fundamental Group of an Algebraic Curve. Amer. J. Math. 55 (1933), no. 1-4, 255-260.

\bibitem{vaquie}   M. Vaqui\'e, Irregularit\'e des rev\^etements cycliques des surfaces projectives non singuli\'eres. Amer. J. Math. 114 (1992), 1187-1199

\bibitem{varchenko}  A. N. Varchenko, Asymptotic Hodge structure on vanishing cohomology. Izv. Akad. Nauk SSSR Ser. Mat. 45 (1981), no. 3, 540-591.

\bibitem{walter} U. Walther, Bernstein-Sato polynomial versus cohomology of the Milnor fiber for generic hyperplane arrangements. Compos. Math. 141 (2005), no. 1, 121-145.

\bibitem{yang} Jin-Gen Yang, Sextic curves with simple singularities. Tohoku Math. J. (2) 48 (1996), no. 2, 203-227. 

\bibitem{zariski29}  O. Zariski,  On the Problem of Existence of Algebraic Functions of Two Variables Possessing a Given Branch Curve. Amer. J. Math. 51 (1929), no. 2, 305-328.

\bibitem{zarhyperplane}  O. Zariski, A theorem on the Poincare group of an algebraic hypersurface. Ann. of Math. (2) 38 (1937), no. 1, 131-141.

\bibitem{zariskialgsurf} O. Zariski, Algebraic surfaces. With appendices by S. S. Abhyankar, J. Lipman and D. Mumford. Preface to the appendices by Mumford. Classics in Mathematics. Springer-Verlag, Berlin, 1995.


\bibitem{zariskipurity} O. Zariski, On the purity of the branch locus
  of algebraic functions.
 Proc. Nat. Acad. Sci. U.S.A. 44 (1958), 791-796. 

\bibitem{zarsam} O. Zariski, P.Samuel, Commutative
  algebra. Graduate Texts in Mathematics, Vol. 29. Springer-Verlag, New York-Heidelberg, 1975.

\end{thebibliography}
\end{document}